\documentclass[11pt,leqno]{amsart}

\usepackage[all,cmtip]{xy}
\usepackage{a4,latexsym}
\usepackage[utf8]{inputenc}
\usepackage{textcomp,fontenc,exscale,ifthen}
\usepackage[hidelinks]{hyperref}
\usepackage{amsxtra,amsmath,amsfonts,amscd, amssymb, mathrsfs, amsthm, bm}
\usepackage{color, mathtools}
\usepackage[shortlabels]{enumitem}

\usepackage{mathrsfs} 
\usepackage{tikz}
\usetikzlibrary{calc, arrows}
\usepackage{extarrows} 
\usepackage{stmaryrd}

\calclayout

\numberwithin{paragraph}{section}
\numberwithin{equation}{section}

\newcommand{\N}{\ensuremath{\mathbb{N}}}
\newcommand{\Z}{\ensuremath{\mathbb{Z}}}
\newcommand{\Q}{\ensuremath{\mathbb{Q}}}
\newcommand{\R}{\ensuremath{\mathbb{R}}}

\newcommand{\Nn}{\ensuremath{\mathbb{N}_{> 0}}}
\newcommand{\Spf}{\textnormal{Spf}}
\newcommand{\val}{\textnormal{val}}
\newcommand{\str}{\textnormal{str}}
\newcommand{\trop}{\textnormal{trop}}
\newcommand{\Spec}{\textnormal{Spec}}
\newcommand{\red}{\textnormal{red}}
\newcommand{\Nor}{\textnormal{Nor}}
\newcommand{\irr}{\textnormal{irr}}
\newcommand{\im}{\textnormal{im}}
\newcommand{\id}{\textnormal{id}}

\newcommand{\coker}{\textnormal{coker}}
\newcommand{\an}{\textnormal{an}}

\newcommand{\et}{\'{e}tale }
\newcommand{\etNS}{\'{e}tale} 

\DeclareTextFontCommand{\textbf}{\upshape\bfseries} 

\newcommand{\bracket}[2]{\left\{ #1 ~ \big| ~ #2 \right\}}
\newcommand{\dotsbra}[2]{\{#1, \dots, #2\}}
\newcommand{\dotsroundbra}[2]{(#1, \dots, #2)}

\newtheorem{theorem}{Theorem}[subsection] 
\newtheorem{mainthm}{Theorem}

\newtheorem{lemma}[theorem]{Lemma}
\newtheorem{proposition}[theorem]{Proposition}
\newtheorem{corollary}[theorem]{Corollary}
\newtheorem*{lemma*}{Lemma}

\theoremstyle{definition}
\newtheorem{definition}[theorem]{Definition}
\newtheorem{remark}[theorem]{Remark}
\newtheorem*{remark*}{Remark}
\newtheorem{example}[theorem]{Example}

\newtheorem{noname}[theorem]{}

\title{Extended skeletons of poly-stable pairs}

\author[T.~Fenzl]{Thomas Fenzl}
\address{T. Fenzl, Mathematik, Universit{\"a}t 
Regensburg, 93040 Regensburg, Germany}
\email{thomas-fenzl@web.de}

\thanks{The author was partially supported by the collaborative research 
center SFB 1085 \emph{Higher Invariants - Interactions between Arithmetic Geometry and Global Analysis} 
funded by the Deutsche Forschungsgemeinschaft.}

\begin{document}

\begin{abstract}
We introduce the notion of poly-stable pairs of formal schemes over the 
valuation ring of a non-archimedean field. 
For such pairs we define and investigate the dual intersection complex. 
We proceed to develop the so called extended skeleton of a poly-stable pair
via an approximation process using the
classical skeletons from \cite[§\,5]{Ber99}. This is essentially a generalization of \cite[§§\,3--4]{GRW16} 
from the strictly semi-stable case to the arbitrary poly-stable case
and we extend their results.
Our exposition follows closely the ideas from \cite{Ber99} and \cite{GRW16}. 
\end{abstract}


\maketitle

\tableofcontents

\section*{Introduction}

The magnificent results and fruitful applications of complex analysis made it desirable to establish an analogous theory
for non-archimedean fields. Today non-archimedean analytic geometry is a thriving branch of
mathematics at the heart of modern research. There are three well-known approaches to this topic.
First there is the language of \emph{rigid analytic spaces} initiated by John Tate
in the early 1960s. His main motivation was to formulate a uniformization theory 
for elliptic curves with bad reduction over non-archimedean fields.
Then there are the \emph{adic spaces} by Roland Huber, which gained recent attention since they provide
the framework for Peter Scholze's celebrated theory of \emph{perfectoid spaces}.
Finally there are the \emph{Berkovich analytic spaces}. One of the most striking advantages of these
is their particularly nice topology, which allows to access methods of real and tropical geometry.
Regarding this, the concept of the so called \emph{skeletons} plays an important role. 
In the present work these skeletons will be our primary focus and our goal is to explore and develop
a new and more general notion.

Skeletons already appear in Vladimir Berkovich's monumental work \cite{Ber90}. There in §\,4.3 he starts with a 
smooth geometrically connected projective curve $X$ of genus $g \geq 1$ over a non-archimedean field and constructs 
a closed subset $\Delta^{\an}(X)$ of the analytification $X^{\an}$, 
which is a strong deformation retract of $X^{\an}$. He calls it the \emph{analytic skeleton}. 
This allows him to show local contractibility
of analytic curves. In his effort to prove local contractibility of smooth analytic spaces of 
arbitrary dimension, he broadly generalizes and systematizes this concept in \cite{Ber99} and \cite{Ber04}.
There we are given a poly-stable formal scheme $\mathfrak{X}$
over the valuation ring $K^{\circ}$, where $K$ is a complete field with respect to a non-archimedean absolute value. 
Then by \cite[Theorem 5.2]{Ber99} there exists a canonical closed subset 
$S(\mathfrak{X})$ of the generic fiber $\mathfrak{X}_{\eta}$ and a proper strong deformation retraction
$\Phi: \mathfrak{X}_{\eta} \times [0, 1] \to \mathfrak{X}_{\eta}$ onto $S(\mathfrak{X})$. This skeleton $S(\mathfrak{X})$
has a piecewise linear structure as indicated by \cite[Theorem 5.2]{Ber99}. It states that the skeleton is canonically homeomorphic
to the geometric realization of a certain poly-simplicial set associated to $\mathfrak{X}$, which one can think of
as the dual intersection complex of the special fiber $\mathfrak{X}_s$. 
Roughly speaking it looks like a polyhedral complex whose faces are poly-simplices and 
correspond to certain \emph{strata} of $\mathfrak{X}_s$. The strata, which Berkovich employs in his paper,
form a disjoint locally finite cover of $\mathfrak{X}_s$ by locally closed subsets, which are obtained by iteratively
considering the irreducible components of the normality locus and taking the complement.
The astounding consequence is, that the homotopy type of the generic fiber of $\mathfrak{X}$ is determined by its special fiber.
Additionally, Berkovich expands his construction to a broader class of formal schemes called \emph{pluri-stable} via the introduction 
of \emph{poly-stable fibrations}.
It has not been known for quite some time whether every log smooth scheme over $K^{\circ}$
admits a poly-stable modification. This question was answered affirmatively in \cite{ALPT19},
which further emphasizes the relevance of the poly-stable case in the present work. 
In fact, many aspects of our exposition can be formulated and understood from the viewpoint
of \emph{log geometry}, however we decided to stick with the classical language from \cite{Ber99} and \cite{Ber04}.

After Berkovich's pioneering treatise on the skeleton,
a multitude of similar notions has been created by different authors 
and many applications, not only for non-archimedean geometry, have been found.
The list of contributions to this topic is long, therefore we will limit ourselves to a brief survey
of the developments most relevant for our work.

One case that has been extensively explored is that of curves. 
For instance let $X$ be a smooth projective curve over $K$ and suppose we have a
formal scheme $\mathfrak{X}$ with $\mathfrak{X}_{\eta} = X^{\an}$. 
The associated skeleton $S(\mathfrak{X})$ is then a metric graph whose underlying graph
is the incidence graph of $\mathfrak{X}_s$. The segments appearing in this skeleton all have finite lengths. 
With the tropicalization of $X$ in mind, which also features unbounded faces, Ilya Tyomkin associates in
\cite{Tyo12} in a canonical way a tropical curve to a pair $(X, H)$, where $H$ consists of some marked points
on $X$. This tropical curve can be interpreted as the classical skeleton $S(\mathfrak{X})$ extended by some infinite
rays in the direction of each marked point of $H$. In his setting he additionally assumes that 
$K$ is discretely valued and the residue field $\widetilde{K}$ is algebraically closed. 
The formal scheme $\mathfrak{X}$ is here the formal completion of the stable model for $(X, H)$.

A similar construction was performed by Matthew Baker, Sam Payne and Joseph Rabinoff in \cite{BPR12} resp. \cite{BPR13}.
In their situation, $K$ is algebraically closed and non-trivially valued and $X$ is a smooth connected curve over $K$.
They consider the smooth completion $\widehat{X}$ of $X$ and the set $D := \widehat{X} \setminus X$ of punctures
takes the role of the set of marked points in Tyomkins work. The analytification $X^{\an}$ 
decomposes into a semi-stable vertex set $V$ and
a disjoint union of open balls and finitely many generalized open annuli. They use this to define the skeleton $\Sigma(X, V)$
as the union of $V$ and the skeletons of each generalized open annuli in the semi-stable decomposition. 
The set $D$ is in bijective correspondence with the set of punctured open balls in the above decomposition
and each associated skeleton of these is an infinite ray, in accordance with Tyomkins ideas. 

This concept was taken a step further in the paper \cite{GRW16} by Walter Gubler, Joseph Rabinoff and Annette Werner. 
They introduce the notion of a \emph{formal strictly semistable pair} $(\mathfrak{X}, H)$ consisting 
of a connected quasicompact admissible formal $K^{\circ}$-scheme $X$ and a Cartier divisor $H$ on $\mathfrak{X}$
of a certain form. The field $K$ is here algebraically closed and non-trivially valued.
Locally $\mathfrak{X}$ admits an \et morphism to a standard scheme $\Spf(K^{\circ}\{T_0,\dots,T_d\}/(T_0 \cdots T_r - \pi))$ 
with $\pi \in K^{\circ \circ} \setminus \{0\}$
and the divisor $H$ can be written as a sum of effective divisors, such that each summand has irreducible support
on the generic fiber and locally is trivial or the pullback of a coordinate function $T_j$ with $j > r$. 
Motivated by the explicit form of the strata in the strictly poly-stable case, see \cite[Proposition 2.5]{Ber99},
they establish a notion of \emph{vertical} and \emph{horizontal strata} of the pair $(\mathfrak{X}, H)$.
The authors continue to define the skeleton of a standard pair via an approximation method applied
to the classical skeletons. To get a rough idea, one can imagine that the punctured unit balls,
which appear when taking the complement of the divisor $H$, 
get exhausted by annuli of outer radius $1$ and decreasing inner radius. The classical skeleton
of these annuli are segments whose length increases as the inner radius tends towards $0$.
In the limit this produces an infinite ray.
After that the skeleton is introduced for a so called \emph{building block} of a stratum. 
The building block skeletons form the “faces” from which later the skeleton of a strictly semi-stable pair is put together.
The authors call these faces \emph{canonical polyhedra}, because they are canonically homeomorphic 
to a polyhedron via the tropicalization map up to a permutation of the coordinates.

This leads us to the content of the paper. The main task was to generalize the construction 
in \cite[§\,4]{GRW16} from strictly semi-stable pairs to arbitrary poly-stable pairs. 
As a first challenge, one had to come up with a suitable notion of poly-stability
for pairs $(\mathfrak{X}, \mathfrak{H})$, where $\mathfrak{H}$ should be a closed formal subscheme
of $\mathfrak{X}$ of a certain form. 
We achieve this by a straightforward generalization of standard pairs, 
where we take a product of the usual standard pairs from above and also allow degenerate factors with $\pi=0$.
One can now define a pair as \emph{strictly poly-stable}, if locally it is an \et pullback 
of a standard pair. Then \emph{poly-stable} pairs should be the ones,
which admit a surjective \et morphism from a strictly poly-stable pair to them.
Moreover we make only the same assumptions about our ground field $K$ as Berkovich in \cite{Ber99},
in particular $K$ does not have to be algebraically closed and its absolute value may be trivial.
Our strategy is to deal with the non-trivially valued case first
and later extend the results to the trivially valued case.
In the non-trivially valued case, we have the approximation procedure from \cite[§\,4]{GRW16} 
at our disposal. Then using techniques from \cite[§\,5]{Ber99}
we can associate an \emph{extended skeleton} $S(\mathfrak{X}, \mathfrak{H})$ to each poly-stable pair $(\mathfrak{X}, \mathfrak{H})$.
Instead of “extended skeleton” we will most of the time just say “skeleton” for brevity.
However we do not only want to understand the skeleton as a subset of $Z := \mathfrak{X}_{\eta} \setminus \mathfrak{H}_{\eta}$,
but also how we can describe its piecewise linear structure.  
After all it is still made up from building blocks as before, which provide canonical polyhedra.
For the poly-stable case it is not so clear what “canonical up to a permutation of the coordinates” means, though.
In order to give a precise formulation, we devote the second section to developing a notion 
of the dual intersection complex $C(\mathfrak{X}, \mathfrak{H})$ and its faces, where we follow the ideas of \cite[§§\,3--4]{Ber99}.
Roughly speaking, $C(\mathfrak{X}, \mathfrak{H})$ is made up of the faces $\Delta(x)$, which are extended poly-simplices
associated to each generic point $x$ of a stratum of $\mathfrak{X}_s$ or $\mathfrak{H}_s$.
The way these faces intersect each other is reflected by how the closures of the strata contain one another.
An important role is also played by the open faces $\Delta^{\circ}(x)$, 
which can be seen as the “interiors” of the faces $\Delta(x)$  
and which form a disjoint cover of $C(\mathfrak{X}, \mathfrak{H})$.  
We can then formulate our main theorem, see Theorem~\ref{MainThmA}, as follows:
\begin{mainthm} \label{MainA}
	Let $(\mathfrak{X}, \mathfrak{H})$ be a poly-stable pair and $Z := \mathfrak{X}_{\eta} \setminus \mathfrak{H}_{\eta}$.
	The extended skeleton $S(\mathfrak{X}, \mathfrak{H})$ is a subset of $Z$ satisfying the following properties:
	\begin{enumerate}
		\item $S(\mathfrak{X}, \mathfrak{H})$ is a closed subset of $Z$.
		\item There is a canonical proper strong deformation retraction $\Phi: Z \times [0, 1] \to Z$ onto $S(\mathfrak{X}, \mathfrak{H})$.
		\item $S(\mathfrak{X}, \mathfrak{H})$ is canonically homeomorphic to the
			dual intersection complex $C(\mathfrak{X}, \mathfrak{H})$.
		\item For every generic point $x \in \str(\mathfrak{X}_s, \mathfrak{H}_s)$ of a stratum this canonical homeomorphism
			restricts to a homeomorphism $S(\mathfrak{X}, \mathfrak{H}) \cap \red^{-1}_{\mathfrak{X}}(x) \xrightarrow{\sim} \Delta^{\circ}(x)$.  
	\end{enumerate} 
\end{mainthm}
\noindent Assertion (iv) together with Lemma \ref{FaceFacts} (ii) tells us, 
that the strata of $(\mathfrak{X}_s, \mathfrak{H}_s)$, which give a disjoint decomposition of $\mathfrak{X}_s$,
are in bijective correspondence with the open faces of the dual intersection complex,
which in turn can canonically be associated with the formal fibers of $\mathfrak{X}$ intersected with the skeleton 
and make up a disjoint decomposition of $S(\mathfrak{X}, \mathfrak{H})$.
Moreover the codimension of the stratum is equal to the dimension of the associated face.
This very satisfying result is occasionally referred to as the \emph{stratum-face-correspondence}
and fortunately it is also true in the context of our notions.
	
Furthermore we consider the \emph{closure}
$\overline{S}(\mathfrak{X}, \mathfrak{H}) := S(\mathfrak{X}, \mathfrak{H}) \cup S(\mathfrak{H})$
and show in analogy to \cite[Theorem 4.13]{GRW16} the following result:
\begin{mainthm} \label{MainB}
		The closure $\overline{S}(\mathfrak{X}, \mathfrak{H})$ is equal to
		the topological closure of $S(\mathfrak{X}, \mathfrak{H})$ in $\mathfrak{X}_{\eta}$.
		There exists a canonical proper strong deformation retraction 
		$\Phi: \mathfrak{X}_{\eta} \times [0, 1] \to \mathfrak{X}_{\eta}$ 
		onto the closure $\overline{S}(\mathfrak{X}, \mathfrak{H})$. 
\end{mainthm}
\noindent This can be found in Theorem \ref{MainThmB} and is a nice supplement to the above, since the 
closure of the extended skeleton is in a natural way a deformation retract of $\mathfrak{X}_{\eta}$ instead of $Z$.
Each of our constructions and results can be extended to trivially valued fields $K$ 
and the technique for doing so is inspired by \cite{BJ18}.
		
Next we want to give a brief outline of the structure of the paper. Section $1$ is dedicated to some preparatory
work for later use. In particular Subsections $1.1$ and $1.2$ recall the fundamental aspects 
of formal geometry and Berkovich analytic geometry and introduce all necessary notations.
In Subsection $1.3$ we give the definition of standard schemes, standard pairs, strictly poly-stable
pairs and poly-stable pairs. We also give examples and state some easy first observations.
The next Subsection $1.4$ deals with the notion of strata of pairs. 
We extend the definition by Berkovich such that basic results from \cite[§\,2]{Ber99} still hold in our situation
and explain the partial order of the strata.
Also an explicit description of the strata of $(\mathfrak{X}_s, \mathfrak{H}_s)$ for a 
strictly poly-stable pair $(\mathfrak{X}, \mathfrak{H})$ is given.  
The final Subsection $1.5$ of the first section is about building blocks.
In the strictly poly-stable case, we can find open affine neighborhoods around every point, which
admit an \et morphism to a standard pair such that the special fibers share the same intersection
behavior of their respective irreducible components. This situation is particularly easy to handle
and it provides the pieces of which later the dual intersection complex and the skeleton 
are put together, as the name “building block” suggests. We also shed some light on the structure
of $\mathfrak{H}_s$, which will be important for later application.
		
In Section $2$ we construct the dual intersection complex of a poly-stable pair.
We begin with Subsection $2.1$, where the language of extended poly-simplices 
and their geometric realizations is introduced in analogy to \cite[§§\,3--4]{Ber99}. 
Note that throughout the paper we will use the additive notion to describe simplices
as opposed to the multiplicative notion employed by Berkovich. 
These are the tools which allow us in
Subsection $2.2$ to introduce canonical polyhedra to each stratum of $(\mathfrak{X}_s, \mathfrak{H}_s)$ for a 
strictly poly-stable pair $(\mathfrak{X}, \mathfrak{H})$. 
The idea is, that whenever we consider a building block in $x \in \str(\mathfrak{X}_s, \mathfrak{H}_s)$,
which is the set of generic points of the strata,
then the associated standard pair is determined by $x$ 
up to a “change of coordinates”, which is made precise by the term
“isomorphism of extended poly-simplices”. In order to get rid of the choice
of coordinates, we identify all the geometric realizations of these possibilities 
and get a geometric object which we call the canonical polyhedron.
As it turns out, the faces of the canonical polyhedron associated to $x$ can be interpreted
as the canonical polyhedra of $y$ for all $y \in \str(\mathfrak{X}_s, \mathfrak{H}_s)$
with $x \in \overline{\{y\}}$. They are included via so called face embeddings.
The dual intersection complex of a strictly poly-stable pair,
which is the content of Subsection $2.3$, is then obtained by gluing together
all canonical polyhedra along the face embeddings.
We investigate the structure of these and see that they are close to being
polyhedral complexes. However the intersection of faces might consist of several faces.
With a suitable subdivision along the lines of \cite[§\,5]{Gub10} it is possible to obtain a
polyhedral complex in the conventional sense. 
In Subsection $2.4$ we introduce the dual intersection complex of an arbitrary poly-stable pair.
It is obtained by identifying faces of the dual intersection complex of a suitable strictly poly-stable pair 
according to how the strata are mapped under the chosen \et morphism. The construction is
independent of this choice up to a canonical homeomorphism.
		
Section $3$ establishes our notion of skeleton as a subset of an analytic space.
In Subsection $3.1$ we give a brief review of the classical construction by Berkovich in \cite[§\,5]{Ber99}.
After that we assume that $K$ is non-trivially valued and Subsection $3.2$ takes care of the skeletons of standard pairs, which are introduced 
with an approximation procedure by classical skeletons in the same manner as in \cite[4.2]{GRW16}.
We see that the tropicalization map induces a homeomorphism between the skeleton and 
a geometric extended poly-simplex.
Then the goal of Subsection $3.3$ is to deal with building blocks.
Our technique is comparable with Step $6$ from the Proof of Theorems 5.2–5.4 in \cite[§\,5]{Ber99}. 
The most notable result here is that the choice of the building block determines the homeomorphism
between the skeleton and the geometric extended poly-simplex up to a “change of coordinates”
in the sense of extended poly-simplices, see Proposition \ref{CoordChange}.
We then have all the tools to define in Subsection $3.4$ the skeleton of a strictly poly-stable pair
and describe its relation with the dual intersection complex.
In Subsection $3.5$ we carry out the construction for the most general case, namely
that of an arbitrary poly-stable pair. We show all the desired properties of the skeleton
claimed in Theorem \ref{MainA}. One important observation is, that a surjective \et morphism
of formal schemes induces a quotient map between the generic fibers, which 
allows to consider the skeleton of a poly-stable pair as a certain coequalizer 
with respect to the skeleton of the chosen strictly poly-stable pair.
Finally we deal with trivially valued fields, see Subsection $3.6$, essentially 
by performing a base change to a non-trivially valued field and extend 
our results from the previous sections.
	
As an addendum, I included in Section $4$ some ideas towards the closure
of the extended skeleton and prove Theorem \ref{MainB}. Here we make use of the skeletons
of degenerate standard pairs.

\addtocontents{toc}{\protect\setcounter{tocdepth}{0}}
\section*{Acknowledgements}
\addtocontents{toc}{\protect\setcounter{tocdepth}{1}}
I would like to dearly thank my advisor Walter Gubler for introducing me
to non-archimedean analytic geometry and his excellent support during 
the writing of my thesis, which is the foundation of this paper. 	

I also want to give my thanks to Klaus Künnemann and Christian Vilsmeier
for helpful discussions and
I am grateful to Vladimir Berkovich for answering a question that arose during the course of this work.
	
I wish to emphasize my appreciation for the support which I
received as an associate member of the collaborative research center “SFB 1085: Higher Invariants”
funded by the Deutsche Forschungsgemeinschaft. It bore my travel expenses regarding scientific conferences 
and temporarily provided me with direct funding.


\section{Preliminaries}

This first section is devoted to introducing notations and
to explore the crucial definitions which provide the foundation for the objects studied
in this paper.  
We start off by fixing the following conventions throughout the paper.

\begin{noname}
	Here the natural numbers $\N$ contain $0$. For two sets $A$, $B$ the complement of 
	$A$ in $B$ is denoted by $B \setminus A$. 
	We write $A \subseteq B$, if and only if $A$ is a subset of $B$, and this also includes the case $A = B$.
	We denote the cardinality of $A$ by $|A|$.
	All rings are considered to be commutative and with unity unless specified otherwise. 
	For the group of units of a ring $R$ we write $R^*$.
	For an $n$-tuple $\textbf{r} = (r_1,\dots,r_n)$ of real numbers we denote $|\textbf{r}| := r_1 + \cdots + r_n$.
	Instead of a 1-tuple $(x)$ we will occasionally just write the entry $x$.
\end{noname}

\begin{noname}
	We follow the convention for “compactness” established by Bourbaki, which means the following:
	Let $X$ be a topological space.
	We will call $X$ \emph{quasicompact}, if every open cover of $X$ admits a finite subcover.
	If additionally the space $X$ is Hausdorff, we call it \emph{compact}.
	
	The space $X$ is said to be \emph{locally compact}, if $X$ is Hausdorff and
	every point of $X$ has a compact neighborhood.
	
	The space $X$ is said to be \emph{paracompact}, if $X$ is Hausdorff 
	and every open cover of $X$ admits a locally finite subcover.
	
	A formal scheme or a usual scheme is called \emph{quasicompact}, if its underlying topological space is quasicompact.
\end{noname}

\begin{noname}
	Let $X$ be a topological space and $A \subseteq X$ a subspace. We call a continuous map
	$\tau: X \to A$ a \emph{retraction}, if the restriction of $\tau$ to $A$ is the identity map on $A$.
	Moreover a continuous map $\Phi: X \times [0, 1] \to X$ is called a \emph{deformation retraction}
	onto $A$, if $\Phi(\cdot, 0)$ is the identity map on $X$ and $\Phi(\cdot, 1)$ is a retraction onto $A$.
	If additionally $\Phi(a, t) = a$ for all $a \in A$ and $t \in [0, 1]$, then $\Phi$
	is called a \emph{strong deformation retraction}.
	
	A continuous map $f: X \to Y$ between topological spaces is called \emph{proper}, 
	if the preimage of every quasicompact subset of $Y$ under $f$ is quasicompact.
	In the case that $X$ is Hausdorff and $Y$ is locally compact, this definition is equivalent to 
	the one in \cite[Chapter 1, §\,10.1, Definition 1]{Bou95}, see \emph{loc.cit.}, §\,10.3, Proposition 7.
	In particular in this case proper maps are closed, see \emph{loc.cit.}, §\,10.1, Proposition 1, and	 
	the product of two proper maps is again proper, see \emph{loc.cit.}, §\,10.2, Corollary 3.	
\end{noname}

\begin{noname}
	We will work with the compactified non-negative real numbers $\overline{\R}_{\geq 0} := \R_{\geq 0} \cup \{\infty\}$ 
	equipped with the obvious topology. We employ the usual conventions for calculations,
	namely $0 \cdot \infty = 0$, $a \cdot \infty = \infty$ for all $a \in \overline{\R}_{\geq 0} \setminus \{0\}$ 
	and a sum of finitely many numbers in $\overline{\R}_{\geq 0}$ 
	is equal to $\infty$ if and only if at least one of the summands is $\infty$.
	Otherwise it is the usual sum. Also we define $-\log(0) := \infty$. 
	Then $-\log$ induces an isomorphism $[0, 1] \xrightarrow{\sim} \overline{\R}_{\geq 0}$ of topological monoids,
	i.\,e. a homeomorphism turning multiplication in $[0, 1]$ into addition in $\overline{\R}_{\geq 0}$.
	Its inverse is given by $\overline{\R}_{\geq 0} \xrightarrow{\sim} [0, 1]$, $x \mapsto \exp(-x)$,
	where we define $\exp(-\infty) := 0$. 
\end{noname}



\subsection{Analytic geometry}

Here we give a brief overview of the notions from non-ar\-chi\-me\-dean analytic geometry
in the sense of Berkovich.

\begin{noname}
	Throughout the whole paper let $K$ be a field which is complete with respect to a non-archimedean
	absolute value $|\phantom{x}|$. The trivial absolute value is also allowed. 
	We denote by $K^{\circ} := \bracket{x \in K}{|x| \leq 1}$ its 
	valuation ring and by $K^{\circ \circ} := \bracket{x \in K}{|x| < 1}$ the maximal ideal.
	The residue field is $\widetilde{K} := K^{\circ}/K^{\circ \circ}$ and the valuation map $\val: K^* \rightarrow \R$ 
	is given by $\val := -\log|\phantom{x}|$. We extend it by setting $\val(0) := \infty$. 
	We write $|K^*|$ resp. $\Gamma := \val(K^*)$ for the multiplicative resp. additive value group.
\end{noname}

\begin{noname}
	The \emph{Berkovich spectrum} of a Banach ring $\mathscr{A}$ is denoted by $\mathscr{M}(\mathscr{A})$.
	It consists of all multiplicative bounded seminorms on $\mathscr{A}$ in the sense of \cite[§\,1.1]{Ber90}.
	
	For every $x \in \mathscr{M}(\mathscr{A})$
	we will denote the evaluation of $x$ in an element $f \in \mathscr{A}$ by $|f(x)|$.
	Recall that $\mathscr{M}(\mathscr{A})$ carries the initial topology with respect to the
	evaluation functions $\mathscr{M}(\mathscr{A}) \to \R_{\geq 0}$, $x \mapsto |f(x)|$
	for all $f \in \mathscr{A}$.
	The well-known result \cite[Theorem 1.2.1]{Ber90} states that $\mathscr{M}(\mathscr{A})$
	is a non-empty compact space.
	
	For every $x \in \mathscr{M}(\mathscr{A})$ the kernel $\wp_x$ of $\mathscr{A} \to \R_{\geq 0}$, $f \mapsto |f(x)|$
	is a closed prime ideal of $\mathscr{A}$. The resulting valuation on the integral domain $\mathscr{A}/\wp_x$  
	extends to a valuation on the fraction field. The completion of this fraction field with respect to
	the valuation is denoted by $\mathscr{H}(x)$. The residue field of the valued field $\mathscr{H}(x)$
	is denoted by $\widetilde{\mathscr{H}}(x)$. 
\end{noname}

\begin{noname}
	The \emph{Tate algebra} in $n$ variables $T_1,\dots,T_n$ with coefficients in $K$ resp. $K^{\circ}$ is denoted by
	$K\{T_1,\dots,T_n\}$ resp. $K^{\circ}\{T_1,\dots,T_n\}$.  We denote by $B^n := \mathscr{M}(K\{T_1,\dots,T_n\})$ 
	the $n$-dimensional Berkovich analytic ball and by $\mathfrak{B}^n := \Spf(K^{\circ}\{T_1,\dots,T_n\})$
	the $n$-dimensional formal ball. In the 1-dimensional case we will often write $B$ with variable $T$ instead of $B^1$
	with variable $T_1$. Furthermore we use the notations $B(a, r) := \bracket{x \in B}{|(T-a)(x)| \leq r}$ resp.
	$B^{\circ}(a, r) := \bracket{x \in B}{|(T-a)(x)| < r}$ where $a \in K^{\circ}$ and $r \in [0, 1]$ resp. $r \in (0, 1]$, 
	which are known as the \emph{closed} resp. \emph{open discs}.
\end{noname} 

\begin{noname}
	When we talk about an \emph{analytic space}, we mean a $K$-analytic space in the sense of \cite[§\,1.2]{Ber93}.
	By definition, every $K$-analytic space is locally Hausdorff and each of its points has a neighborhood, which is a finite
	union of \emph{$K$-affinoid spaces}, which roughly speaking are Berkovich spectra of quotients of generalized Tate algebras.
	Consequently Hausdorff analytic spaces are locally compact. In the case, that one only needs quotients of classical
	Tate algebras, as introduced above, the $K$-analytic space is called \emph{strictly $K$-analytic}. 
	
	Berkovich analytic spaces are closely related with the older concept of rigid $K$-analytic spaces from
	\cite[§\,9]{BGR84}. More precisely, there is an equivalence between the category of paracompact strictly $K$-analytic spaces
	and the category of quasi-separated rigid $K$-analytic spaces which have an admissible affinoid covering of finite type, 
	see \cite[Theorem 1.6.1]{Ber93}. This equivalence allows us to apply results and constructions from the classical 
	rigid analytic setting to our Berkovich analytic setting.
\end{noname}

\begin{noname}
	A Banach $K$-algebra $\mathscr{A}$ is called \emph{peaked} over $K$, if for every valuation field $F$ over $K$
	the norm on the base change $\mathscr{A} \widehat{\otimes}_K F$ is multiplicative.
	
	A point $x$ of a $K$-analytic space is called \emph{peaked} over $K$, if $\mathscr{H}(x)$ is peaked over $K$.
	
	This notion of “peaked points” is synonymous with the notion of “universal points” introduced in \cite{Poi13},
	which is also used in the modern literature. For the present work however, we will stick to 
	Berkovich's classical notion of “peaked points” from \cite[§\,5.2]{Ber90}. 
\end{noname}




\subsection{Formal geometry}

This subsection explains the formal schemes which are relevant for our work.
In some sense they form a bridge between analytic geometry and algebraic geometry.
The understanding of generic and special fibers 
as well as the reduction map will be of utmost importance for us.

\begin{noname}
	Recall that a family $(X_i)_{i \in I}$ of subspaces of a topological space $X$ is called
	\emph{locally finite}, if every point of $X$ has an open neighborhood
	which intersects only finitely many of the $X_i$.
	
	A ring $A$ is called \emph{topologically finitely presented} over $K^{\circ}$,
	if it is a  $K^{\circ}$-algebra of the form
	$K^{\circ}\{T_1,\dots,T_n\}/\mathfrak{a}$ for some finitely generated ideal 
	$\mathfrak{a} \subseteq K^{\circ}\{T_1,\dots,T_n\}$. 
	If additionally $A$ is flat over $K^{\circ}$, then $A$ is called an 
	\emph{admissible} $K^{\circ}$-algebra. 
	
	We call a formal scheme $\mathfrak{X}$ \emph{locally finitely presented} over $K^{\circ}$,
	if it is a locally finite union of open affine subschemes of the form
	$\Spf(A)$, where $A$ is a ring which is topologically finitely presented over $K^{\circ}$.
	If additionally every $A$ is an admissible $K^{\circ}$-algebra, then $\mathfrak{X}$
	is called \emph{admissible}. 
\end{noname}

\begin{noname}
	Let $\mathfrak{X}$ be an admissible formal $K^{\circ}$-scheme. 
	We denote by $\mathfrak{X}_s$ its \emph{special fiber}, which is a scheme of locally
	finite type over $\widetilde{K}$. In the affine case $\mathfrak{X} = \Spf(A)$, the special fiber is given as 
	$\mathfrak{X}_s = \Spec(\widetilde{A})$, where $\widetilde{A} := A/K^{\circ\circ}A$. 
	Moreover we denote by $\mathfrak{X}_{\eta}$ its \emph{generic fiber}, which is a paracompact locally compact strictly $K$-analytic space. 
	In the affine case $\mathfrak{X}_{\eta} = \mathscr{M}(\mathscr{A})$, where $\mathscr{A} := A \otimes_{K^{\circ}} K$
	is a strictly $K$-affinoid algebra. The general case is obtained by gluing together the generic fibers 
	from the affine situation. 
	We also want to mention, that if $(\mathfrak{X}_i)_{i \in I}$ is a locally finite 
	open cover of $\mathfrak{X}$ by admissible affines, then $(\mathfrak{X}_{i, \eta})_{i \in I}$ is a locally finite 
	cover of $\mathfrak{X}_{\eta}$ by closed analytic domains. 
	
	We also have the well-known \emph{reduction map} $\red_{\mathfrak{X}}: \mathfrak{X}_{\eta} \rightarrow \mathfrak{X}_s$.
	It is anti-continuous, which means that the preimage of every open subset is closed.
	We can explicitly describe it in the affine case: Every point $x \in \mathfrak{X}_{\eta} = \mathscr{M}(\mathscr{A})$
	gives rise to a character $\widetilde{\chi}_x: \widetilde{A} \to \widetilde{\mathscr{H}}(x)$.
	Then the kernel of $\widetilde{\chi}_x$ is a prime ideal of $\widetilde{A}$ and we have
	$\red_{\mathfrak{X}}(x) = \ker(\widetilde{\chi}_x) \in \Spec(\widetilde{A}) = \mathfrak{X}_s$.
	
	For a morphism $\psi: \mathfrak{Y} \to \mathfrak{X}$ of admissible formal $K^{\circ}$-schemes,
	we denote by $\psi_s$ resp. $\psi_{\eta}$ the induced morphism of special fibers
	resp. generic fibers. The reduction map is functorial in the sense that the following 
	diagram commutes:
	\begin{center}\begin{tikzpicture}[scale = 2]
		\node (1) at (0, 1) {$\mathfrak{Y}_{\eta}$};
		\node (2) at (1, 1) {$\mathfrak{X}_{\eta}$};
		\node (3) at (0, 0) {$\mathfrak{Y}_s$};
		\node (4) at (1, 0) {$\mathfrak{X}_s$};		
		
		\draw[->, thick] (1)--(2) node[pos=0.5, above] {$\psi_{\eta}$};
		\draw[->, thick] (1)--(3) node[pos=0.5, left] {$\red_{\mathfrak{Y}}$};
		\draw[->, thick] (2)--(4) node[pos=0.5, right] {$\red_{\mathfrak{X}}$};
		\draw[->, thick] (3)--(4) node[pos=0.5, below] {$\psi_s$};
		\end{tikzpicture}\end{center}
	Details concerning these constructions can be found in \cite[§\,1]{Ber94}.
\end{noname}



\begin{noname}
	Let $\psi: \mathfrak{Y} \to \mathfrak{X}$ be a morphism of formal $K^{\circ}$-schemes.
	We say that $\psi$ is \emph{\etNS}, if for every ideal of definition $\mathcal{J} \subseteq \mathcal{O}_{\mathfrak{X}}$ 
	the induced morphism $(\mathfrak{Y}, \mathcal{O}_{\mathfrak{Y}}/\mathcal{J}\mathcal{O}_{\mathfrak{Y}}) 
	\to (\mathfrak{X}, \mathcal{O}_{\mathfrak{X}}/\mathcal{J})$ of schemes is \etNS.
	
	More explicitly, if the absolute value on $K$ is non-trivial, we fix a non-zero element $a \in K^{\circ\circ}$.
	In the trivially valued case, we set $a := 0$. We denote for every $n \in \N_{>0}$ by $\mathfrak{X}_n$
	the scheme $(\mathfrak{X}, \mathcal{O}_{\mathfrak{X}}/a^n\mathcal{O}_{\mathfrak{X}})$,
	which is locally finitely presented over $K^{\circ}/(a^n)$. Analogously define $\mathfrak{Y}_n$.  
	Then $\psi: \mathfrak{Y} \to \mathfrak{X}$
	is \etNS, if and only if for all $n \in \N_{>0}$ the induced morphisms $\psi_n: \mathfrak{Y}_n \to \mathfrak{X}_n$ 
	of schemes are \etNS.
	
	Note that in this case in particular the induced morphism $\psi_s: \mathfrak{Y}_s \to \mathfrak{X}_s$ 
	of the special fibers is \etNS.	
\end{noname}

\begin{noname}
	We usually will denote formal $K^{\circ}$-schemes by Fraktur letters $\mathfrak{X}$, $\mathfrak{Y}$, and so on.
	Algebraic $\widetilde{K}$-schemes will be denoted by calligraphic letters $\mathcal{X}$, $\mathcal{Y}$, and so on. 
	Analytic spaces are denoted by capital Roman letters $X$, $Y$, and so on.
\end{noname}


\subsection{Standard schemes and poly-stability}

Next we introduce the basic notions of standard schemes, standard pairs and poly-stable pairs.
In terms of the special fiber, the standard schemes are made up of factors which
define the concept of having “simple normal crossings”
and exhibit in some sense the tamest kind of singularity a scheme can have.
We then obtain a pair by including a divisor cut out by the coordinate functions 
of a ball, which can also be a factor of a standard scheme.
Requiring this shape \et locally gives rise to the notion of “strictly poly-stable”,
which can then be further generalized to “poly-stable”.
The idea for considering pairs and the way how to define them originates from \cite[§\,3]{GRW16}.

\begin{definition}
	For $n \in \N$ and $a \in K^{\circ}$ we define the formal $K^{\circ}$-scheme 
	$\mathfrak{T}(n, a) := \Spf(K^{\circ}\{T_0,\dots,T_n\}/(T_0 \cdots T_n - a))$.
	We will refer to the formal scheme $\mathfrak{T}(n, 1)$ as \emph{torus}.
	
	Let $p \in \Nn$, $\textbf{n} = (n_1,\dots,n_p)$ and $\textbf{a} = (a_1,\dots,a_p)$ with $n_i \in \Nn$ and
	$a_i \in K^{\circ}$ for all $i \in \dotsbra{1}{p}$. 
	Let $d \in \N$. We define the following fiber product 
	over $\Spf(K^{\circ})$ in the category of formal $K^{\circ}$-schemes: 
	\[\mathfrak{S}(\textbf{n}, \textbf{a}, d) := \prod\limits_{i=1}^p \mathfrak{T}(n_i, a_i)
	\times \mathfrak{B}^d. \] 
	We call $\mathfrak{S}(\textbf{n},\textbf{a}, d)$ a \emph{standard scheme}, if $a_i \in K^{\circ \circ}$ for all $i \in \dotsbra{1}{p}$. 
	\noindent Usually we will refer to the coordinates of a standard scheme by
	\[\mathfrak{S}(\textbf{n}, \textbf{a}, d) = \prod\limits_{i=1}^p \Spf(K^{\circ}\{T_{i0},\dots,T_{in_i}\}/(T_{i0} \cdots T_{in_i} - a_i)) 
	\times \Spf(K^{\circ}\{T_1,\dots,T_d\}). \]  
	This is an affine formal scheme and we get the associated ring by taking the 
	completed tensor products of the rings corresponding to each factor.
	We also include the case $p = 0$, $\textbf{n} = (0)$ and $\textbf{a} = (1)$ to the standard schemes, i.\,e.
	\[\mathfrak{S}(d) := \mathfrak{S}((0), (1), d) := \mathfrak{B}^d. \]
	Furthermore we denote $\mathfrak{S}(\textbf{n}, \textbf{a}) := \mathfrak{S}(\textbf{n}, \textbf{a}, 0)$.
	Let $s \in \N$ with $s \leq d$.
	Define $\mathfrak{G}(s)$ to be the closed subscheme of $\mathfrak{S}(\textbf{n}, \textbf{a}, d)$ 
	associated to the ideal $(T_1 \cdots T_s)$, where $T_1, \dots, T_d$ denote the coordinates 
	of the factor $\mathfrak{B}^d$ as introduced above.
	In the case $s = 0$ we have $\mathfrak{G}(0) = \emptyset$. 
	One can think of $\mathfrak{G}(s)$ as the closed subset of $\mathfrak{S}(\textbf{n}, \textbf{a}, d)$ 
	cut out by the equation $T_1 \cdots T_s = 0$. 
	We call $(\mathfrak{S}(\textbf{n}, \textbf{a}, d), \mathfrak{G}(s))$ a \emph{standard pair}.
\end{definition}

\begin{noname} \label{ToriCover}
	We can find an open cover of the formal ball $\mathfrak{B}^n$ by tori $\mathfrak{T}(n, 1)$:
	Consider the morphisms $\mathfrak{T}(n, 1) \to \mathfrak{B}^n$ induced by sending 
	each coordinate $T_i$ either to $T_i$ or to $1-T_i$. We obtain $2^n$ different open immersions in this way
	and their images cover $\mathfrak{B}^d$. 
\end{noname}

\begin{noname} \label{RemoveCoord}
	Let $n \in \N$, $k \in \dotsbra{0}{n}$ and $a \in K^{\circ}$. We want to consider the open formal subscheme $\mathfrak{U}$
	of $\mathfrak{T}(n, a) = \Spf(K^{\circ}\{T_0,\dots,T_n\}/(T_0 \cdots T_n - a))$ 
	where $T_0 \cdots T_k$ does not vanish.
	$\mathfrak{U}$ is isomorphic to the formal spectrum of
	\[A := K^{\circ}\{T_0,\dots,T_k, S, T_{k+1},\dots,T_n\}/(T_0 \cdots T_k S - 1, T_0 \cdots T_n - a). \]  
	We have a $K^{\circ}$-algebra isomorphism between $A$ and 
	\[B := K^{\circ}\{T_0,\dots,T_k, S, T_{k+1},\dots,T_n\}/(T_0 \cdots T_k S - 1, T_{k+1} \cdots T_n - a). \] 
	Explicitly it is given by keeping the coordinates $T_0,\dots,T_k, S, T_{k+1},\dots,T_{n-1}$ fixed and
	\begin{align*}
	&A \xrightarrow{\sim} B ~,~ T_n \mapsto ST_n \\ 
	&B \xrightarrow{\sim} A ~,~ T_n \mapsto T_0 \cdots T_k T_n.
	\end{align*} 	
	This shows that $\mathfrak{U}$ is isomorphic to $\mathfrak{T}(k+1, 1) \times \mathfrak{T}(n-k-1, a)$.
	We will use this later to embed a standard scheme from which we removed the zero locus of some coordinates,
	into another standard scheme.
\end{noname}

\begin{definition} \label{StrPstDef}
	Let $\mathfrak{X}$ be an admissible formal $K^{\circ}$-scheme 
	and $\mathfrak{H}$ be a closed subscheme of $\mathfrak{X}$.
	We call $(\mathfrak{X}, \mathfrak{H})$ a \emph{strictly poly-stable pair},
	if $\mathfrak{X}$ can be covered by open subschemes $\mathfrak{U}$ such that for every $\mathfrak{U}$ 
	there exists a standard pair $(\mathfrak{S}, \mathfrak{G})$
	and an \et morphism $\varphi: \mathfrak{U} \to \mathfrak{S}$
	with the property that $\mathfrak{H} \cap \mathfrak{U}$ is equal as a formal scheme
	to the pullback $\varphi^{-1}(\mathfrak{G})$.
	We will write this shortened by saying that we have an \et morphism 
	$\varphi: (\mathfrak{U}, \mathfrak{H} \cap \mathfrak{U}) \to (\mathfrak{S}, \mathfrak{G})$.
	Note that in this case $\varphi$ restricts to an \et morphism 
	$\varphi: \mathfrak{H} \cap \mathfrak{U} \to \mathfrak{G}$.  
	
	If above every standard scheme $\mathfrak{S}$ can be chosen to be of the form $\mathfrak{S}(\textbf{n}, \textbf{a}, d)$ 
	with $1$-tuples $\textbf{n}$ and $\textbf{a}$, then the pair $(\mathfrak{X}, \mathfrak{H})$ 
	is called \emph{strictly semi-stable}.
	
	Moreover $\mathfrak{X}$ is called \emph{strictly poly-stable} resp. \emph{strictly semi-stable}, 
	if $(\mathfrak{X}, \emptyset)$ is strictly poly-stable resp. strictly semi-stable.	 
\end{definition}

\begin{remark}
	This definition of a strictly poly-stable scheme is slightly different from the one by Berkovich,
	since he considers standard pairs to be of the type $\mathfrak{S}(\textbf{n}, \textbf{a}) \times \mathfrak{T}(d, 1)$.
	However they are equivalent, because $\mathfrak{B}^d$ has an open cover by finitely many copies of $\mathfrak{T}(d, 1)$,
	see \ref{ToriCover}.    
\end{remark}

\begin{example}
	Let $\mathfrak{X}$ be an admissible formal $K^{\circ}$-scheme and $\mathfrak{H} \subseteq \mathfrak{X}$ be a closed subscheme.
	\begin{enumerate}
		\item Obviously every standard pair is strictly poly-stable.
		\item $(\mathfrak{X}, \mathfrak{H})$ is strictly poly-stable, if and only if for every open subset $\mathfrak{U} \subseteq \mathfrak{X}$
		the pair $(\mathfrak{U}, \mathfrak{H} \cap \mathfrak{U})$ is strictly poly-stable. 
		\item Let $\mathfrak{X}$ be smooth, 
		which means that $\mathfrak{X}$ can be covered by open subschemes $\mathfrak{U}$ such that for every $\mathfrak{U}$ 
		there exists an integer $d \in \N$ and an \et morphism $\mathfrak{U} \to \mathfrak{S}(d)$.
		Then $\mathfrak{X}$ is strictly poly-stable.  
	\end{enumerate}
\end{example}

\begin{definition} \label{PstDef}
	Let $\mathfrak{X}$ be an admissible formal $K^{\circ}$-scheme 
	and $\mathfrak{H}$ be a closed subscheme of $\mathfrak{X}$.
	We call $(\mathfrak{X}, \mathfrak{H})$ a \emph{poly-stable pair}, if there exists a strictly poly-stable pair
	$(\mathfrak{Y}, \mathfrak{G})$ and a surjective \et morphism 
	$\psi: \mathfrak{Y} \to \mathfrak{X}$ such that $\mathfrak{G}$ is equal as a formal scheme to
	the pullback $\psi^{-1}(\mathfrak{H})$.  
	We will write this shortened by saying that we have a surjective \et morphism 
	$\psi: (\mathfrak{Y}, \mathfrak{G}) \to (\mathfrak{X}, \mathfrak{H})$.
	
	The pair $(\mathfrak{X}, \mathfrak{H})$ is called \emph{semi-stable}, if $(\mathfrak{Y}, \mathfrak{G})$
	can be chosen to be a strictly semi-stable pair.  
	
	Moreover $\mathfrak{X}$ is called \emph{poly-stable} resp. \emph{semi-stable}, 
	if $(\mathfrak{X}, \emptyset)$ is poly-stable resp. semi-stable.
\end{definition}

\begin{remark}
	If $\Spf(A)$ is a poly-stable scheme, then $A$ is a reduced $K^{\circ}$-algebra.
	This is clear, if $K$ is trivially valued. Otherwise, this is an easy consequence of \cite[Proposition~1.4]{Ber99}.
\end{remark}

\begin{definition}
	Let $\mathcal{X}$ be a scheme of finite type over a field $k$. We call $\mathcal{X}$ 
	\emph{strictly poly-stable}, if it can be covered by open subschemes $\mathcal{U}$, each admitting $p, d, n_i \in \N$ and an \et morphism $\mathcal{U}  \to \Spec\left(\bigotimes_{i=1}^p k[T_{i0},\dots,T_{in_i}]/(T_{i0} \cdots T_{in_i}) 
	\otimes k[T_1,\dots,T_d]\right)$.
	
	We call $\mathcal{X}$ \emph{poly-stable}, if there exists a strictly poly-stable $\mathcal{Y}$ and a surjective \et morphism $\mathcal{Y} \to \mathcal{X}$. 
	
	It is an easy consequence of \cite[Proposition 17.5.7]{EGAIV4}, that every poly-stable
	$k$-scheme is reduced.	
\end{definition}

\begin{example}
	Clearly every strictly poly-stable pair is also poly-stable.  
	
	However the \emph{nodal cubic}, by which we mean the affine curve over $\widetilde{K}$ given by the equation 
	$y^2 = x^2(x+1)$, is an example of a poly-stable scheme which is not strictly poly-stable,
	see also Example \ref{NonStrict}.
\end{example}

\begin{proposition} \label{DivStrPst}
	Let $(\mathfrak{X}, \mathfrak{H})$ be a poly-stable resp. strictly poly-stable pair.
	Then $\mathfrak{X}$ and $\mathfrak{H}$ are poly-stable resp. strictly poly-stable. 
	In particular the special fibers $\mathfrak{X}_s$ and $\mathfrak{H}_s$ are poly-stable resp. strictly poly-stable.   
\end{proposition}

\begin{proof}
	The claim is immediately clear for $\mathfrak{X}$. To see the claim for $\mathfrak{H}$
	we note that for every standard pair $(\mathfrak{S}(\textbf{n}, \textbf{a}, d), \mathfrak{G}(s))$
	with $s \geq 1$ the closed subscheme $\mathfrak{G}(s)$ is a standard scheme as well, namely
	\[\mathfrak{G}(s) = \prod\limits_{i=1}^p \mathfrak{T}(n_i, a_i)
	\times \mathfrak{T}(s-1, 0) \times \mathfrak{B}^{d-s}. \] 
	If $s=1$, then $\mathfrak{T}(s-1, 0) = \mathfrak{T}(0, 0) = \Spf(K^{\circ})$
	and this factor can be ignored.		
\end{proof}

\begin{remark}
	Let $\mathcal{X}$ be a strictly poly-stable $\widetilde{K}$-scheme and $x \in \mathcal{X}$. It follows easily from the definition that
	all irreducible components of $\mathcal{X}$ passing through $x$ have the same dimension.
	This also implies that the connected components of $\mathcal{X}$ are equidimensional. 
\end{remark}


\subsection{Stratification}

The goal of this subsection is to generalize the notion of “stratum” from \cite[§\,2]{Ber99} to
our situation involving pairs. Many properties stay true in this context, in particular we get 
a nice explicit description of the strata of strictly poly-stable pairs analogous to \emph{loc.cit.}, Proposition 2.5.

\begin{definition}
	Let $\mathcal{X}$ be a reduced scheme of locally finite type over a field $k$. We set $\Nor(\mathcal{X})$ for the normality locus of $\mathcal{X}$,
	which is open and dense in $\mathcal{X}$. We define  $\mathcal{X}^{(0)} := \mathcal{X}$ and inductively $\mathcal{X}^{(i+1)} := \mathcal{X}^{(i)} \setminus \Nor(\mathcal{X}^{(i)})$ for $i \in \N$. 
	A \emph{stratum} of $\mathcal{X}$ is an irreducible component of $\Nor(\mathcal{X}^{(i)})$ for some $i \in \N$.
	Note that the strata are disjoint locally closed subsets of $\mathcal{X}$ and that the family of all strata  
	yields a locally finite cover of $\mathcal{X}$. We denote by $\str(\mathcal{X})$ the set of the generic points of the strata of $\mathcal{X}$.
	There is a bijective correspondence between $\str(\mathcal{X})$ and the set of strata of $\mathcal{X}$.  
	Moreover we denote by $\irr(\mathcal{X})$ the set of irreducible components of $\mathcal{X}$. 
	
	Now let $\mathcal{H} \subseteq \mathcal{X}$ be a closed subset, which we will consider as a 
	closed subscheme of $\mathcal{X}$ via the induced reduced structure. A \emph{stratum} of the pair $(\mathcal{X}, \mathcal{H})$  
	is either a stratum of $\mathcal{H}$ or a non-empty open subscheme of $S$ of the form $S \setminus \mathcal{H}$, where $S$ is a stratum of $\mathcal{X}$.
	Again they are irreducible, locally closed and disjoint and the family of all strata yields a locally finite cover of $\mathcal{X}$.
	We denote $\str(\mathcal{X}, \mathcal{H}) := \str(\mathcal{X}) \cup \str(\mathcal{H})$. 
	There is a bijective correspondence between $\str(\mathcal{X}, \mathcal{H})$ 
	and the set of strata of $(\mathcal{X}, \mathcal{H})$.
	We call the pair $(\mathcal{X}, \mathcal{H})$ \emph{well-stratified}, if no stratum of $\mathcal{X}$ is contained in $\mathcal{H}$.  
	This is equivalent to the assertion that the sets $\str(\mathcal{X})$ and $\str(\mathcal{H})$ are disjoint.	
	
	We introduce a partial order on $\str(\mathcal{X}, \mathcal{H})$ by setting for all $x, y \in \str(\mathcal{X}, \mathcal{H})$ 
	that $x \leq y$ if and only if $x \in \overline{\{y\}}$, where the closure is taken in $\mathcal{X}$.
	The pair $(\mathcal{X}, \mathcal{H})$ is called \emph{elementary}, if $\str(\mathcal{X}, \mathcal{H})$ has a unique minimal element.
	We call $\mathcal{X}$ \emph{elementary}, if $(\mathcal{X}, \emptyset)$ is elementary. 
	
	A \emph{strata subset} of $\mathcal{X}$ resp. $(\mathcal{X}, \mathcal{H})$
	is a union of strata of $\mathcal{X}$ resp. $(\mathcal{X}, \mathcal{H})$.
\end{definition}

\begin{remark} \label{OpenStratRem}
	Let $\mathcal{X}$ be a reduced scheme of locally finite type over a field $k$ and $\mathcal{U} \subseteq \mathcal{X}$
	an open subscheme. Then the strata of $\mathcal{U}$ are given as the non-empty intersections $S \cap \mathcal{U}$ 
	for all strata $S$ of $\mathcal{X}$. One easily concludes $\str(\mathcal{U}) = \{x \in \str(\mathcal{X}) ~|~ x \in \mathcal{U}\}$.  
\end{remark}

\begin{remark} \label{StratRem}
	Since $\mathcal{X}$ is locally noetherian, there exist 
	no infinite strictly descending chains of elements in $\str(\mathcal{X}, \mathcal{H})$.
	In particular $(\mathcal{X}, \mathcal{H})$ is elementary
	if and only if $\str(\mathcal{X}, \mathcal{H})$ has a least element. 
\end{remark}

\begin{proposition} \label{ClosureStrataSet}
	Let $\mathcal{X}$ be a poly-stable scheme. Then the closure of every stratum of $\mathcal{X}$
	is a strata subset of $\mathcal{X}$. 
\end{proposition}

\begin{proof}
	See \cite[Lemma 2.1 (i)]{Ber99}.
\end{proof}

\begin{proposition} \label{StrataMap}
	Let $\psi: \mathcal{Y} \rightarrow \mathcal{X}$ be an \et morphism of reduced schemes of locally finite type over a field $k$.
	Then $\psi$ induces an \et morphism from each stratum of $\mathcal{Y}$ to a stratum of $\mathcal{X}$. 
	In particular the preimage of a stratum of $\mathcal{X}$ is a strata subset of $\mathcal{Y}$.
\end{proposition}

\begin{proof}
	See \cite[Lemma 2.2 (i)]{Ber99}.
\end{proof}

\begin{proposition} \label{WellStrat}
	If $(\mathfrak{X}, \mathfrak{H})$ is a poly-stable pair, then $(\mathfrak{X}_s, \mathfrak{H}_s)$ is well-stratified.
	In particular every element of $\str(\mathfrak{X}_s, \mathfrak{H}_s)$ corresponds to a stratum of either
	$\mathfrak{X}_s$ or $\mathfrak{H}_s$.
\end{proposition}

\begin{proof}
	First one easily verifies the claim for standard pairs.
	
	Let now $(\mathfrak{X}, \mathfrak{H})$ be strictly poly-stable. Assume for contradiction that there exists a stratum with generic point
	$x \in \str(\mathfrak{X}_s)$ which is contained in $\mathfrak{H}_s$. Then we find an open neighborhood $\mathfrak{U}$ of $x$
	and an \et morphism $\varphi: (\mathfrak{U}, \mathfrak{H} \cap \mathfrak{U}) \to (\mathfrak{S}, \mathfrak{G})$.
	Then $x \in \str(\mathfrak{U}_s)$ and let $S$ be the corresponding stratum in $\mathfrak{U}_s$.
	Now according to Proposition~\ref{StrataMap} the map $\varphi$ induces an \et morphism 
	from $S$ to a stratum of $\mathfrak{S}$. Let $y \in \str(\mathfrak{S}_s)$ be its generic point. 
	Then $y = \varphi_s(x) \in \varphi_s(S) \subseteq \varphi_s((\mathfrak{H} \cap \mathfrak{U})_s) \subseteq \mathfrak{G}_s$.
	This is impossible for standard pairs. 
	
	Finally consider an arbitrary poly-stable pair $(\mathfrak{X}, \mathfrak{H})$. By definition there exists a surjective \et
	morphism $\psi: (\mathfrak{Y}, \mathfrak{G}) \to (\mathfrak{X}, \mathfrak{H})$ with $ (\mathfrak{Y}, \mathfrak{G})$ strictly poly-stable.
	Assume for contradiction that there exists a stratum $S$ of $\mathfrak{X}_s$ which is contained in $\mathfrak{H}_s$.
	The preimage $\varphi^{-1}_s(S)$ is a non-empty strata subset of $\mathfrak{Y}_s$ which is contained in $\mathfrak{G}_s$. But this can not
	happen for strictly poly-stable pairs.
\end{proof}

\begin{example} \label{StratExample}
	The following graphic illustrates the stratification of the $\widetilde{K}$-scheme $\mathcal{X} = \Spec(\widetilde{K}[T_0,T_1,T_2]/(T_0 T_1 T_2))$: 
	\begin{center}\includegraphics[scale=1]{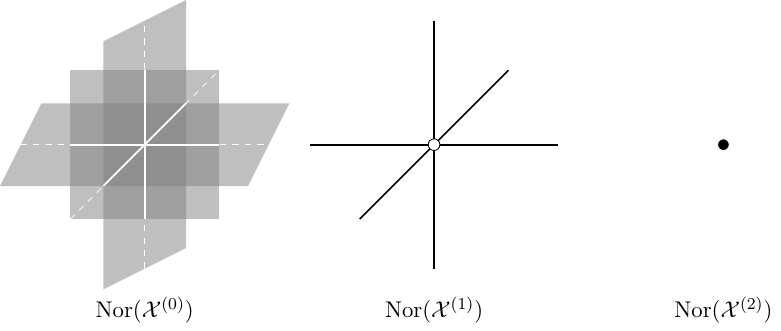}\end{center}
\end{example}

\begin{example}
	Let $(\mathfrak{S}, \mathfrak{G}) := (\mathfrak{S}(\textbf{n}, \textbf{a}, d), \mathfrak{G}(s))$ be a standard pair. Then 
	\[\mathfrak{S}_s = \Spec\left(\bigotimes\limits_{i=1}^p \widetilde{K}[T_{i0},\dots,T_{in_i}]/(T_{i0} \cdots T_{in_i}) 
	\otimes \widetilde{K}[T_1,\dots,T_d]\right) \]
	and the pair $(\mathfrak{S}_s, \mathfrak{G}_s)$ is elementary with the minimal stratum being the closed subset of
	$\mathfrak{S}_s$ cut out by $T_{ij} = 0$ for all $i \in \dotsbra{1}{p}$, $j \in \dotsbra{0}{n_i}$
	and $T_1 = 0, \dots ,T_s = 0$. The minimal stratum has codimension $|\textbf{n}| + s$ in $\mathfrak{S}_s$
	and every other stratum has lower codimension. 
\end{example}

\begin{definition} \label{IrrMetric}
	Let $\mathcal{X}$ be a strictly poly-stable scheme and $x \in \mathcal{X}$. 
	We write $\irr(\mathcal{X}, x)$	for the set of all irreducible components of $\mathcal{X}$ containing $x$.
	This set is equipped with a metric by assigning to two elements $\mathcal{V}, \mathcal{V}' \in \irr(\mathcal{X}, x)$ 
	the codimension of the intersection $\mathcal{V} \cap \mathcal{V}'$ at $x$ as their distance. 
	More explicitly, all irreducible components of $\mathcal{X}$ containing $x$ have the same dimension, say $r$. 
	Since $\mathcal{V} \cap \mathcal{V}'$ is smooth, its irreducible components are disjoint,
	meaning that there is exactly one irreducible component of $\mathcal{V} \cap \mathcal{V}'$
	which contains $x$. Let $s$ be its dimension. Then the requested codimension is equal to $r-s$. 
	
	If $y \in \mathcal{X}$ such that $x \in \overline{\{y\}}$, then the inclusion 
	$\irr(\mathcal{X}, y) \hookrightarrow \irr(\mathcal{X}, x)$ is isometric.
\end{definition}

\begin{proposition} \label{IrrIsometric}
	Let $\psi: \mathcal{Z} \rightarrow \mathcal{Y}$ be an \et morphism  
	of strictly poly-stable $\widetilde{K}$-schemes with $y \in \mathcal{Y}$, $z \in \mathcal{Z}$ such that 
	$y \in \overline{\{\psi(z)\}}$. 
	\begin{enumerate}
		\item For every $\mathcal{V} \in \irr(\mathcal{Z}, z)$ we have $\overline{\psi(\mathcal{V})} \in \irr(\mathcal{Y}, y)$.
		\item The induced map $\alpha: \irr(\mathcal{Z}, z) \rightarrow \irr(\mathcal{Y}, y)$ is isometric.
		If $y = \psi(z)$, then $\alpha$ is an isometric bijection. 
	\end{enumerate}
\end{proposition}

\begin{proof} 
	We start with (i).
	Let $\mathcal{V} \in \irr(\mathcal{Z}, z)$. Since $\mathcal{Z}$ is locally noetherian, there exists a non-empty open affine 
	subset $\mathcal{U} \subset \mathcal{Z}$ contained in $\mathcal{V}$. Note that $\psi(\mathcal{U}) \subset \mathcal{Y}$ 
	is open, because \et morphisms are open. Thus $\overline{\psi(\mathcal{V})}$ is an irreducible
	closed subset which contains a non-empty open subset of $\mathcal{Y}$. This implies that $\overline{\psi(\mathcal{V})}$
	is an irreducible component of $\mathcal{Y}$. From $z \in \mathcal{V}$ we deduce
	$y \in \overline{\{\psi(z)\}} \subseteq \overline{\psi(\mathcal{V})}$.     
	
	In order to show (ii), let $\mathcal{V}, \mathcal{V}' \in \irr(\mathcal{Z}, z)$. Then 
	$\dim(\mathcal{V}) = \dim(\mathcal{V}') =\dim(\alpha(\mathcal{V}')) =\dim(\alpha(\mathcal{V}))$.
	Let $\mathcal{W}$ be the irreducible component of $\alpha(\mathcal{V}) \cap \alpha(\mathcal{V}')$
	which contains $\psi(z)$. Then it also contains $y$. Note that $\psi^{-1}(\alpha(\mathcal{V}))$
	is the union of all $\mathcal{V}'' \in \irr(\mathcal{Z})$ such that $\overline{\psi(\mathcal{V}'')} = \alpha(\mathcal{V})$.
	This union is disjoint, because $\alpha(\mathcal{V})$ is smooth. Also $\psi^{-1}(\mathcal{W})$ is the disjoint
	union of its irreducible components, because $\mathcal{W}$ is smooth.
	It follows that the irreducible component of $\psi^{-1}(\mathcal{W})$ which contains $z$, is
	equal to the irreducible component of $\mathcal{V} \cap \mathcal{V}'$ which contains $z$. 
	Since it has the same dimension as $\mathcal{W}$, we have shown that $\alpha$ is isometric.
	
	Now we additionally assume $y = \psi(z)$. Take any $\mathcal{W} \in \irr(\mathcal{Y}, y)$.
	At least one of the irreducible components of the closed subset $\psi^{-1}(\mathcal{W})$ contains $z$.
	This irreducible component is contained in some $\mathcal{V} \in \irr(\mathcal{Z}, z)$. Then $\alpha(\mathcal{V}) = \mathcal{W}$.  
	Therefore $\alpha$ is also surjective. 
\end{proof}

\begin{proposition} \label{IrredCompProp}
	Let $(\mathfrak{X}, \mathfrak{H})$ be a strictly poly-stable pair.
	For every irreducible component $\mathcal{V}$ of $\mathfrak{H}_s$ there
	is exactly one irreducible component $\mathcal{W}$ of $\mathfrak{X}_s$
	such that $\mathcal{V} \subseteq \mathcal{W}$. Then $\mathcal{V} \neq \mathcal{W}$ 
	and in particular $\irr(\mathfrak{X}_s)$ and $\irr(\mathfrak{H}_s)$ are disjoint.  
\end{proposition}

\begin{proof}
	Let $\mathcal{V}$ be an irreducible component of $\mathfrak{H}_s$.
	Then clearly $\mathcal{V}$ is contained in an irreducible component of $\mathfrak{X}_s$.
	Let $x$ be the generic point of $\mathcal{V}$ and $\mathfrak{U}$ be an open neighborhood of $x$ in $\mathfrak{X}$ such that 
	there exists a standard pair $(\mathfrak{S}, \mathfrak{G})$ and an \et morphism 
	$\varphi: (\mathfrak{U}, \mathfrak{H} \cap \mathfrak{U}) \to (\mathfrak{S}, \mathfrak{G})$.   
	Let $y := \varphi(x)$. Then $\varphi$ induces isometric bijections
	$\alpha: \irr(\mathfrak{U}_s, x) \xrightarrow{\sim} \irr(\mathfrak{S}_s, y)$ and 
	$\beta: \irr((\mathfrak{H} \cap \mathfrak{U})_s, x) \xrightarrow{\sim} \irr(\mathfrak{G}_s, y)$.
	Consider irreducible components $\mathcal{W}_1$ and $\mathcal{W}_2$ of $\mathfrak{X}_s$ which contain $\mathcal{V}$.
	We assume for contradiction that $\mathcal{W}_1 \neq \mathcal{W}_2$.
	Then $\beta(\mathcal{V} \cap \mathfrak{U}_s)$ is contained in the two distinct irreducible components 
	$\alpha(\mathcal{W}_1 \cap \mathfrak{U}_s)$ and $\alpha(\mathcal{W}_2 \cap \mathfrak{U}_s)$ of $\mathfrak{S}_s$
	and thus $\beta(\mathcal{V} \cap \mathfrak{U}_s)$ has codimension at least $2$ in $\mathfrak{S}_s$, which is absurd, since every 
	irreducible component of $\mathfrak{G}_s$ has codimension $1$ in $\mathfrak{S}_s$. We conclude $W_1 = W_2$.
	Because $\alpha(\mathcal{W}_1 \cap \mathfrak{U}_s)$ and $\beta(\mathcal{V} \cap \mathfrak{U}_s)$
	have different dimensions, it follows that $\mathcal{V} \neq \mathcal{W}_1$.
\end{proof}

\begin{definition}
	For every subset $A \subseteq \irr(\mathfrak{H}_s)$ let us denote
	\[I^{(A)} := \{\mathcal{W} \in \irr(\mathfrak{X}_s) ~|~ 
	\mathcal{V} \subseteq \mathcal{W} \textnormal{ for some } \mathcal{V} \in A\} \subseteq \irr(\mathfrak{X}_s).\] 
\end{definition}

\begin{proposition} \label{StrataFormSingle}
	Let $\mathfrak{X}$ be strictly poly-stable. Then the intersection of any set of irreducible components of 
	$\mathfrak{X}_s$ is smooth and the family of strata of $\mathfrak{X}_s$ coincides with the family of
	irreducible components of non-empty sets of the form 
	\begin{align*}
	\bigcap\limits_{\mathcal{V} \in A} \mathcal{V} \setminus 
	\bigcup\limits_{\mathcal{W} \in \irr(\mathfrak{X}_s) \setminus A} \mathcal{W},
	\end{align*}
	where $A \subseteq \irr(\mathfrak{X}_s)$ is a finite non-empty subset. In particular every non-empty set of the above form 
	is a disjoint union of strata of $\mathfrak{X}_s$. 
	
	Moreover all strata of a poly-stable scheme are smooth.
\end{proposition}

\begin{proof}
	See \cite[Proposition 2.5]{Ber99} and \cite[Corollary 2.6]{Ber99}.
	Also note that the irreducible components of a smooth scheme are disjoint.  
\end{proof}

\begin{proposition} \label{StrataForm}
	Let $(\mathfrak{X}, \mathfrak{H})$ be a strictly poly-stable pair. We denote 
	\[\irr(\mathfrak{X}_s+\mathfrak{H}_s) := \irr(\mathfrak{X}_s) \cup \irr(\mathfrak{H}_s).\]
	Let $A \subseteq \irr(\mathfrak{X}_s+\mathfrak{H}_s)$ be finite and non-empty such that the intersection
	of all $\mathcal{V} \in A$ is non-empty.
	We define the locally closed subset
	\[\mathcal{X}_A := \bigcap\limits_{\mathcal{V} \in A} \mathcal{V} \setminus 
	\bigcup\limits_{\mathcal{W} \in \irr(\mathfrak{X}_s+\mathfrak{H}_s) \setminus A} \mathcal{W} \]	
	of $\mathfrak{X}_s$. If $\mathcal{X}_A$ is non-empty, then
	$A \subseteq \irr(\mathfrak{X}_s)$ or $A \cap \irr(\mathfrak{X}_s) = I^{(A \cap \irr(\mathfrak{H}_s))}$.  
	In this case the irreducible components of $\mathcal{X}_A$ are strata of $(\mathfrak{X}_s, \mathfrak{H}_s)$.
	Conversely, every stratum of $(\mathfrak{X}_s, \mathfrak{H}_s)$ is given as an irreducible component of 
	$\mathcal{X}_A$ for a suitable $A$, and this $A$ is uniquely determined.  
\end{proposition} 

\begin{proof}
	Assume that $\mathcal{X}_A$ and $A \cap \irr(\mathfrak{H}_s)$ are non-empty.
	It is clear that $A \cap \irr(\mathfrak{X}_s) \supseteq I^{(A \cap \irr(\mathfrak{H}_s))}$. 
	Now let $\mathcal{W} \in A \cap \irr(\mathfrak{X}_s)$ and $x \in \mathcal{X}_A$.
	Because $A \cap \irr(\mathfrak{H}_s)$ is non-empty, it follows that $x \in \mathfrak{H}_s$.
	The intersection $\mathcal{W} \cap \mathfrak{H}_s$ is a union of irreducible components of $\mathfrak{H}_s$.
	Let $\mathcal{V}$ be the one which contains $x$. Then $\mathcal{V} \in A \cap \irr(\mathfrak{H}_s)$
	and therefore $\mathcal{W} \in I^{(A \cap \irr(\mathfrak{H}_s))}$.
	
	The remaining claims are obvious due to Proposition \ref{IrredCompProp} and Proposition \ref{StrataFormSingle}.      
\end{proof} 

\begin{remark}
	The statement above implies, that our notion of “stratum” agrees with the 
	notion of “vertical stratum” as introduced in \cite[3.15.1]{GRW16}.
\end{remark}

\begin{corollary} \label{ElementaryCor}
	Let $(\mathfrak{X}, \mathfrak{H})$ be a strictly poly-stable pair such that
	$(\mathfrak{X}_s, \mathfrak{H}_s)$ is elementary. Then also $\mathfrak{X}_s$ and $\mathfrak{H}_s$ 
	are elementary. The minimal element of $\str(\mathfrak{X}_s, \mathfrak{H}_s)$ and $\str(\mathfrak{H}_s)$
	resp. $\str(\mathfrak{X}_s)$ is equal to the generic point of the intersection of all elements in 
	$\irr(\mathfrak{H}_s)$ resp. $\irr(\mathfrak{X}_s)$.
\end{corollary}

\begin{proof}
	It follows from Remark \ref{StratRem} and Proposition \ref{StrataForm} that the generic point of the minimal
	stratum is contained in every element of $\irr(\mathfrak{X}_s + \mathfrak{H}_s)$. Moreover the intersection of any set of
	elements in $\irr(\mathfrak{X}_s+\mathfrak{H}_s)$ has to be non-empty and irreducible. Indeed
	the smoothness of such an intersection implies that its irreducible components are disjoint and there can 
	only exist a least element of $\str(\mathfrak{X}_s, \mathfrak{H}_s)$, if there is exactly one irreducible component. 
	Now the claims become apparent and we additionally discover that the least elements of $\str(\mathfrak{X}_s, \mathfrak{H}_s)$ 
	and $\str(\mathfrak{H}_s)$ coincide, since every irreducible component of $\mathfrak{H}_s$ is contained 
	in an irreducible component of $\mathfrak{X}_s$. 
\end{proof}


\subsection{Building blocks}

Let for this subsection $(\mathfrak{X}, \mathfrak{H})$ be a strictly poly-stable pair
and $x \in \mathfrak{X}$. We establish the notion of “building blocks”, which basically means a 
suitable neighborhood of $x$ where on the special fiber we have the same stratification and intersection behavior
of irreducible components as for standard pairs. This is essentially the situation which is classically
dealt with in Step $6$ from the Proof of Theorems 5.2–5.4 in \cite[§\,5]{Ber99}, when constructing the skeleton.
I also was heavily influenced by \cite[Proposition 5.2]{Gub10} and \cite[Proposition 4.1]{GRW16}
to make these considerations. For later reference we will study the structure of the building blocks in detail.

\begin{proposition} \label{RespDivCoord}
	Let $\mathfrak{U}$ resp. $\mathfrak{U}'$ be an open neighborhood of $x$ such that 
	there exists a standard pair $(\mathfrak{S}, \mathfrak{G})$resp. $(\mathfrak{S}', \mathfrak{G}')$ and an \et morphism 
	$\varphi: (\mathfrak{U}, \mathfrak{H} \cap \mathfrak{U}) \to (\mathfrak{S}, \mathfrak{G})$ resp. 
	$\varphi': (\mathfrak{U}', \mathfrak{H} \cap \mathfrak{U}') \to (\mathfrak{S}', \mathfrak{G}')$,
	where $\mathfrak{G} = \{T_1 \cdots T_s = 0\}$ and $\mathfrak{G}' = \{T'_1 \cdots T'_{s'} = 0\}$.
	The morphism $\varphi$ resp. $\varphi'$ induces an isometric bijection 
	$\beta: \irr((\mathfrak{H} \cap \mathfrak{U})_s, x) \xrightarrow{\sim} \irr(\mathfrak{G}_s, y)$
	resp. $\beta': \irr((\mathfrak{H} \cap \mathfrak{U}')_s, x) \xrightarrow{\sim} \irr(\mathfrak{G}'_s, y')$, where $y := \varphi(x)$
	and $y' := \varphi'(x)$. 
	\begin{enumerate}
		\item For every $\mathcal{V} \in \irr(\mathfrak{H}_s, x)$ there exists exactly one $i \in \dotsbra{1}{s}$
		such that $\beta(\mathcal{V} \cap \mathfrak{U}_s) \subseteq \{T_i = 0\}$.
		\item Let $\mathcal{V}_1, \mathcal{V}_2  \in \irr(\mathfrak{H}_s, x)$ such that 
		$\beta(\mathcal{V}_1 \cap \mathfrak{U}_s) \subseteq \{T_i = 0\}$ and 
		$\beta(\mathcal{V}_2 \cap \mathfrak{U}_s) \subseteq \{T_i = 0\}$
		for some $i \in \dotsbra{1}{s}$.
		Then there exists an index $j \in \dotsbra{1}{s'}$ such that
		$\beta'(\mathcal{V}_1 \cap \mathfrak{U}'_s) \subseteq \{T'_j = 0\}$ and 
		$\beta'(\mathcal{V}_2 \cap \mathfrak{U}'_s) \subseteq \{T'_j = 0\}$.	  
	\end{enumerate}
\end{proposition}

\begin{proof}
	(i) follows immediately from the explicit description of the irreducible components of $\mathfrak{G}_s$.
	
	Next we want to prove (ii). 
	According to Proposition \ref{IrredCompProp} there exist unique $\mathcal{W}_1, \mathcal{W}_2 \in \irr(\mathfrak{X}_s, x)$
	such that $\mathcal{V}_1 \subseteq \mathcal{W}_1$ and $\mathcal{V}_2 \subseteq \mathcal{W}_2$.
	Let $d$ be the distance between $\mathcal{W}_1$ and $\mathcal{W}_2$ with respect to the metric
	introduced in Definition \ref{IrrMetric}. Then $\mathcal{V}_1$ and $\mathcal{V}_1$
	have distance $d+1$. The inclusions induce isometric bijections 
	$\irr((\mathfrak{H} \cap \mathfrak{U})_s, x) \xrightarrow{\sim} \irr(\mathfrak{H}_s, x)$
	and $\irr((\mathfrak{H} \cap \mathfrak{U}')_s, x) \xrightarrow{\sim} \irr(\mathfrak{H}_s, x)$.
	If there were distinct indexes $j, k \in \dotsbra{1}{s'}$ such that
	$\beta'(\mathcal{V}_1 \cap \mathfrak{U}'_s) \subseteq \{T'_j = 0\}$ and 
	$\beta'(\mathcal{V}_2 \cap \mathfrak{U}'_s) \subseteq \{T'_k = 0\}$, then 
	$\beta'(\mathcal{V}_1 \cap \mathfrak{U}'_s) \cap \beta'(\mathcal{V}_2 \cap \mathfrak{U}'_s)$
	would have codimension $d+2$ in $\mathfrak{S}'_s$, implying that 
	$\beta(\mathcal{V}_1 \cap \mathfrak{U}_s) \cap \beta(\mathcal{V}_2 \cap \mathfrak{U}_s)$
	would have codimension $d+2$ in $\mathfrak{S}_s$.
	But $\beta(\mathcal{V}_1 \cap \mathfrak{U}_s) \cap \beta(\mathcal{V}_2 \cap \mathfrak{U}_s)$ 
	has codimension $d+1$ in $\mathfrak{S}_s$, which finishes the proof.
\end{proof}

\begin{proposition} \label{BuildingBlockExist}
	There exists an affine open neighborhood $\mathfrak{U}$ of $x$, a standard pair $(\mathfrak{S}, \mathfrak{G})$
	and an \et morphism $\varphi: (\mathfrak{U}, \mathfrak{H} \cap \mathfrak{U}) \to (\mathfrak{S}, \mathfrak{G})$
	such that the following hold: 
	\begin{enumerate}
		\item The pair $(\mathfrak{U}_s, (\mathfrak{H} \cap \mathfrak{U})_s)$ is elementary
		and its minimal stratum contains $x$.
		\item The point $\varphi(x)$ is contained in every irreducible component of $\mathfrak{S}_s$ and $\mathfrak{G}_s$.		
	\end{enumerate}
\end{proposition}

\begin{proof}
	Let $\mathfrak{U}'$ be an open neighborhood of $x$ such that there exists an \et morphism 
	$\varphi: (\mathfrak{U}', \mathfrak{H} \cap \mathfrak{U}') \to (\mathfrak{S}(\textbf{n}, \textbf{a}, d), \mathfrak{G}(s))$
	with $p$-tuples $\textbf{n}$ and $\textbf{a}$.
	We remove from $\mathfrak{S}(\textbf{n}, \textbf{a}, d)$ all closed subsets $\{T_{ij} = 0\}$ with
	$i \in \dotsbra{1}{p}$, $j \in \dotsbra{0}{n_i}$ such that $\varphi(x) \notin \{T_{ij} = 0\}$.
	The result is an open subset of $\mathfrak{S}(\textbf{n}, \textbf{a}, d)$ which
	is a product $\mathfrak{T} := \mathfrak{S}(\textbf{n}', \textbf{a}') \times \mathfrak{B}^d \times \mathfrak{T}(d', 1)$,
	as indicated in \ref{RemoveCoord}.   
	Let $\mathfrak{U}'' := \varphi^{-1}(\mathfrak{T})$. This is an open neighborhood of $x$.
	The torus $\mathfrak{T}(d', 1)$ can be openly embedded into a ball.
	Therefore we get an \et morphism $\varphi: \mathfrak{U}'' \to \mathfrak{S}$ to a standard scheme 
	$\mathfrak{S} = \mathfrak{S}(\textbf{n}', \textbf{a}') \times \mathfrak{B}^d \times \mathfrak{B}^{d'}$,
	whose irreducible components each contain $\varphi(x)$.
	
	Consider the subset of all elements $i \in \dotsbra{1}{s}$ such that $\varphi(x) \in \{T_i = 0\}$.
	After a change of coordinates we may assume it is of the form $\dotsbra{1}{s'}$
	for some $s' \leq s$. Let then $\mathfrak{G}$ be the closed subset $\{T_1 \cdots T_{s'} = 0\}$ of $\mathfrak{S}$.
	Then $\varphi(x)$ is contained in every irreducible component of $\mathfrak{G}_s$.
	
	Finally we remove from $\mathfrak{U}''$ the closed subsets $\overline{\{y\}}$ for all $y \in \str(\mathfrak{X}_s, \mathfrak{H}_s)$
	with $x \notin \overline{\{y\}}$. This yields an open subset $\mathfrak{U} \subseteq \mathfrak{U}''$ such that
	$(\mathfrak{U}_s, (\mathfrak{H} \cap \mathfrak{U})_s)$ is elementary with its minimal stratum containing $x$,
	see Proposition \ref{ClosureStrataSet}.
	Obviously we may replace $\mathfrak{U}$ by an affine open neighborhood of $x$ in $\mathfrak{U}$.  
	The resulting \et morphism $\varphi: \mathfrak{U} \to \mathfrak{S}$ satisfies 
	$\varphi^{-1}(\mathfrak{G}) = \mathfrak{H} \cap \mathfrak{U}$: 
	
	The inclusion “$\subseteq$” is clear. On the other hand, consider the \et composition 
	$\mathfrak{H} \cap \mathfrak{U} \to \mathfrak{H} \cap \mathfrak{U}' \to \mathfrak{G}(s)$.
	Every irreducible component of $(\mathfrak{H} \cap \mathfrak{U})_s$ contains $x$ by Corollary~\ref{ElementaryCor}, 
	therefore its image under $\varphi$ is contained in an irreducible component of $\mathfrak{G}(s)$, in which $\varphi(x)$ lies.
	In particular it is contained in $\{T_1 \cdots T_{s'} = 0\}$. This shows “$\supseteq$”.  
\end{proof}

\begin{definition}
	We call an open neighborhood $\mathfrak{U}$ of $x$ together with an \et morphism $\varphi$ satisfying the properties
	from Proposition \ref{BuildingBlockExist} a \emph{building block} of $(\mathfrak{X}, \mathfrak{H})$ in $x$.
	
	We will also say an \et morphism $\psi: (\mathfrak{Y}, \mathfrak{F}) \to (\mathfrak{S}, \mathfrak{G})$ to a standard scheme
	is a \emph{building block}, if $\mathfrak{Y}$ is affine, $(\mathfrak{Y}_s, \mathfrak{F}_s)$ is elementary
	and the image of the minimal element in $\str(\mathfrak{Y}_s, \mathfrak{F}_s)$ under $\psi$
	is  the minimal element in $\str(\mathfrak{S}_s, \mathfrak{G}_s)$. 
	
	If we just say $\mathfrak{U}$ is a \emph{building block} of $(\mathfrak{X}, \mathfrak{H})$, we mean
	that $\mathfrak{U}$ is an open affine subset of $\mathfrak{X}$ such that there exists an \et morphism 
	$(\mathfrak{U}, \mathfrak{H} \cap \mathfrak{U}) \to (\mathfrak{S}, \mathfrak{G})$
	being a building block in the above sense. 
	
	Moreover we say $\psi: \mathfrak{Y} \to \mathfrak{S}$ is a \emph{building block}, if
	$\psi: (\mathfrak{Y}, \emptyset) \to (\mathfrak{S}, \emptyset)$ is a building block.  
\end{definition}

\begin{noname} \label{SliceBuildingBlock}
	Let $\varphi: (\mathfrak{U}, \mathfrak{H} \cap \mathfrak{U}) \to (\mathfrak{S}, \mathfrak{G}) 
	= (\mathfrak{S}(\textbf{n}, \textbf{a}, d), \mathfrak{G}(s))$
	be a building block of $(\mathfrak{X}, \mathfrak{H})$ in $x$. We want to consider a point $y \in \mathfrak{X}$
	such that $x \in \overline{\{y\}}$. Note that $y \in \mathfrak{U}$, because otherwise we would have $y \in \mathfrak{X} \setminus \mathfrak{U}$
	and thus $\overline{\{y\}} \subseteq \mathfrak{X} \setminus \mathfrak{U}$, which contradicts $x \in \overline{\{y\}}$.
	We can apply the same construction as in Proposition~\ref{BuildingBlockExist} to obtain a building
	block in $y$ from $\mathfrak{U}$ and $\varphi$. More precisely we get an open subset $\mathfrak{U}' \subseteq \mathfrak{U}$ 
	which contains $y$, a standard pair 
	$(\mathfrak{S}', \mathfrak{G}') = (\mathfrak{S}(\textbf{n}', \textbf{a}', d+d'), \mathfrak{G}(s'))$, an open subset 
	$\mathfrak{T} = \mathfrak{S}(\textbf{n}', \textbf{a}') \times \mathfrak{B}^d \times \mathfrak{T}(d', 1)$ of $\mathfrak{S}$
	which also openly embeds into $\mathfrak{S}'$ via some torus embedding, and an \et morphism
	$\varphi': (\mathfrak{U}', \mathfrak{H} \cap \mathfrak{U}') \to (\mathfrak{S}', \mathfrak{G}')$, which is just the restriction of 
	$\varphi$ to $\mathfrak{U}'$ followed by $\mathfrak{T} \hookrightarrow \mathfrak{S}'$
	after a suitable permutation of the coordinates of $\mathfrak{B}^d$. This open subset $\mathfrak{U}'$ is together with $\varphi'$
	a building block of $(\mathfrak{X}, \mathfrak{H})$ in $y$.   
	Note that the open immersion $(\mathfrak{T}, \mathfrak{G}' \cap \mathfrak{T}) \hookrightarrow (\mathfrak{S}', \mathfrak{G}')$ 
	is a building block as well.
\end{noname}

\begin{proposition} \label{PreImIrrComp}
	Let $\mathfrak{U}$ together with $\varphi: (\mathfrak{U}, \mathfrak{H} \cap \mathfrak{U}) \to (\mathfrak{S}, \mathfrak{G})$
	be a building block of $(\mathfrak{X}, \mathfrak{H})$ in $x$. Then the following statements hold:
	\begin{enumerate}
		\item The point $x$ is contained in every irreducible component of $\mathfrak{U}_s$ and $(\mathfrak{H} \cap \mathfrak{U})_s$.
		\item The preimage of every irreducible component $\mathcal{V}$ of $\mathfrak{S}_s$ resp. $\mathfrak{G}_s$  under 
		$\varphi_s$ is an irreducible component of $\mathfrak{U}_s$ resp. $(\mathfrak{H} \cap \mathfrak{U})_s$. 
		\item If $x \in \str(\mathfrak{X}_s, \mathfrak{H}_s)$, then the minimal stratum of 
		$(\mathfrak{U}_s, (\mathfrak{H} \cap \mathfrak{U})_s)$ has generic point $x$.
	\end{enumerate}
\end{proposition}

\begin{proof}
	(i) is an immediate consequence of Corollary~\ref{ElementaryCor}.
	
	For the proof of (ii) we note that $\mathcal{V}$ is smooth, thus also $\varphi^{-1}_s(\mathcal{V})$ is smooth. 
	Furthermore every irreducible component of 
	$\varphi^{-1}_s(\mathcal{V})$ is contained in an irreducible component of $\mathfrak{U}_s$ resp. $(\mathfrak{H} \cap \mathfrak{U})_s$. 
	By dimensionality reasons $\varphi^{-1}_s(\mathcal{V})$ then has to be a union of irreducible components of $\mathfrak{U}_s$
	resp. $(\mathfrak{H} \cap \mathfrak{U})_s$.
	Since all of these contain $x$, this union has to consist of a single member. 
	
	(iii) follows from Remark~\ref{OpenStratRem}. 	
\end{proof}	

\begin{proposition}
	There is a unique number $s$ such that for all building blocks $\mathfrak{U}$ together with  
	$\varphi: (\mathfrak{U}, \mathfrak{H} \cap \mathfrak{U}) \to (\mathfrak{S}, \mathfrak{G})$
	of $(\mathfrak{X}, \mathfrak{H})$ in $x$, we have $\mathfrak{G} = \{T_1 \cdots T_s = 0\}$.
	We denote this number $s$ in the following by $s_x$. Observe that $s_x = 0$ if and only if $x \notin \mathfrak{H}$.
\end{proposition}

\begin{proof}
	Let $\mathfrak{U}$ together with  
	$\varphi: (\mathfrak{U}, \mathfrak{H} \cap \mathfrak{U}) \to (\mathfrak{S}, \mathfrak{G})$
	be a building block of $(\mathfrak{X}, \mathfrak{H})$ in $x$. We set $y := \varphi(x)$.
	Then we have bijections $\irr(\mathfrak{X}_s, x) \xrightarrow{\sim} \irr(\mathfrak{U}_s, x)
	\xrightarrow{\sim} \irr(\mathfrak{S}_s, y) = \irr(\mathfrak{S}_s)$ and
	$\irr(\mathfrak{H}_s, x) \xrightarrow{\sim} \irr((\mathfrak{H} \cap \mathfrak{U})_s, x)
	\xrightarrow{\sim} \irr(\mathfrak{G}_s, y) = \irr(\mathfrak{G}_s)$. Let $s \in \N$
	such that $\mathfrak{G} = \{T_1 \cdots T_s = 0\}$. Then $|\irr(\mathfrak{G}_s)| = s \cdot |\irr(\mathfrak{S}_s)|$
	and therefore $|\irr(\mathfrak{H}_s, x)| = s \cdot |\irr(\mathfrak{X}_s, x)|$.
	This shows that the number $s$ does only depend on $x$ and not on the choice of the building block.
\end{proof}

\begin{proposition} \label{DimensionProp}
	Let $\mathfrak{U}$ together with $\varphi: (\mathfrak{U}, \mathfrak{H} \cap \mathfrak{U}) 
	\to (\mathfrak{S}, \mathfrak{G}) = (\mathfrak{S}(\textbf{n}, \textbf{a}, d), \mathfrak{G}(s_x))$
	be a building block of $(\mathfrak{X}, \mathfrak{H})$ in $x$. We assume that $x \in \str(\mathfrak{X}_s, \mathfrak{H}_s)$
	and denote by $\zeta$ the least element in $\str(\mathfrak{S}_s, \mathfrak{G}_s)$.
	Then the codimension of the stratum corresponding to $x$ in the connected component of $\mathfrak{X}_s$ containing $x$,
	is equal to $|\textbf{n}| + s_x$. Moreover $\varphi^{-1}_s(\zeta) = \{x\}$.  
\end{proposition}

\begin{proof}
	By Proposition \ref{StrataMap} the morphism $\varphi$ maps the stratum of $(\mathfrak{U}_s, (\mathfrak{H} \cap \mathfrak{U})_s)$ 
	corresponding to $x$ to the stratum of $(\mathfrak{S}_s, \mathfrak{G}_s)$ corresponding to $\varphi(x)$.
	Since $\varphi(x)$ is contained in every irreducible component of $\mathfrak{S}_s$ and $\mathfrak{G}_s$, 
	$\varphi(x)$ is equal to the least element $\zeta$ of $\str(\mathfrak{S}_s, \mathfrak{G}_s)$  
	and the codimension of the corresponding stratum in $\mathfrak{S}_s$ is equal to $|\textbf{n}| + s_x$. 
	We know that $\mathfrak{U}_s$ is connected and equidimensional with the same dimension
	as $\mathfrak{S}_s$, which shows the first claim. The second claim $\varphi^{-1}_s(\zeta) = \{x\}$
	is clear by dimensionality reasons.      
\end{proof}	

\begin{definition}
	Let $\mathfrak{U}$ together with  
	$\varphi: (\mathfrak{U}, \mathfrak{H} \cap \mathfrak{U}) \to (\mathfrak{S}, \mathfrak{G})$
	be a building block of $(\mathfrak{X}, \mathfrak{H})$ in $x$. 
	It induces an isometric bijection $\beta: \irr((\mathfrak{H} \cap \mathfrak{U})_s, x) \xrightarrow{\sim} \irr(\mathfrak{G}_s)$.
	For all $i \in \dotsbra{1}{s_x}$ denote by $I_{x, i}$ the set of all $\mathcal{V} \in \irr(\mathfrak{H}_s, x)$
	such that $\beta(\mathcal{V} \cap \mathfrak{U}_s) \subseteq \{T_i = 0\}$. 
	We define $D_x$ to be the set $\bracket{I_{x,i}}{i \in \dotsbra{1}{s_x}}$.
	
	By Proposition \ref{RespDivCoord} the definition of $D_x$ does not depend on the choice
	of the building block, the elements of $D_x$ are disjoint subsets of $\irr(\mathfrak{H}_s, x)$
	each consisting of $|\irr(\mathfrak{X}_s, x)|$ elements and the union of all elements of $D_x$
	is equal to $\irr(\mathfrak{H}_s, x)$. 
	
	More precisely for every $I \in D_x$ and every $\mathcal{W} \in \irr(\mathfrak{X}_s, x)$
	there exists a unique  $\mathcal{V} \in I$ such that $\mathcal{V} \subseteq \mathcal{W}$.
	We denote this $\mathcal{V}$ by $\mathcal{V}_{\mathcal{W}, I}$. 
	Furthermore we will denote by $\mathcal{V}_I$ the closed subset of $\mathfrak{H}_s$
	which is the union of all elements of $I$. 
\end{definition}

\begin{theorem} \label{BijThm}
	Let $\mathfrak{U}$ together with  
	$\varphi: (\mathfrak{U}, \mathfrak{H} \cap \mathfrak{U}) \to (\mathfrak{S}, \mathfrak{G})$
	be a building block of $(\mathfrak{X}, \mathfrak{H})$ in $x$.  
	Then there exists a unique bijection $\gamma_{\varphi}: \dotsbra{1}{s_x} \xrightarrow{\sim} D_x$
	such that for all $i \in \dotsbra{1}{s_x}$ we have
	\[\mathcal{V}_{\gamma_{\varphi}(i)} \cap \mathfrak{U}_s = \varphi^{-1}_s(\{T_i = 0\}).\]
	Consider the isometric bijection 
	$\alpha: \irr(\mathfrak{S}_s) \xrightarrow{\sim} \irr(\mathfrak{U}_s, x) \xrightarrow{\sim} \irr(\mathfrak{X}_s, x)$
	induced by $\varphi$ and the open immersion $\mathfrak{U} \hookrightarrow \mathfrak{X}$.
	Then for all $i \in \dotsbra{1}{s_x}$ and $\mathcal{T} \in \irr(\mathfrak{S}_s)$ we have
	\[\mathcal{V}_{\alpha(\mathcal{T}), \gamma_{\varphi}(i)} \cap \mathfrak{U}_s = \varphi^{-1}_s(\mathcal{T} \cap \{T_i = 0\}).\]
\end{theorem}

\begin{proof}
	$\varphi$ induces an isometric bijection $\beta: \irr((\mathfrak{H} \cap \mathfrak{U})_s, x) \xrightarrow{\sim} \irr(\mathfrak{G}_s)$.
	For all $i \in \dotsbra{1}{s_x}$ denote by $I_{x, i}$ the set of all $\mathcal{V} \in \irr(\mathfrak{H}_s, x)$
	such that $\beta(\mathcal{V} \cap \mathfrak{U}_s) \subseteq \{T_i = 0\}$.
	We claim that $\gamma_{\varphi}$ defined by sending $i$ to $I_{x, i}$ for all $i \in \dotsbra{1}{s_x}$
	satisfies the desired property. The inclusion “$\subseteq$” is clear from the definitions.
	So let us choose a point $y \in \mathfrak{U}_s$ with $\varphi(y) \in \{T_i = 0\}$.
	We want to show that there exists an element $\mathcal{V} \in I_{x, i}$ such that $y \in \mathcal{V}$. 
	Note that $\varphi$ also induces an isometric bijection 
	$\irr((\mathfrak{H} \cap \mathfrak{U})_s, y) \xrightarrow{\sim} \irr(\mathfrak{G}_s, \varphi(y))$.
	Let $\mathcal{W} \in \irr(\mathfrak{G}_s, \varphi(y))$ with $\mathcal{W} \subseteq \{T_i = 0\}$.
	Then the preimage of $\mathcal{W}$ under this bijection is an element of $\irr((\mathfrak{H} \cap \mathfrak{U})_s, y)$,
	which contains $x$ since $\mathfrak{U}$ together with $\varphi$ is a building block,
	and we can choose $\mathcal{V}$ to be the corresponding element in $\irr(\mathfrak{H}_s, y)$.
	This proves the inclusion “$\supseteq$”.
	
	The uniqueness of $\gamma_{\varphi}$ is clear, since the $\mathcal{V}_I \cap \mathfrak{U}_s$
	for all $I \in D_x$ are distinct to one another.
	
	The second claim follows immediately from Proposition \ref{PreImIrrComp} (ii).
\end{proof}

\begin{proposition} \label{DivInjection}
	Let $y \in \mathfrak{X}$ with $x \in \overline{\{y\}}$. Then the inclusion 
	$\irr(\mathfrak{H}_s, y) \hookrightarrow \irr(\mathfrak{H}_s, x)$ induces an injective map
	$j_{y, x}: D_y \hookrightarrow D_x$ such that for all $I \in D_y$ we have 
	$I \subseteq j_{y, x}(I)$. If $z \in \mathfrak{X}$ with $y \in \overline{\{z\}}$ then in particular
	$x \in \overline{\{z\}}$ and $j_{z, x} = j_{y, x} \circ j_{z, y}$.
	
	Moreover let us denote $D_{x, y} := \{I \in D_x ~|~ I \cap \irr(\mathfrak{H}_s, y) \neq \emptyset\}$.
	Then there is a bijection $k_{x, y}: D_{x, y} \xrightarrow{\sim} D_y$ 
	given by $k_{x, y}(I) := I \cap \irr(\mathfrak{H}_s, y)$ for all $I \in D_{x, y}$.
\end{proposition}

\begin{proof}
	These are immediate consequences of Proposition \ref{RespDivCoord}. 
\end{proof}

\begin{proposition} \label{EtaleMorProp}
	Let $\psi: \mathfrak{Y} \to \mathfrak{X}$ be an \et morphism of formal schemes.
	Consider the pullback $\mathfrak{G} := \psi^{-1}(\mathfrak{H})$ as a closed subscheme of $\mathfrak{Y}$.
	Let $y \in \mathfrak{Y}$ and $x := \psi(y)$. Then the following statements hold:
	\begin{enumerate}
		\item $(\mathfrak{Y}, \mathfrak{G})$ is a strictly poly-stable pair.
		\item $\psi$ induces isometric bijections $\irr(\mathfrak{Y}_s, y) \xrightarrow{\sim} \irr(\mathfrak{X}_s, x)$   
		and $\beta: \irr(\mathfrak{G}_s, y) \xrightarrow{\sim} \irr(\mathfrak{H}_s, x)$.
		\item $\psi$ induces a bijection $\delta: D_y \xrightarrow{\sim} D_x$ such that for all 
		$I \in D_{y}$ we have $\delta(I) = \bracket{\beta(\mathcal{V})}{\mathcal{V} \in I} \in D_x$.   
	\end{enumerate}
\end{proposition}

\begin{proof}
	(i) follows straightforward from the definition and (ii) is just an application of Proposition \ref{IrrIsometric}.
	
	In order to show (iii) let $\mathfrak{U}$ together with  
	$\varphi: (\mathfrak{U}, \mathfrak{H} \cap \mathfrak{U}) \to (\mathfrak{S}, \mathfrak{F})$
	be a building block of $(\mathfrak{X}, \mathfrak{H})$ in $x$. We define $\mathfrak{U}'$ to be $\psi^{-1}(\mathfrak{U})$
	without $\overline{\{z\}}$ for all $z \in \str(\mathfrak{Y}_s, \mathfrak{G}_s)$ such that $y \notin \overline{\{z\}}$.
	Then $\mathfrak{U}'$ together with 
	$\varphi' := \varphi \circ \psi: (\mathfrak{U}', \mathfrak{G} \cap \mathfrak{U}') \to (\mathfrak{S}, \mathfrak{F})$
	is a building block of $(\mathfrak{Y}, \mathfrak{G})$ in $y$. Now the claim is evident from the definition
	of $D_y$ resp. $D_x$.
\end{proof}


\section{Dual intersection complexes} \label{DualIntComplexChap} 

The goal of this section is to define and explore the “dual intersection complex” of a poly-stable pair.
It is a geometric object which is combinatorial in its nature, in the sense that it encodes the
stratification behavior of our pair. We will need it in order to explain the piecewise linear structure of 
the extended skeletons, which will be considered later.

Our exposition relies on the ideas from \cite[§§\,3--4]{Ber99} and leads us to the language of extended poly-simplices.
However I decided to not include a theory of poly-simplicial sets, but instead lay my focus on the geometric objects.

I adapted the term “canonical polyhedron” from \cite[4.6]{GRW16} to emphasize that the extended poly-simplices
obtained from the building blocks are canonical up to “permutation of coordinates”, which also gets a concrete meaning
due to the language from §\,\ref{PolySimplexSec}. In the strictly poly-stable case, the dual intersection complex
is then obtained by suitably gluing together the canonical polyhedra along faces.
The arbitrary poly-stable case is then constructed from the strict case by allowing to identify even more faces.


\subsection{Poly-simplices} \label{PolySimplexSec}

We will have a look at simplices, poly-simplices, colored poly-simplices, extended poly-simplices
and morphisms between them. These objects admit geometric realizations,
which later will constitute the pieces our dual intersection complex is made up of,
at least in the strict case.

\begin{definition}
	Let $n \in \N$. A \emph{simplex} is a set $[n] := \{0,\dots,n\}$.
	
	Let $p \in \Nn$ and $\mathbf{n} = (n_1,\dots,n_p)$ be a $p$-tuple with $n_1,\dots,n_p \in \Nn$.
	We also allow the case $p = 0$ for which we write $\mathbf{n} := (0)$.  
	A \emph{poly-simplex} is a set $[\mathbf{n}] := [n_1]\times \cdots \times [n_p]$.
	We define a metric on this set by assigning to each pair of $p$-tuples in $[\mathbf{n}]$ 
	the number of components in which they do not agree.
	
	Let $[\mathbf{n}']$ be a poly-simplex with $\mathbf{n}'$ being a $p'$-tuple. A \emph{morphism} from $[\mathbf{n}]$ to $[\mathbf{n}']$
	consists of the following data: A map $c: [\textbf{n}] \rightarrow [\textbf{n}']$,
	a subset $J \subseteq \{1,\dots,p\}$, 
	an injective map $f: J \rightarrow \{1,\dots,p'\}$ and a
	family $(c_l)_{l \in \{1,\dots,p'\}}$ of maps satisfying these conditions:
	Let $l \in \{1,\dots,p'\}$. If $l \in \im(f)$ then $c_l$ is an injective map $[n_{f^{-1}(l)}] \rightarrow [n'_l]$ 
	and if $l \notin \im(f)$ then $c_l$ is a map $[0] \rightarrow [n'_l]$. The image $(i_1, \dots , i_{p'}) \in [\textbf{n}']$ 
	of an element $(j_1, \dots , j_p) \in [\textbf{n}]$ under $c$ is given by
	\[ i_l = \left\{ 
	\begin{array}{cl}
	c_l(j_{f^{-1}(l)}) & \text{, if } l \in \im(f) \\
	c_l(0) & \text{, else}  	
	\end{array}
	\right. \]  
	for all $l \in \dotsbra{1}{p'}$. In the case $\mathbf{n} = (0)$ we require $J = \emptyset$.  
	Note that if $\mathbf{n} \neq (0)$ and $c$ is injective, then it follows that $J = \{1,\dots,p\}$.  
	
	If $\mathbf{n} \neq (0)$ let $\mathbf{r} := (r_1,\dots,r_p)$ with $r_i \in \overline{\R}_{>0} := \R_{>0} \cup \{\infty\}$ 
	for all $i \in \dotsbra{1}{p}$. 
	In the case $\mathbf{n} = (0)$ we require $\mathbf{r} = (0)$.
	A \emph{colored poly-simplex} is a pair $[\mathbf{n}, \mathbf{r}] := (\mathbf{n}, \mathbf{r})$.
	
	Let $[\mathbf{n}', \mathbf{r}']$ be a colored poly-simplex. 
	A \emph{morphism} from $[\mathbf{n}, \mathbf{r}]$ to $[\mathbf{n}', \mathbf{r}']$ is defined by the same data 
	$(c, J, f, (c_l)_{l \in \{1,\dots,p'\}})$ as a morphism from $[\mathbf{n}]$ to $[\mathbf{n}']$ 
	with the additional condition that $r_i = r'_{f(i)}$ for all $i \in J$.
	Note that this condition is trivially satisfied if $p=0$.
	The morphism is called \emph{injective}, if $c$ is injective.
	
	Let $s \in \N$. 
	An \emph{extended poly-simplex} is a triple $[\mathbf{n}, \mathbf{r}, s] := (\mathbf{n}, \mathbf{r}, s)$,
	such that $[\mathbf{n}, \mathbf{r}]$ is a colored poly-simplex.
	
	Let $[\mathbf{n}', \mathbf{r}', s']$ be an extended poly-simplex. 
	A \emph{morphism} from $[\mathbf{n}, \mathbf{r}, s]$ to $[\mathbf{n}', \mathbf{r}', s']$
	consists of a morphism from $[\mathbf{n}, \mathbf{r}]$ to $[\mathbf{n}', \mathbf{r}']$ 
	and a map $g: \dotsbra{0}{s} \to \dotsbra{0}{s'}$ with $g(0) = 0$ 
	and such that every element in $\dotsbra{1}{s'}$ has at most one preimage under $g$.
	It is called \emph{injective}, if the morphism from $[\mathbf{n}, \mathbf{r}]$ to $[\mathbf{n}', \mathbf{r}']$ is injective
	and $g$ is injective. If $g$ is injective, we will often just consider its restriction $\dotsbra{1}{s} \to \dotsbra{1}{s'}$. 
	
	One defines in the obvious way the category of extended poly-simplices. 
\end{definition}

\begin{proposition} \label{IsomorphIsometric}
	Let $[\mathbf{n}]$ and $[\mathbf{m}]$ be two poly-simplices with $p$-tuple
	$\mathbf{n}$ and with $q$-tuple $\mathbf{m}$. 
	Then a map $[\mathbf{m}] \to [\mathbf{n}]$ is isometric if and only if it is an injective morphism of poly-simplices.
	If an injective morphism $[\mathbf{m}] \to [\mathbf{n}]$ exists, then $q \leq p$. 
	In particular an isomorphism between $[\mathbf{m}]$ and $[\mathbf{n}]$
	is the same as an isometric bijection. If $[\mathbf{m}]$ is isomorphic to $[\mathbf{n}]$,
	then $p=q$.   
\end{proposition}

\begin{proof}
	This follows from \cite[Lemma 3.1]{Ber90} and the definitions. 
\end{proof}

\begin{example} \label{MetrEx}
	Let $\mathfrak{S} := \mathfrak{S}(\textbf{n},\textbf{a},d)$ be a standard scheme. There is a canonical
	bijection $[\textbf{n}] \xrightarrow{\sim} \irr(\mathfrak{S}_s)$. Explicitly the tuple $(k_1, \dots, k_p) \in [\textbf{n}]$ 
	corresponds to the irreducible component $\{T_{1k_1}=0, \dots, T_{pk_p}=0\}$. Consider the generic point $\eta$
	of the intersection of all irreducible components of $\mathfrak{S}_s$, which we can imagine as $\{\mathbf{0}\} \times \mathfrak{B}^d_s$. 
	If we enhance $\irr(\mathfrak{S}_s) = \irr(\mathfrak{S}_s, \eta)$ with the metric from Definition \ref{IrrMetric},
	then the above map becomes an isometric bijection.
\end{example}

\begin{example}
	Consider a standard scheme $\mathfrak{S} := \mathfrak{S}(\textbf{n},\textbf{a},d)$ and an isomorphism
	from $[\mathbf{n}, \mathbf{r}, d]$ to $[\mathbf{n}', \mathbf{r}', d]$ of extended poly-simplices
	given by the data $c: [\mathbf{n}] \xrightarrow{\sim} [\mathbf{n}']$, $f: \dotsbra{1}{p} \xrightarrow{\sim} \dotsbra{1}{p}$, 
	$(c_l)_{l \in \{1,\dots,p\}}$ and $g: \dotsbra{0}{d} \xrightarrow{\sim} \dotsbra{0}{d}$,
	where $\mathbf{r} := \val(\textbf{a})$. Then we obtain an isomorphism 
	$\mathfrak{S}' := \mathfrak{S}(\textbf{n}',\textbf{a}',d) \xrightarrow{\sim} \mathfrak{S}(\textbf{n},\textbf{a},d)$ 
	of formal schemes, where $\textbf{a}' = (a_{f^{-1}(1)}, \dots, a_{f^{-1}(p)})$. 
	It is induced by $T_{ij} \mapsto T'_{f(i), c_{f(i)}(j)}$ for all
	$i \in \dotsbra{1}{p}$, $j \in \dotsbra{0}{n_i}$ and $T_k \mapsto T'_g(k)$ for all
	$k \in \dotsbra{1}{d}$. With the canonical map from Example \ref{MetrEx}
	we have the following commutative diagram with isometric bijections:
	\begin{center}\begin{tikzpicture}[scale = 2]
		\node (1) at (0, 1) {$\irr(\mathfrak{S}'_s)$};
		\node (2) at (1, 1) {$\irr(\mathfrak{S}_s)$};
		\node (3) at (0, 0) {$[\mathbf{n}']$};
		\node (4) at (1, 0) {$[\mathbf{n}]$};	
		
		\draw[->, thick] (1)--(2) node[pos=0.5, above] {$\sim$};
		\draw[->, thick] (1)--(3) node[pos=0.5, sloped, below] {$\sim$};
		\draw[->, thick] (4)--(3) node[pos=0.5, above] {$\sim$} node[pos=0.5, below] {$c$};
		\draw[->, thick] (2)--(4) node[pos=0.5, sloped, above] {$\sim$};
		\end{tikzpicture}\end{center}
\end{example}

\begin{definition} \label{GeomSimplex}
	Let $n \in \N$ and $r \in \R_{\geq 0}$.
	Then we define the \emph{geometric simplex} as 
	\[\Delta(n, r) := \bracket{(x_0,\dots,x_n) \in \R^{n+1}_{\geq 0}}{x_0 + \cdots + x_n = r}, \]
	which is endowed with the subspace topology of $\R^{n+1}$. 
	We also define some degenerate geometric simplex at infinity, namely
	\[\Delta(n, \infty) := \bracket{(x_0,\dots,x_n) \in \overline{\R}^{n+1}_{\geq 0}}{x_i = \infty \text{ for at least one }i},\]
	where $\overline{\R}_{\geq 0} := \R_{\geq 0} \cup \{\infty\}$ is equipped with the obvious topology. 
	A \emph{vertex} of a geometric simplex is a point where at most one coordinate is distinct from $0$. 
	
	Let now $p, s \in \N$ and $\mathbf{n} = (n_1, \dots, n_p)$, $\mathbf{r} = (r_1, \dots, r_p)$ with
	$n_i \in \N$ and $r_i \in \overline{\R}_{\geq 0}$ for all $i \in \dotsbra{1}{p}$.
	In the case $p=0$ we set $\mathbf{n} = (0)$ and $\mathbf{r} = (0)$. 
	Then we consider the topological space 
	\[\Delta(\mathbf{n}, \mathbf{r}, s) := \prod\limits_{i=1}^p \Delta(n_i, r_i) \times \R_{\geq 0}^s \]
	endowed with the product topology. It is called a \emph{geometric extended poly-simplex}.
	We write $\Delta(\mathbf{n}, \mathbf{r}) := \Delta(\mathbf{n}, \mathbf{r}, 0)$.
	A \emph{vertex} of a geometric extended poly-simplex is a point whose projection
	to each geometric simplex factor is a vertex and every coordinate of the projection to the factor 
	$\R_{\geq 0}^s$ is $0$. 
	
	We also define \emph{open} versions of the above by setting
	\begin{align*}
	\Delta^{\circ}(n, r) &:=\left\{ 
	\begin{array}{ll}
	\Delta(n, 0), & \text{, if~} r=0 \\
	\bracket{(x_0,\dots,x_n) \in \Delta(n, r)}{x_i \neq 0 \text{ for all }i}, & \text{, if~} r \in \overline{\R}_{> 0},  	
	\end{array}
	\right. \\ 
	\Delta^{\circ}(\mathbf{n}, \mathbf{r}, s) &:= \prod\limits_{i=1}^p \Delta(n_i, r_i)^{\circ} \times \R_{> 0}^s.
	\end{align*}
\end{definition}

\begin{remark}
	The degenerate simplices $\Delta(n, \infty)$ look somewhat different than their finite counterparts,
	however they still share the same structure in the sense, that they can be decomposed into their
	faces, which are simplices of lower dimension.
\end{remark}

\begin{definition}
	Let $[\mathbf{n}', \mathbf{r}', s']$ be an extended poly-simplex with $p'$-tuple
	$\mathbf{n}'$ and $F$ a morphism from $[\mathbf{n}, \mathbf{r}, s]$ to $[\mathbf{n}', \mathbf{r}', s']$
	given by the data $(c, J, f, (c_l)_{l \in \{1,\dots,p'\}})$ and $g: \dotsbra{0}{s} \to \dotsbra{0}{s'}$. 
	In the following we want to establish a functor from the category of extended poly-simplices to the category of
	topological spaces:
	
	On objects we define the \emph{geometric realization} of $[\mathbf{n}, \mathbf{r}, s]$ as the topological space
	$\Delta(\mathbf{n}, \mathbf{r}, s)$.
	
	On morphisms define the \emph{geometric realization} of $F$ as the continuous map 
	$\Delta(F): \Delta(\mathbf{n}, \mathbf{r}, s) \to \Delta(\mathbf{n}', \mathbf{r}', s')$
	defined as follows: First we note that every point in $\Delta(\mathbf{n}, \mathbf{r}, s)$
	can be represented by two tuples, namely $(x_{ij})_{i \in \{1,\dots,p\}, j \in \{0,\dots,n_i\}}$ 
	and $(y_j)_{j \in \{1,\dots,s\}}$ with $x_{ij} \in \overline{\R}_{\geq 0}$, $y_j \in \R_{\geq 0}$
	such that $\sum^{n_i}_{j=0} x_{ij} = r_i$ for all $i \in \{1,\dots,p\}$. Similarly for $\Delta(\mathbf{n}', \mathbf{r}', s')$.
	Then the image of the point associated with $(x_{ij})$ and $(y_j)$ under $\Delta(F)$
	is the point associated with $(u_{ij})$ and $(v_j)$, 
	where the components are defined like this: \newline
	For all $i \in \{1,\dots,p'\}, j \in \{0,\dots,n'_i\}$ 
	\begin{itemize}[leftmargin=*]
		\item if $i \notin \im(f)$, we set 
		\[ u_{ij} = \left\{ 
		\begin{array}{cl}
		r'_i & \text{, if } j = c_i(0) \\
		0 & \text{, else,}  	
		\end{array}
		\right. \] 
		\item if $i \in \im(f)$, we set 
		\[ u_{ij} = \left\{ 
		\begin{array}{cl}
		x_{f^{-1}(i), c_i^{-1}(j)} & \text{, if } j \in \im(c_i) \\
		0 & \text{, else.}  	
		\end{array}
		\right. \]
	\end{itemize}	
	For all $j \in \{1,\dots,s'\}$	
	\begin{itemize}[leftmargin=*]		 
		\item if $j \notin \im(g)$, we set $v_j = 0$,   
		\item if $j \in \im(g)$, we set $v_j = y_{g^{-1}(j)}$.
	\end{itemize}	
	This defines the \emph{geometric realization functor} $\Delta$ from the category of extended poly-simplices to the 
	category of topological spaces.
	
	Note that if $F$ is an isomorphism, then $\Delta(F)$ restricts to a homeomorphism 
	$\Delta^{\circ}(F): \Delta^{\circ}(\mathbf{n}, \mathbf{r}, s) \to \Delta^{\circ}(\mathbf{n}', \mathbf{r}', s')$.
	Thus we also get a functor $\Delta^{\circ}$ from the category of extended poly-simplices with isomorphisms between them 
	to the category of topological spaces with homeomorphisms between them.	
\end{definition}

\begin{definition} \label{AffLinFct}
	We enhance a geometric extended poly-simplex $\Delta(\mathbf{n}, \mathbf{r}, s)$ 
	with a sheaf of monoids consisting of all functions which locally are restrictions of functions of the form 
	\[(x_{10},\dots,x_{pn_p},y_1,\dots,y_s) \mapsto \lambda + \sum\limits_{i=1}^p \sum\limits_{j=0}^{n_i} a_{ij} x_{ij} + \sum\limits_{i=1}^s b_i y_i \] 
	with $\lambda \in \Gamma \cap \R_{\geq 0}$, $a_{ij} \in \N$ and $b_i \in \N$. We call it the \emph{sheaf of locally affine linear functions}.  
	
	One easily sees that the geometric realization functor $\Delta$ defined above respects this sheaf in the sense, that the pullback of 
	a locally affine linear function with respect to the geometric realization of a morphism of extended poly-simplices
	yields again a locally affine linear function. 
\end{definition}


\subsection{Canonical polyhedra} \label{CanPolySubSec}

In the upcoming two subsections 
we consider the case of a strictly poly-stable pair $(\mathfrak{X}, \mathfrak{H})$.
We denote $\mathcal{X} := \mathfrak{X}_s$, $\mathcal{H} := \mathfrak{H}_s$.
Let for the rest of this subsection be $x \in \str(\mathcal{X}, \mathcal{H})$.

We want to introduce “combinatorial charts” associated to building blocks
and use them to identify geometric extended poly-simplices. In this way we obtain for
every stratum a “canonical polyhedron”.
Depending on how the closures of the strata contain one another, we also receive
“face embeddings” between them.

\begin{noname} \label{CombChartConstr}
	As seen in Proposition \ref{BuildingBlockExist}, there exists
	an affine open neighborhood $\mathfrak{U}$ of $x$, a standard pair 
	$(\mathfrak{S}, \mathfrak{G}) = (\mathfrak{S}(\textbf{n},\textbf{a},d), \mathfrak{G}(s_x))$
	and an \et morphism $\varphi: (\mathfrak{U}, \mathfrak{H} \cap \mathfrak{U}) \to (\mathfrak{S}, \mathfrak{G})$ 
	forming a building block of $(\mathfrak{X}, \mathfrak{H})$ in $x$.
	We denote $\mathcal{U} := \mathfrak{U}_s$, $\mathcal{S} := \mathfrak{S}_s$ and $\mathcal{G} := \mathfrak{G}_s$. 
	
	Proposition \ref{IrrIsometric} implies that the open immersion $\mathcal{U} \to \mathcal{X}$ and the \et morphism $\varphi_s$ induce isometric bijections 
	$\irr(\mathcal{U}) = \irr(\mathcal{U}, x) \xrightarrow{\sim} \irr(\mathcal{X}, x) $
	and $\irr(\mathcal{U}, x) \xrightarrow{\sim} \irr(\mathcal{S}, \varphi(x)) = \irr(\mathcal{S})$.
	Furthermore there is a canonical isometric bijection $\irr(\mathcal{S}) \xrightarrow{\sim} [\textbf{n}]$, see Example~\ref{MetrEx}.
	These induce an isometric bijection $[\textbf{n}] \xrightarrow{\sim} \irr(\mathcal{X}, x)$, which we will denote by $\alpha_{\varphi}$.
	
	Next we recall that our building block also induces a bijection $\gamma_{\varphi}: \dotsbra{1}{s_x} \xrightarrow{\sim} D_x$, 
	see Theorem \ref{BijThm}.
	
	We collect this data to a triple $C_{\varphi} := ([\mathbf{n}, \mathbf{r}, s], \alpha_{\varphi}, \gamma_{\varphi})$,
	where $\mathbf{r} := \val(\textbf{a})$. 
\end{noname}

\begin{remark} \label{StdChangeRem}
	With the same notation as above, let $[\textbf{n}']$ be a poly-simplex and 
	$\alpha': [\textbf{n}'] \xrightarrow{\sim} \irr(\mathcal{X}, x)$ an isometric bijection.
	Then Example \ref{MetrEx} applied to $\alpha^{-1}_{\varphi} \circ \alpha'$ shows that
	there is a building block $\mathfrak{U}$ together with $\varphi'$ of 
	$(\mathfrak{X}, \mathfrak{H})$ in $x$ such that $\alpha_{\varphi'} = \alpha'$.
\end{remark}

\begin{proposition} \label{ColorProp}
	Let $[\textbf{n}]$ be a poly-simplex with $\textbf{n} = (n_1, \dots , n_p)$ 
	and $\alpha: [\textbf{n}] \xrightarrow{\sim} \irr(\mathcal{X}, x)$
	an isometric bijection. Then there exists a tuple $\mathbf{r} = (r_1, \dots , r_p)$ of elements in $\overline{\R}_{>0}$
	such that for every building block 
	$\varphi: (\mathfrak{U}, \emptyset) \to (\mathfrak{S}(\textbf{n},\textbf{a},d), \emptyset)$ 
	of $(\mathfrak{X}, \emptyset)$ in $x$ with $\alpha_{\varphi} = \alpha$ we have $\mathbf{r} = \val(\textbf{a})$.
	We denote this tuple by $\mathbf{r}_{\alpha}$.   
\end{proposition}

\begin{proof}
	This is essentially \cite[Proposition 4.3]{Ber99}. 
\end{proof}

\begin{proposition} \label{ColorChangeProp}
	In the same situation as in Proposition \ref{ColorProp}, let $c: [\textbf{n}'] \hookrightarrow [\textbf{n}]$ be an injective morphism
	of poly-simplices with data $(c, J, f, (c_i)_{i \in \{1,\dots,p\}})$ and $\textbf{n}' = (n'_1, \dots , n'_{p'})$.  
	We assume that there is an element $y \in \str(\mathcal{X}, \mathcal{H})$ with $x \leq y$
	such that the image of $\alpha \circ c$ is equal to $\irr(\mathcal{X}, y)$.
	Then $\alpha \circ c: [\textbf{n}'] \xrightarrow{\sim} \irr(\mathcal{X}, y)$ is an isometric bijection.
	If $\textbf{n}' \neq (0)$ we get $\mathbf{r}_{\alpha \circ c} = (r_{f(1)}, \dots, r_{f(p')})$
	and otherwise $\mathbf{r}_{\alpha \circ c} = (0)$.
\end{proposition}

\begin{proof}
	The case $\textbf{n}' = (0)$ is trivial. Therefore we may assume now that $\textbf{n}' \neq (0)$. 
	By~Remark~\ref{StdChangeRem} there exists a building block $\mathfrak{U}$ together with 
	$\varphi: (\mathfrak{U}, \emptyset) \to (\mathfrak{S}(\textbf{n},\textbf{a},d), \emptyset)$
	of $(\mathfrak{U}, \emptyset)$ in $x$ such that $\alpha_{\varphi} = \alpha$. 
	Note that $y \in \mathfrak{U}$, because otherwise we would have $y  \in \mathfrak{X} \setminus \mathfrak{U}$
	and thus 
	$\overline{\{y\}} \subseteq \mathfrak{X} \setminus \mathfrak{U}$, which contradicts $x \in \overline{\{y\}}$. 
	We now apply the same procedure as in the proof of Proposition \ref{BuildingBlockExist} to obtain a 
	suitable building block in $y$:
	
	We remove from $\mathfrak{S}(\textbf{n}, \textbf{a}, d)$ all closed subsets $\{T_{ij} = 0\}$ with
	$i \in \dotsbra{1}{p}$, $j \in \dotsbra{0}{n_i}$ such that $\varphi(y) \notin \{T_{ij} = 0\}$.
	By our assumptions $\varphi(y) \in \{T_{ij} = 0\}$ is equivalent to $i \in \im(f)$ and $j \in \im(c_i)$.
	We point to Proposition \ref{PreImIrrComp} here. 
	Therefore the result is an open subset of $\mathfrak{S}(\textbf{n}, \textbf{a}, d)$ which
	is of the form $\mathfrak{T} := \mathfrak{S}(\textbf{n}', \textbf{a}', d) \times \mathfrak{T}(d', 1)$
	with $\textbf{a}' = (a_{f(1)}, \dots, a_{f(p')})$. 
	Let $\mathfrak{U}' := \varphi^{-1}(\mathfrak{T})$. This is an open neighborhood of $y$.
	The torus $\mathfrak{T}(d', 1)$ can be openly embedded into a ball.
	We get an \et morphism $\varphi: \mathfrak{U}' \to \mathfrak{S}$ to a standard scheme 
	$\mathfrak{S}(\textbf{n}', \textbf{a}', d + d')$,
	whose irreducible components each contain $\varphi(y)$.
	
	Finally we remove from $\mathfrak{U}'$ the closed subsets $\overline{\{z\}}$ for all $z \in \str(\mathcal{X})$
	with $y \notin \overline{\{z\}}$. This yields an open subset $\mathfrak{U}'' \subseteq \mathfrak{U}'$ such that
	$\mathfrak{U}''_s$ is elementary with its minimal stratum containing $y$.
	Obviously we may replace $\mathfrak{U}''$ by an affine open neighborhood of $y$ in $\mathfrak{U}''$. 
	Then $\mathfrak{U}''$ together with 
	$\varphi' := \varphi|_{\mathfrak{U}''}: (\mathfrak{U}'', \emptyset) \to (\mathfrak{S}(\textbf{n}', \textbf{a}', d + d'), \emptyset)$
	is a building block of $(\mathfrak{X}, \emptyset)$ in $y$. It satisfies $\alpha_{\varphi'} = \alpha \circ c$
	and the claim is evident.
\end{proof}

\begin{definition}
	A \emph{combinatorial chart} in $x$ is a triple $C := ([\mathbf{n}, \mathbf{r}, s_x], \alpha, \gamma)$ 
	consisting of an extended poly-simplex $[\mathbf{n}, \mathbf{r}, s_x]$,
	an isometric bijection $\alpha: [\textbf{n}] \xrightarrow{\sim} \irr(\mathcal{X}, x)$ 
	and a bijection $\gamma: \dotsbra{1}{s_x} \xrightarrow{\sim} D_x$ such that $\mathbf{r} = \mathbf{r}_{\alpha}$.	 
\end{definition}

\begin{proposition} \label{ChartIsom}
	Let $C := ([\mathbf{n}, \mathbf{r}, s_x], \alpha, \gamma)$ and $C' := ([\mathbf{n}', \mathbf{r}', s_x], \alpha', \gamma')$
	be combinatorial charts in $x$. Then $(\alpha')^{-1} \circ \alpha$ and $(\gamma')^{-1} \circ \gamma$
	induce an isomorphism $[\mathbf{n}, \mathbf{r}, s_x] \xrightarrow{\sim} [\mathbf{n}', \mathbf{r}', s_x]$
	of extended poly-simplices, which we will denote by $h_{C, C'}$. 
	It holds that $h_{C, C} = \id$ and $h_{C', C} = h^{-1}_{C, C'}$. 
	If $C''$ is another combinatorial chart in $x$, then $h_{C, C''} = h_{C', C''} \circ h_{C, C'}$.  
\end{proposition}

\begin{proof}
	We get an isometric bijection $c := (\alpha')^{-1} \circ \alpha: [\mathbf{n}] \xrightarrow{\sim} [\mathbf{n}']$ and according to
	Proposition~\ref{IsomorphIsometric} it is an isomorphism of poly-simplices. 
	Moreover $(\gamma')^{-1} \circ \gamma: \dotsbra{1}{s_x} \xrightarrow{\sim} \dotsbra{1}{s_x}$ is a bijection. 
	Applying Proposition~\ref{ColorChangeProp} tells us, that $c$ is actually an isomorphism
	$[\mathbf{n}, \mathbf{r}] \xrightarrow{\sim} [\mathbf{n}', \mathbf{r}']$ of colored poly-simplices. 
	Consequently we obtain the desired isomorphisms $h_{C, C'}$ and their stated properties are obvious from the construction. 
\end{proof}

\begin{example}
	Our construction from \ref{CombChartConstr} shows that combinatorial charts exist. In fact it is an easy consequence
	of Example \ref{MetrEx} and Proposition \ref{ChartIsom} that every combinatorial chart in $x$ is of the form $C_{\varphi}$ 
	for a suitable building block $\mathfrak{U}$ together with $\varphi$ of $(\mathfrak{X}, \mathfrak{H})$ in $x$. 
\end{example}

\begin{definition} \label{CanPolyDef}
	Let $\mathsf{Chart}_x$ be the category whose objects are the combinatorial charts in $x$.
	The set of morphisms between two combinatorial charts $C$ and $C'$ in $x$ consists only of the isomorphism $h_{C, C'}$
	constructed above. One easily confirms that this in fact defines a category in the obvious way
	and we have a forgetful functor $([\mathbf{n}, \mathbf{r}, s_x], \alpha, \gamma) \mapsto [\mathbf{n}, \mathbf{r}, s_x]$ 
	from $\mathsf{Chart}_x$ to the category of extended poly-simplices with isomorphisms between them. 
	Let $F_x$ be the composition of this forgetful functor with the geometric realization functor~$\Delta$. 
	We define $\Delta(x)$ to be the colimit of the functor $F_x$ 
	and we call $\Delta(x)$ the \emph{canonical polyhedron} of $x$.
	
	If instead of $\Delta$ we consider the functor $\Delta^{\circ}$, we get the \emph{open canonical polyhedron} 
	$\Delta^{\circ}(x)$ of $x$, which canonically embeds into $\Delta(x)$ as a subspace.
	
	We note that every combinatorial chart $([\mathbf{n}, \mathbf{r}, s_x], \alpha, \gamma)$ in $x$
	gives rise to a homeomorphism $\Delta(\mathbf{n}, \mathbf{r}, s_x) \xrightarrow{\sim} \Delta(x)$.
	This motivates us to define the \emph{dimension} of $\Delta(x)$ resp. $\Delta^{\circ}(x)$ to be $|\mathbf{n}| + s_x$,
	which is well-defined by the above considerations. From Proposition \ref{DimensionProp} we obtain that for all $d \in \N$ 
	the $d$-dimensional canonical polyhedra correspond to the strata of codimension $d$.  
\end{definition} 

\begin{example}
	Let $\mathfrak{X} := \mathfrak{S}(2, a)$ with $a \in K^{\circ \circ} \setminus \{0\}$. 
	Then $\mathfrak{X}_s = \widetilde{K}[T_0, T_1, T_2]/(T_0T_1T_2)$, see Example \ref{StratExample}.
	We choose $x$ to be the minimal stratum of $\mathfrak{X}_s$, i.\,e. the origin.
	The canonical polyhedron $\Delta(x)$ can then be visualized as follows. 
	\begin{center}\includegraphics[scale=1]{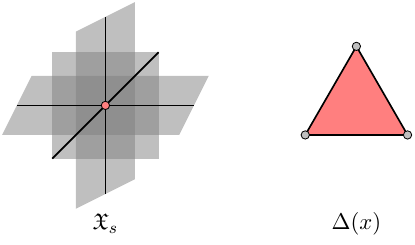}\end{center} 
\end{example}

\begin{example}
	We want to illustrate the situation of the standard pair 
	$(\mathfrak{X}, \mathfrak{H}) := (\mathfrak{S}(1,a,1), \mathfrak{G}(1))$ for some $a \in K^{\circ \circ} \setminus \{0\}$.
	We can write $\mathfrak{X}_s = \Spec(\widetilde{K}[T_0,T_1,T_2]/(T_0 T_1))$ and $\mathfrak{H}_s = \{T_2 = 0\}$. 
	Then $\mathfrak{X}_s$ resp. $\mathfrak{H}_s$ decomposes into three strata with generic points 
	$x_1$, $x_2$, $x_3$ resp. $h_1$, $h_2$, $h_3$. Again the origin is the minimal stratum,
	denoted by $h_3$. 
	\begin{center}\includegraphics[scale=1]{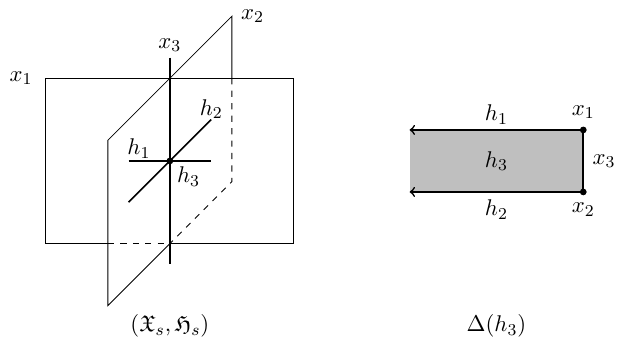}\end{center} 
\end{example}

\begin{noname} \label{FaceEmbDef}
	Let $y \in \str(\mathcal{X}, \mathcal{H})$ with $x \leq y$. Let us consider a combinatorial chart 
	$C := ([\mathbf{n}, \mathbf{r}, s_x], \alpha, \gamma)$ resp. $D := ([\mathbf{n}', \mathbf{r}', s_y], \alpha', \gamma')$
	in $x$ resp. $y$. The inclusion $\irr(\mathcal{X}, y) \hookrightarrow \irr(\mathcal{X}, x)$ induces an injective
	morphism $c := \alpha^{-1} \circ \alpha': [\mathbf{n}'] \to [\mathbf{n}]$ and by Proposition~\ref{ColorChangeProp}
	it is even an injective morphism $[\mathbf{n}', \mathbf{r}'] \to [\mathbf{n}, \mathbf{r}]$ of colored poly-simplices.
	Moreover let $j_{y, x}: D_y \hookrightarrow D_x$ be the injective map from Proposition~\ref{DivInjection}.
	Then we obtain an injective map $g := \gamma^{-1} \circ j_{y, x} \circ \gamma': \dotsbra{1}{s_y} \hookrightarrow \dotsbra{1}{s_x}$.
	
	Now $c$ and $g$ induce an injective morphism $[\mathbf{n}', \mathbf{r}', s_y] \to [\mathbf{n}, \mathbf{r}, s_x]$
	of extended poly-simplices. We denote it by $\iota_{D, C}$.
\end{noname}

\begin{proposition} \label{ChartEmbed}
	Let $y \in \str(\mathcal{X}, \mathcal{H})$ with $x \leq y$. On the one hand we consider combinatorial charts $C$ and $C'$ in $x$
	and on the other hand we consider combinatorial charts $D$ and $D'$ in $y$. 
	Then $h^{-1}_{C, C'} \circ \iota_{D', C'} \circ h_{D, D'} = \iota_{D, C}$
\end{proposition}

\begin{proof}
	This is immediately clear from the way the morphisms have been constructed.
\end{proof}

\begin{definition} 
	Let $y \in \str(\mathcal{X}, \mathcal{H})$ with $x \leq y$.
	From Proposition~\ref{ChartEmbed} we can conclude that the injective morphisms $\iota_{D, C}$ give rise
	to a well-defined injective continuous map $\iota_{y, x}: \Delta(y) \hookrightarrow \Delta(x)$. We call it the
	\emph{face embedding} from $y$ to $x$.   
\end{definition}

\begin{proposition}
	Let $y, z \in \str(\mathcal{X}, \mathcal{H})$ with $x \leq y \leq z$. Then $\iota_{x, x} = \id$
	and $\iota_{z, x} = \iota_{y, x} \circ \iota_{z, y}$. This allows us to identify $\Delta(y)$
	with a subspace of $\Delta(x)$. 
\end{proposition}

\begin{proof}
	Again just go through the construction and use the cocycle property stated in Proposition~\ref{DivInjection}. 
\end{proof}

\begin{lemma} \label{FaceEmbStr} 
	Let $C := ([\mathbf{n}, \mathbf{r}, s_x], \alpha, \gamma)$ be a combinatorial chart in $x$.
	Then for every injective morphism $[\mathbf{n}', \mathbf{r}', s'] \to [\mathbf{n}, \mathbf{r}, s_x]$ 
	of extended poly-simplices consisting of
	$c: [\mathbf{n}'] \hookrightarrow [\mathbf{n}]$ and $g: \dotsbra{1}{s'} \hookrightarrow \dotsbra{1}{s_x}$ 
	there exists a unique stratum $y \in \str(\mathcal{X}, \mathcal{H})$ 
	with $x \leq y$ such that $\im(\gamma \circ g) = D_{x, y}$ 
	and $D:= ([\mathbf{n}', \mathbf{r}', s'], \alpha \circ c, k_{x, y} \circ \gamma \circ g)$ is a combinatorial chart in $y$,
	where $D_{x, y}$ and $k_{x, y}$ are as in Proposition~\ref{DivInjection}. 
	In this case $c$ and $g$ induce the injective morphism $\iota_{D, C}$  
	and consequently the face embedding $\iota_{y, x}$. 
	Moreover $y$ does only depend on $\im(c)$ and $\im(g)$, and conversely $y$ uniquely determines $\im(c)$ and $\im(g)$.
\end{lemma}

\begin{proof}
	Let $J$ be the subset of all $\mathcal{V} \in \irr(\mathcal{H}, x)$ such that $\mathcal{V}$
	is contained in some element of $\im(\alpha \circ c) \subseteq \irr(\mathcal{X}, x)$.  
	Consider the intersection 
	\[\bigcap_{\mathcal{W} \in \im(\alpha \circ c)} \mathcal{W} \cap \bigcap_{I \in \im(\gamma \circ g)} \bigcap_{\mathcal{V} \in I \cap J} \mathcal{V}.\]
	Note that for every $I \in \im(\gamma \circ g)$ the intersection $I \cap J$ consists of $|\im(\alpha \circ c)|$ elements. 
	The generic points of the irreducible components of this intersection are elements of 
	$\str(\mathcal{X}, \mathcal{H})$, see Proposition \ref{StrataForm}.
	Moreover this intersection is smooth, thus it has exactly one irreducible component which contains $x$.
	Let $y$ be the generic point of this irreducible component. By construction we have $y \in \str(\mathcal{X}, \mathcal{H})$  
	and it satisfies the conditions $x \leq y$, $\im(\alpha \circ c) = \irr(\mathcal{X}, y)$
	and $\im(\gamma \circ g) = D_{x, y}$. Conversely this three conditions force us to make the choice for $y$
	as we did, which shows uniqueness. 
	
	Clearly $y$ does only depend on $\im(c)$ and $\im(g)$. Note that another choice of $c$ and $g$ such that
	$\im(c)$ or $\im(g)$ changes, would give rise to a different $y$.
	We convince ourselves now that $D$ is a combinatorial chart in $y$.
	
	It follows from Proposition \ref{IsomorphIsometric} that 
	$\alpha \circ c: [\mathbf{n}'] \xrightarrow{\sim} \irr(\mathcal{X}, y)$ is an isometric bijection.
	Obviously $k_{x, y} \circ \gamma \circ g: \dotsbra{1}{s'} \xrightarrow{\sim} D_y$ is a bijection.
	The only thing left to check is the equality $\mathbf{r}' = \mathbf{r}_{\alpha \circ c}$, 
	but this is a consequence of Proposition \ref{ColorChangeProp}. 
\end{proof}

\begin{lemma} \label{DisjointUnion}
	We continue to denote $x \in \str(\mathcal{X}, \mathcal{H})$. 
	\begin{enumerate}
		\item Then 
		\[\Delta(x) = \bigcup \Delta^{\circ}(y),\]
		where the union ranges over all $y \in \str(\mathcal{X}, \mathcal{H})$ with $x \leq y$.
		Moreover this union is disjoint. 
		\item Let $y, z \in \str(\mathcal{X}, \mathcal{H})$
		with $x \leq y$ and $x \leq z$. If $\Delta(y) \cap \Delta^{\circ}(z) \neq \emptyset$,
		then $\Delta(z) \subseteq \Delta(y)$ and $y \leq z$.
	\end{enumerate} 
\end{lemma}

\begin{proof}
	First one easily verifies the following elementary geometric facts:
	\begin{itemize}
		\item Every geometric poly-simplex $\Delta(\mathbf{n}, \mathbf{r}, s)$ can be written as the
		union of open simplices $\Delta^{\circ}(\mathbf{n}', \mathbf{r}', s')$, which are considered
		as subspaces via the geometric realization of injective morphisms 
		$[\mathbf{n}', \mathbf{r}', s'] \to [\mathbf{n}, \mathbf{r}, s]$
		\item Let $(c', g'): [\mathbf{n}', \mathbf{r}', s'] \to [\mathbf{n}, \mathbf{r}, s]$ and
		$(c'', g''): [\mathbf{n}'', \mathbf{r}'', s''] \to [\mathbf{n}, \mathbf{r}, s]$ be 
		two injective morphisms. If $\im(c') = \im(c'')$ and $\im(g') = \im(g'')$
		then $\Delta^{\circ}(\mathbf{n}', \mathbf{r}', s') = \Delta^{\circ}(\mathbf{n}'', \mathbf{r}'', s'')$.
		Otherwise $\Delta^{\circ}(\mathbf{n}', \mathbf{r}', s')$ and $\Delta^{\circ}(\mathbf{n}'', \mathbf{r}'', s'')$ 
		are disjoint.
		\item If $\Delta(\mathbf{n}', \mathbf{r}', s') \cap \Delta^{\circ}(\mathbf{n}'', \mathbf{r}'', s'') \neq \emptyset$,
		then $\Delta(\mathbf{n}'', \mathbf{r}'', s'') \subseteq \Delta(\mathbf{n}', \mathbf{r}', s')$. 
		In this case we have $\im(c'') \subseteq \im(c')$ and $\im(g'') \subseteq \im(g')$,
		in particular there exists an injective morphism 
		$(\widetilde{c}, \widetilde{g}): [\mathbf{n}'', \mathbf{r}'', s''] \to [\mathbf{n}', \mathbf{r}', s']$  
		such that $(c'', g'') = (c', g') \circ (\widetilde{c}, \widetilde{g})$.
	\end{itemize}
	Now the claims are straightforward consequences of Lemma \ref{FaceEmbStr}.
\end{proof}


\subsection{The strict case}  

Again let $(\mathfrak{X}, \mathfrak{H})$ be a strictly poly-stable pair 
and $\mathcal{X} := \mathfrak{X}_s$, $\mathcal{H} := \mathfrak{H}_s$.

In the following, we define the dual intersection complex
by gluing together the canonical polyhedra along the face embeddings.
We also provide an easy example.

\begin{noname}
	We consider the partially ordered set $\str(\mathcal{X}, \mathcal{H})$ in the obvious way
	as a category. 
	Explicitly its objects are the elements of $\str(\mathcal{X}, \mathcal{H})$ and the set of morphisms from
	$y$ to $x$ is the set containing the pair $(y, x)$ if $x \leq y$ and the empty set otherwise. 
	Let $F_{(\mathfrak{X}, \mathfrak{H})}$ be the functor 
	from $\str(\mathcal{X}, \mathcal{H})$ to the category of topological spaces which maps $x$ to $\Delta(x)$ and
	a pair $(y, x)$ to the face embedding $\iota_{y, x}: \Delta(y) \hookrightarrow \Delta(x)$.
\end{noname}

\begin{definition}	
	We define the \emph{dual intersection complex} $C(\mathfrak{X}, \mathfrak{H})$ of the strictly poly-stable pair $(\mathfrak{X}, \mathfrak{H})$ 
	to be the colimit of this functor $F_{(\mathfrak{X}, \mathfrak{H})}$ from above. 
\end{definition}

\begin{lemma} \label{EquivRel}
	More explicitly we can describe the dual intersection complex as 
	\[C(\mathfrak{X}, \mathfrak{H}) = \bigsqcup_{x \in \str(\mathcal{X}, \mathcal{H})} \Delta(x) / \sim~,\]
	where $\sim$ is the equivalence relation on the disjoint union given as follows: Two points $a \in \Delta(x)$ 
	and $b \in \Delta(y)$ are in relation $a \sim b$ iff there exists an element $z \in \str(\mathcal{X}, \mathcal{H})$
	with $x \leq z$ and $y \leq z$ such that $a, b \in \Delta(z)$ and $a$ agrees with $b$ in $\Delta(z)$, i.\,e. there exists
	a point $c \in \Delta(z)$ such that $\iota_{z, x}(c) = a$ and $\iota_{z, y}(c) = b$.	
\end{lemma}

\begin{proof}
	The only non-trivial thing to check is, that the relation $\sim$ is in fact transitive: 
	Let $x, y, z \in \str(\mathcal{X}, \mathcal{H})$ and $a \in \Delta(x)$, $b \in \Delta(y)$, 
	$c \in \Delta(z)$ with $a \sim b$ and $b \sim c$. By definition there exist $v, w \in \str(\mathcal{X}, \mathcal{H})$
	with $x \leq v$, $y \leq v$, $y \leq w$, $z \leq w$ such that $a, b \in \Delta(v)$, $b, c \in \Delta(w)$ 
	and $a$ agrees with $b$ in $\Delta(v)$ and $b$ agrees with $c$ in $\Delta(w)$. 
	
	Let $u \in \str(\mathcal{X}, \mathcal{H})$ be the unique element with $y \leq u$ and $b \in \Delta^{\circ}(u)$,
	see Lemma \ref{DisjointUnion}. 
	Since $b \in \Delta(v) \cap \Delta^{\circ}(u)$ and
	$b \in \Delta(w) \cap \Delta^{\circ}(u)$ we get $v \leq u$ and $w \leq u$.
	In particular $x \leq u$, $z \leq u$ and $a, b, c$ all agree in $\Delta(u)$,
	which shows $a \sim c$. 	
\end{proof}

\begin{definition}
	For every $x \in \str(\mathcal{X}, \mathcal{H})$ the canonical map $\Delta(x) \to C(\mathfrak{X}, \mathfrak{H})$ is injective
	according to Lemma \ref{EquivRel} and we call its image the \emph{face} of $C(\mathfrak{X}, \mathfrak{H})$ 
	associated to $x$. It is a closed subset of $C(\mathfrak{X}, \mathfrak{H})$ and we denote it by $\Delta(x)$ as well.
	Similarly we consider $\Delta^{\circ}(x)$ as a subspace of $C(\mathfrak{X}, \mathfrak{H})$, 
	called the \emph{open face} of $C(\mathfrak{X}, \mathfrak{H})$ associated to $x$.
	We point out that $\Delta^{\circ}(x)$ is in general not an open subset of $C(\mathfrak{X}, \mathfrak{H})$.
	It is however an open subset of $\Delta(x)$.   
	
	Let $y \in \str(\mathcal{X}, \mathcal{H})$ with $x \leq y$. The injective continuous map $\Delta(y) \hookrightarrow \Delta(x)$
	of subspaces of $C(\mathfrak{X}, \mathfrak{H})$ induced by the face embedding $\iota_{y, x}$ is called a \emph{face embedding} as well.
\end{definition}

\begin{corollary} \label{FaceIntersection}
	Let $x, y \in \str(\mathcal{X}, \mathcal{H})$. Then we have an equality  
	\[\Delta(x) \cap \Delta(y) = \bigcup \Delta(z) \]
	of closed subsets of $C(\mathfrak{X}, \mathfrak{H})$, where the union ranges over all 
	$z \in \str(\mathcal{X}, \mathcal{H})$ with $x \leq z$ and $y \leq z$. 
\end{corollary}

\begin{proof}
	This follows immediately from Lemma~\ref{EquivRel}. 
\end{proof}

\begin{corollary} \label{OpenFacesUnion}
	$C(\mathfrak{X}, \mathfrak{H})$ is equal to the union of its open faces, and this union is disjoint. 
\end{corollary}

\begin{proof}
	Use Lemma~\ref{DisjointUnion}~(i) and Corollary~\ref{FaceIntersection}. 
\end{proof}

\begin{example}
	Let $\mathfrak{X} := \mathfrak{S}(1, a) = \Spf(K^{\circ}\{T_0, T_1\}/(T_0T_1 - a))$ 
	for some $a \in K^{\circ \circ} \setminus \{0\}$. 
	Pick $b_1, b_2, b_3 \in K^{\circ} \setminus K^{\circ \circ}$ such that $b_2$ and $b_3$
	do not reduce to the same element in the residue field $\widetilde{K}$.
	We denote by $\mathfrak{H}_1$, $\mathfrak{H}_2$, $\mathfrak{H}_3$ the disjoint closed subschemes
	of $\mathfrak{X}$ defined by the ideals $(T_1 - b_1)$, $(T_0 - b_2)$, $(T_0 - b_3)$,
	respectively. Then consider $\mathfrak{H} := \mathfrak{H}_1 \cup \mathfrak{H}_2 \cup \mathfrak{H}_3$. 
	If we embed $\mathfrak{X}_{\eta}$ into the unit ball $B = \mathscr{M}(K\{T\})$ via $T \mapsto T_1$
	as an annulus of inner radius $|a|$ and outer radius $1$,
	we can describe $\mathfrak{H}_{\eta}$ as the three points of type $1$ corresponding 
	to the elements $b_1$, $ab^{-1}_2$, $ab^{-1}_3$ of $K^{\circ}$.
	The special fiber $\mathfrak{H}_s$ consists of $3$ points, namely the point $\mathfrak{H}_{1, s}$, which lies on $\{T_0 = 0\}$, 
	and the points $\mathfrak{H}_{2, s}$, $\mathfrak{H}_{3, s}$, which lie on $\{T_1 = 0\}$. 
	In the dual intersection complex $C(\mathfrak{X}, \mathfrak{H})$ 
	the infinite rays correspond to these three points.
	The line segment has length $\val(a)$ and is just the dual graph of $\mathfrak{X}_s$.  
	\begin{center}\includegraphics[scale=1]{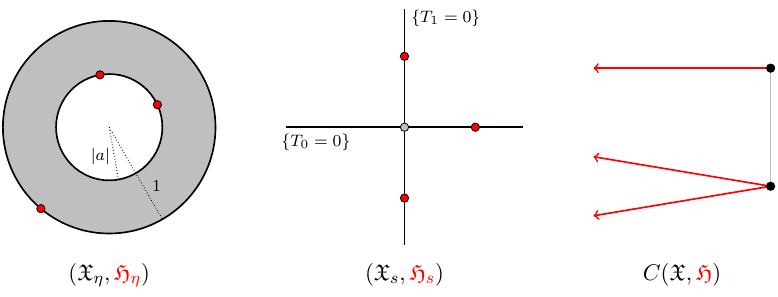}\end{center}
\end{example}

\begin{lemma} \label{IntersectionComplexMap}
	Let $(\mathfrak{X}, \mathfrak{H})$ be a strictly poly-stable pair and
	$\psi: \mathfrak{Y} \to \mathfrak{X}$ an \et morphism. Consider $\mathfrak{G} := \psi^{-1}(\mathfrak{H})$. 
	Then the induced morphism $\psi: (\mathfrak{Y}, \mathfrak{G}) \to (\mathfrak{X}, \mathfrak{H})$
	of strictly poly-stable pairs induces a continuous map $C(\psi): C(\mathfrak{Y}, \mathfrak{G}) \to C(\mathfrak{X}, \mathfrak{H})$
	of dual intersection complexes respecting faces and face embeddings.
	
	Moreover this construction is functorial, i.\,e. if $\psi': \mathfrak{Z} \to \mathfrak{Y}$ is another \et morphism,
	then $C(\psi \circ \psi') = C(\psi) \circ C(\psi')$.
\end{lemma}

\begin{proof}
	We continue to use the notations from Proposition \ref{EtaleMorProp} and its proof.
	We have already seen there that $\psi$ induces a bijection $\delta: D_y \xrightarrow{\sim} D_x$.
	In our situation we choose $y \in \str(\mathfrak{Y}_s, \mathfrak{G}_s)$ and therefore 
	$x = \psi(y) \in \str(\mathfrak{X}_s, \mathfrak{H}_s)$. 
	Moreover $\psi$ induces an isometric bijection $\alpha: \irr(\mathfrak{Y}_s, y) \xrightarrow{\sim} \irr(\mathfrak{X}_s, x)$
	such that $\alpha_{\varphi} = \alpha \circ \alpha_{\varphi'}$.
	One now easily sees that $\alpha$ and $\delta$ give rise to a homeomorphism $h_{y, x}: \Delta(y) \xrightarrow{\sim} \Delta(x)$.
	If we have $y' \in \str(\mathfrak{Y}_s, \mathfrak{G}_s)$ with $y' \leq y$, then 
	$x' := \psi(y') \in \str(\mathfrak{X}_s, \mathfrak{H}_s)$ satisfies $x' \leq x$ and
	$h_{y', x'} \circ \iota_{y, y'} = \iota_{x, x'} \circ h_{y, x}$.
	In conclusion we obtain a continuous map $C(\psi): C(\mathfrak{Y}, \mathfrak{G}) \to C(\mathfrak{X}, \mathfrak{H})$,
	which is on the faces given by our constructed homeomorphisms $h_{y, x}$. 
	The functoriality is clear.  
\end{proof}


\subsection{The general case}

Let now $(\mathfrak{X}, \mathfrak{H})$ be a poly-stable pair.
We choose a strictly poly-stable pair $(\mathfrak{Y}, \mathfrak{G})$ 
and a surjective \et morphism $\psi: (\mathfrak{Y}, \mathfrak{G}) \to (\mathfrak{X}, \mathfrak{H})$.

With the help of the coequalizer we define the dual intersection complex using the one from the strict case.
Up to a canonical homeomorphism it turns out to be independent of the choice of $(\mathfrak{Y}, \mathfrak{G})$ 
and $\psi$. Moreover we study its structure and give an example.

\begin{definition}
	Let $f:Z \to Y$ and $g:Z \to Y$ be two continuous maps of topological spaces. Then we define the \emph{coequalizer}
	of $f$ and $g$, denoted by
	\[\coker(Z \rightrightarrows Y),\]
	as the quotient space $Y/{\sim}$, where $\sim$ is the equivalence relation on $Y$ identifying
	two elements $y, y' \in Y$ iff there exists an element $z \in Z$ such that $f(z) = y$ and $g(z) = y'$.  
\end{definition}

\begin{definition} 
	We consider the fiber product $\mathfrak{Z} := \mathfrak{Y} \times_{\mathfrak{X}} \mathfrak{Y}$ via $\psi$
	and the two canonical projections $p_1$ and $p_2$ from $\mathfrak{Z}$ to $\mathfrak{Y}$.
	Let us denote $\mathfrak{F} := p^{-1}_1(\mathfrak{G}) = p^{-1}_2(\mathfrak{G})$, which we 
	see as a closed subscheme of $\mathfrak{Z}$.
	Applying Lemma \ref{IntersectionComplexMap} yields us two continuous maps $C(p_1)$ and $C(p_2)$
	from $C(\mathfrak{Z}, \mathfrak{F})$ to $C(\mathfrak{Y}, \mathfrak{G})$.
	
	We now define the \emph{dual intersection complex} $C(\mathfrak{X}, \mathfrak{H})$ as the coequalizer
	\[\coker(C(\mathfrak{Z}, \mathfrak{F}) \rightrightarrows C(\mathfrak{Y}, \mathfrak{G}))\]
	with respect to $C(p_1)$ and $C(p_2)$. 
	There is a canonical projection $\pi: C(\mathfrak{Y}, \mathfrak{G}) \to C(\mathfrak{X}, \mathfrak{H})$
	by which $C(\mathfrak{X}, \mathfrak{H})$ inherits the quotient topology.   	
	
	We define a \emph{face} resp. \emph{open face} of $C(\mathfrak{X}, \mathfrak{H})$ to be the image of 
	a face resp. open face of $C(\mathfrak{Y}, \mathfrak{G})$ under $\pi$. 	
\end{definition}

\begin{lemma} \label{StrataFaceCorresp}
	Each of the projections $p_1$ and $p_2$ induces a morphism $\str(\mathfrak{Z}_s) \to \str(\mathfrak{Y}_s)$
	resp. $\str(\mathfrak{F}_s) \to \str(\mathfrak{G}_s)$ of partially ordered sets, see Proposition \ref{StrataMap}.
	Then the morphism $\psi$ induces isomorphisms of partially ordered sets
	\[\coker(\str(\mathfrak{Z}_s) \rightrightarrows \str(\mathfrak{Y}_s)) \xrightarrow{\sim} \str(\mathfrak{X}_s) 
	\text{\quad and \quad} \coker(\str(\mathfrak{F}_s) \rightrightarrows \str(\mathfrak{G}_s)) \xrightarrow{\sim} \str(\mathfrak{H}_s).\]
	Moreover every pair $x \leq x'$ in $\str(\mathfrak{X}_s, \mathfrak{H}_s)$ comes from a pair
	$y \leq y'$ in $\str(\mathfrak{Y}_s, \mathfrak{G}_s)$.
\end{lemma}

\begin{proof}
	The first part is an immediate application of \cite[Corollary 2.8]{Ber99}.
	
	The remaining claim is evident from the first one in the case 
	$x, x' \in \str(\mathfrak{X}_s)$ or $x, x' \in \str(\mathfrak{H}_s)$. 
	So let us consider the case $x \in \str(\mathfrak{H}_s)$ and $x' \in \str(\mathfrak{X}_s)$.
	The preimage of the stratum corresponding to $x$ resp. $x'$ under $\psi_s$
	is a non-empty strata subset $S$ resp. $T$ of $\mathfrak{G}_s$ resp. $\mathfrak{Y}_s$. 
	Pick any stratum of $\mathfrak{G}_s$ which lies over the stratum corresponding to $x$
	and let $y$ be its generic point. Since taking the topological closure is compatible
	with taking preimage under an open continuous map, we conclude from $x \leq x'$ and $y \in S$ 
	that $y$ is contained in the closure of $T$. But this means that $y \in \overline{\{y'\}}$
	for some $y' \in \str(\mathfrak{Y}_s)$ whose associated stratum lies over the 
	associated stratum of $x'$. 
\end{proof}

\begin{proposition} \label{FaceProp}
	Let $y \in \str(\mathfrak{Y}_s, \mathfrak{G}_s)$.
	Then the preimage of $\pi(\Delta^{\circ}(y))$ resp. $\pi(\Delta(y))$ under $\pi$ 
	is equal to the union of all $\Delta^{\circ}(y')$ resp. $\Delta(y')$ with 
	$y' \in \str(\mathfrak{Y}_s, \mathfrak{G}_s)$ such that $\psi_s(y) = \psi_s(y')$.
	In particular the open faces of $C(\mathfrak{Y}, \mathfrak{G})$ are disjoint. 
\end{proposition}

\begin{proof}
	The claim for the open faces $\pi(\Delta^{\circ}(y))$ follows from Lemma~\ref{StrataFaceCorresp}, Corollary~\ref{OpenFacesUnion}
	and Lemma~\ref{IntersectionComplexMap}. The assertion about the faces $\pi(\Delta(y))$ is then an easy consequence
	using Lemma~\ref{DisjointUnion}. 
\end{proof}

\begin{lemma} \label{FaceFacts}
	We state some fundamental properties of faces of the dual intersection complex:
	\begin{enumerate}
		\item Every face of $C(\mathfrak{X}, \mathfrak{H})$ is a closed subset of $C(\mathfrak{X}, \mathfrak{H})$.
		\item There is a bijective correspondence between the faces resp. open faces of $C(\mathfrak{X}, \mathfrak{H})$
		and the strata of $(\mathfrak{X}_s, \mathfrak{H}_s)$. For every $x \in \str(\mathfrak{X}_s, \mathfrak{H}_s)$
		we will denote by $\Delta(x)$ resp. $\Delta^{\circ}(x)$ the corresponding face resp. open face. 
		Then $\Delta^{\circ}(x)$ is contained in $\Delta(x)$ as an open subset. 
		\item For $x, x' \in \str(\mathfrak{X}_s, \mathfrak{H}_s)$ with $x \leq x'$
		we have an inclusion $\iota_{x', x}: \Delta(x') \hookrightarrow \Delta(x)$,
		which we call a \emph{face embedding}.
		More explicitly, there exist $y, y' \in \str(\mathfrak{Y}_s, \mathfrak{G}_s)$ with $x := \psi_s(y)$,
		$x' := \psi_s(y')$ and $y \leq y'$. Then $\iota_{x', x}$ is induced by $\iota_{y', y}$ via $\pi$
		and $\iota_{x', x}$ does not depend on the choice of $y, y'$.
		
		If $x'' \in \str(\mathfrak{X}_s, \mathfrak{H}_s)$ with $x' \leq x''$ we have
		$\iota_{x'', x} = \iota_{x', x} \circ \iota_{x'', x'}$. We use the face embeddings to 
		consider $\Delta(x')$ as a subspace of $\Delta(x)$. 
		\item The restriction of $\pi$ to an open face of $C(\mathfrak{Y}, \mathfrak{G})$ 
		induces a homeomorphism onto an open face of $C(\mathfrak{X}, \mathfrak{H})$. 
	\end{enumerate}
\end{lemma}

\begin{proof}
	Every face of $C(\mathfrak{X}, \mathfrak{H})$ is equal to $\pi(\Delta(y))$ for some 
	$y \in \str(\mathfrak{Y}_s, \mathfrak{G}_s)$. 
	By Proposition \ref{FaceProp} the preimage of $\pi(\Delta(y))$ under $\pi$ 
	is equal to the union of all $\Delta(y')$ such that $\psi_s(y) = \psi_s(y')$.
	This union of the $\Delta(y')$ is a closed subset of $C(\mathfrak{Y}, \mathfrak{G})$,
	therefore $\pi(\Delta(y)) \subseteq C(\mathfrak{X}, \mathfrak{H})$ is closed, which proves (i).
	
	It also becomes clear how to define the correspondence in (ii). For every 
	$x \in \str(\mathfrak{X}_s, \mathfrak{H}_s)$ define $\Delta(x)$ to be the image under $\pi$
	of the union of all $\Delta(y')$ with $y' \in \str(\mathfrak{Y}_s, \mathfrak{G}_s)$
	such that $\psi_s(y') = x$. According to the above $\Delta(x)$ is a face of $C(\mathfrak{X}, \mathfrak{H})$.
	Obviously every face of $C(\mathfrak{X}, \mathfrak{H})$ can be obtained in this way
	and for each face the $x \in \str(\mathfrak{X}_s, \mathfrak{H}_s)$ 
	is uniquely determined. Proceed analogously for open faces to complete 
	the claim~(ii). 
	
	We handle (iii) again with Lemma \ref{StrataFaceCorresp} and the claim 
	follows from the fact that the whole construction respects the face embeddings
	in the strictly poly-stable case.
	
	Finally (iv) follows from Corollary~\ref{OpenFacesUnion}.
\end{proof}

\begin{remark}
	Despite their notation, in the poly-stable case the faces $\Delta(x)$ do in general no longer
	look like geometric poly-simplices. However they can be obtained from a geometric poly-simplex
	by gluing together some of its faces. Which faces are glued together and in which way, is dictated by
	how the strata of $(\mathfrak{Y}_s, \mathfrak{G}_s)$ map to the ones of $(\mathfrak{X}_s, \mathfrak{H}_s)$
	under $\psi_s$.
\end{remark}

\begin{theorem}
	Let us collect some important results concerning the structure of the dual intersection complex.
	Let $x, x' \in \str(\mathfrak{X}_s, \mathfrak{H}_s)$.
	\begin{enumerate}
		\item We have 
		\[\Delta(x) = \bigcup \Delta^{\circ}(x''),\]
		where the union ranges over all $x'' \in \str(\mathfrak{X}_s, \mathfrak{H}_s)$ 
		with $x \leq x''$, and this union is disjoint.
		\item If $\Delta(x) \cap \Delta^{\circ}(x') \neq \emptyset$, then $x \leq x'$.
		\item We have an equality  
		\[\Delta(x) \cap \Delta(x') = \bigcup \Delta(x'')\]
		of closed subsets of $C(\mathfrak{X}, \mathfrak{H})$, where the union ranges over all 
		$x'' \in \str(\mathfrak{X}_s, \mathfrak{H}_s)$ with $x \leq x''$ and $x' \leq x''$.
		\item $C(\mathfrak{X}, \mathfrak{H})$ is equal to the union of its open faces, and this union is disjoint.  
	\end{enumerate}	
\end{theorem}

\begin{proof}
	We prove (i) by writing $\Delta(x) = \pi(\Delta(y))$ for some $y \in \str(\mathfrak{Y}_s, \mathfrak{G}_s)$ 
	and decomposing $\Delta(y)$ into its open faces.
	Then the claim is clear from Proposition~\ref{FaceProp}.
	
	In order to verify (ii) we assume $\Delta(x) \cap \Delta^{\circ}(x') \neq \emptyset$. By (i) we can decompose
	$\Delta(x)$ into open faces $\Delta^{\circ}(x'')$ with $x'' \in \str(\mathfrak{X}_s, \mathfrak{H}_s)$ 
	such that $x \leq x''$. Since the open faces are disjoint, we conclude $\Delta^{\circ}(x'') = \Delta^{\circ}(x')$
	for some $x''$, therefore $x'' = x'$ and $x \leq x'$. 
	
	Similarly (iii) is shown by decomposing into open faces and (iv) is clear.     
\end{proof}

\begin{proposition} \label{TheDualIntComplex}
	Let $\chi: \mathfrak{V} \to \mathfrak{Y}$ be a surjective \et morphism of admissible formal schemes
	and $\mathfrak{E} := \chi^{-1}(\mathfrak{G})$. We denote by $\mathfrak{W}$ the fiber product
	$\mathfrak{V} \times_{\mathfrak{X}} \mathfrak{V}$ via $\psi \circ \chi$
	and by $\mathfrak{D}$ the pullback of $\mathfrak{E}$ with respect to one of the 
	canonical projections $\mathfrak{W} \to \mathfrak{V}$.
	Then $C(\chi)$ induces a face respecting homeomorphism of coequalizers
	\[\coker(C(\mathfrak{W}, \mathfrak{D}) \rightrightarrows C(\mathfrak{V}, \mathfrak{E}))
	\xrightarrow{\sim} \coker(C(\mathfrak{Z}, \mathfrak{F}) \rightrightarrows C(\mathfrak{Y}, \mathfrak{G})).\]
\end{proposition}

\begin{proof}
	Denote by $\pi'$ the projection 
	$C(\mathfrak{V}, \mathfrak{E}) \to \coker(C(\mathfrak{W}, \mathfrak{D}) \rightrightarrows C(\mathfrak{V}, \mathfrak{E}))$.
	Let $v \in \str(\mathfrak{V}_s, \mathfrak{E}_s)$ and $x := \psi(\chi(v)) \in \str(\mathfrak{X}_s, \mathfrak{H}_s)$.
	Then $C(\chi)$ induces a homeomorphism from $\pi'(\Delta(v))$ to $\pi(\Delta(y))$,
	where $y$ is any element in $\str(\mathfrak{Y}_s, \mathfrak{G}_s)$ with $\psi(y) = x$.
	The homeomorphism between the coequalizers is obtained by gluing together these homeomorphisms
	between the faces. Accordingly, one gets the inverse by gluing together the inverse homeomorphisms between the faces.
\end{proof}

\begin{remark}
	So far we made all our considerations with respect to the choice of a 
	strictly poly-stable pair $(\mathfrak{Y}, \mathfrak{G})$ 
	and a surjective \et morphism $\psi: (\mathfrak{Y}, \mathfrak{G}) \to (\mathfrak{X}, \mathfrak{H})$. 
	However the dual intersection complex $C(\mathfrak{X}, \mathfrak{H})$
	only depends on this choice up to a canonical homeomorphism respecting the faces.
	One way to prove this, is to pick another pair $(\mathfrak{Y}', \mathfrak{G}')$ 
	and another morphism $\psi': (\mathfrak{Y}', \mathfrak{G}') \to (\mathfrak{X}, \mathfrak{H})$
	and then apply Proposition \ref{TheDualIntComplex} with $\mathfrak{V}$ being the
	fiber product $\mathfrak{Y} \times_{\mathfrak{X}} \mathfrak{Y}'$  
	via $\psi$ and $\psi'$. Then we can choose $\chi$ to be the first as well as the second projection. 
	This justifies speaking of “the” dual intersection complex
	of $(\mathfrak{X}, \mathfrak{H})$. 
\end{remark}

\begin{example} \label{NonStrict}
	We want to give a sketch of a non-strict example.
	There is an admissible formal scheme $\mathfrak{X}$ whose generic fiber is a smooth analytic curve
	and whose special fiber is a curve with a nodal singularity, as seen in the graphic below.
	It admits a strictly semi-stable formal scheme $\mathfrak{Y}$ and a surjective \et morphism 
	$\mathfrak{Y} \to \mathfrak{X}$. The generic fiber $\mathfrak{Y}_{\eta}$ is smooth 
	and the special fiber $\mathfrak{Y}_s$ has two ordinary double points at the intersection
	of the irreducible components. 
	\begin{center}\includegraphics[scale=1]{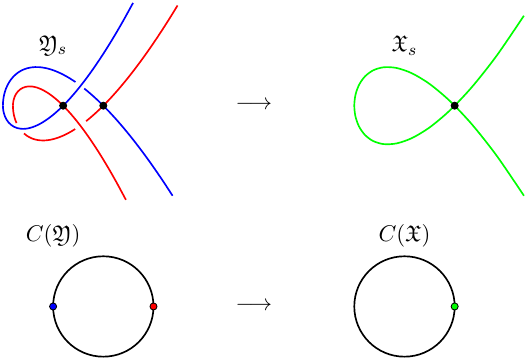}\end{center}
\end{example}


\section{Extended Skeletons}

We approach the main purpose of this paper, namely the development and investigation
of the “extended skeleton”. Its definition relies on the classical skeleton described by Berkovich,
which is why we give a brief outline of his method from \cite[§\,5]{Ber99}.

In the case of a non-trivially valued field $K$, our technique for defining the extended skeleton
of a standard pair and later of a building block, is essentially the $\varepsilon$-approximation method
already used in \cite[4.2 and 4.3]{GRW16}. One can think of taking successively growing classical skeletons
and taking their union in order to obtain the extended skeleton.

Following \cite[4.6]{GRW16}, the extended skeleton of a strictly poly-stable pair 
is then made up of the building block skeletons,
as their name already suggests. 

In the arbitrary poly-stable case we apply the procedure from Step 9 of the Proof of Theorems 5.2–5.4 in \cite[§\,5]{Ber99}
to our situation. We are then able to formulate and prove our main theorems.

Finally the case of a trivially valued field $K$ has to be considered separately, since here our $\varepsilon$-approximation method
is not available. However it is possible by passing to a non-trivially valued extension of $K$, to extend all of our results
from before to this situation as well.  


\subsection{Classical construction}

In this subsection we give an overview of the construction of the skeleton 
as performed in \cite[§\,5]{Ber99}.

\begin{noname} \label{ClassicStart}
	First we describe how to construct the skeleton of a formal scheme 
	$\mathfrak{X} := \mathfrak{S}(\textbf{n}, \textbf{a})$, where
	$p \in \Nn$, $\textbf{n} = (n_1,\dots,n_p)$ and $\textbf{a} = (a_1,\dots,a_p)$ with $n_i \in \Nn$ and
	$a_i \in K^{\circ}$ for all $i \in \dotsbra{1}{p}$.  
	
	Let $\mathscr{A}$ be the $K$-affinoid algebra $K\{T_{10},\dots ,T_{1n_1}, \dots , T_{p0},\dots ,T_{pn_p}\} / \mathfrak{a}$,
	where $\mathfrak{a}$ is the ideal generated by $T_{i0} \cdots T_{in_i} - a_i$ for all $i \in \dotsbra{1}{p}$.
	Then we set $X := \mathfrak{X}_{\eta} = \mathscr{M}(\mathscr{A})$
	and $\textbf{r} := \val(\textbf{a}) = \dotsroundbra{\val(a_1)}{\val(a_p)}$.
	We define the \emph{tropicalization map} as
	\begin{align*}
	\trop: X &\to \Delta(\mathbf{n}, \mathbf{r}) \\
	x &\mapsto \dotsroundbra{-\log|T_{10}(x)|}{-\log|T_{pn_p}(x)|}.
	\end{align*}
	Recall that $\Delta(\mathbf{n}, \mathbf{r})$ is the geometric poly-simplex from Definition~\ref{GeomSimplex}. 
	The map $\trop$ is well-defined and continuous.
	Furthermore we define a map $\sigma: \Delta(\mathbf{n}, \mathbf{r}) \to X$, $\textbf{s} \mapsto |\enspace|_{\textbf{s}}$, where
	\[\Big|\sum_{\boldsymbol{\mu}} c_{\boldsymbol{\mu}} T^{\boldsymbol{\mu}}\Big|_{\textbf{s}} := 
	\max\limits_{\boldsymbol{\mu}} |c_{\boldsymbol{\mu}}|\exp(-\boldsymbol{\mu} \cdot \textbf{s})\]
	and $\boldsymbol{\mu} \cdot \textbf{s} := \mu_{10} s_{10} + \dots + \mu_{pn_p} s_{pn_p}$ for all 
	$\boldsymbol{\mu} = \dotsroundbra{\mu_{10}}{\mu_{pn_p}}$.
	Note here that every element $f \in \mathscr{A}$ can be uniquely represented by a formal power series 
	$\sum_{\boldsymbol{\mu}} c_{\boldsymbol{\mu}} T^{\boldsymbol{\mu}}$ such that $c_{\boldsymbol{\mu}} = 0$ for all 
	$\boldsymbol{\mu}$ with $T^{\boldsymbol{\mu}}$ being divisible by $T_{i0} \cdots T_{in_i}$ for some $i \in \dotsbra{1}{p}$.
	This is the representative for $f$ one has to choose in the above definition.
	It is easy to show that $\sigma$ is well-defined and continuous.
	
	A straightforward calculation confirms that $\sigma$ is a right inverse to $\trop$ and that $\tau := \sigma \circ \trop$ is a retraction of 
	$X$ onto the closed compact subset $S := \sigma(\Delta(\textbf{n}, \textbf{r}))$, which is homeomorphic to
	$\Delta(\textbf{n}, \textbf{r})$ via $\trop$. We call $S$ the \emph{skeleton} of $\mathfrak{X}$. 
\end{noname}

\begin{proposition} \label{StdRed}
	We use the same notations as above. 
	Let $x$ be the generic point of the minimal stratum of $\mathfrak{X}_s$. 
	Then $\trop$ restricts to a homeomorphism  
	\[S \cap \red^{-1}_{\mathfrak{X}}(x) \xrightarrow{\sim} \Delta^{\circ}(\textbf{n}, \textbf{r}). \]    
\end{proposition}

\begin{proof}
	It is enough to prove that for all $\textbf{s} \in \Delta^{\circ}(\textbf{n}, \textbf{r})$
	we have $\red_{\mathfrak{X}}(\sigma(\textbf{s})) = x$. Let us write
	\[\widetilde{A} := \bigotimes\limits_{i=1}^p \widetilde{K}[T_{i0},\dots,T_{in_i}]/(T_{i0} \cdots T_{in_i} - \widetilde{a}_i),\]
	where $\widetilde{a}_i$ is the reduction of $a_i$ modulo $K^{\circ\circ}$.
	Then $x \in \mathfrak{X}_s = \Spec(\widetilde{A})$ is the ideal generated by $T_{ij}$ 
	for all $i \in \dotsbra{1}{p}$ and $j \in \dotsbra{1}{n_i}$ such that $\widetilde{a}_i = 0$,
	in other words $|a_i| < 1$. It is now easy to see from the definitions that 
	$\red_{\mathfrak{X}}(\sigma(\textbf{s})) \in \mathfrak{X}_s$ is an ideal containing $x$,
	because $|T_{ij}|_{\textbf{s}} < 1$.
	To show that it is also contained in $x$, note that for the indexes $i \in \dotsbra{1}{p}$
	with $|a_i| = 1$ we have $\Delta(n_i, -\log|a_i|) = \Delta(n_i,0) = \{\mathbf{0}\}$.
\end{proof}

\begin{example} \label{AnnulusSkeleton}
	Let $a \in K^{\circ\circ} \setminus \{0\}$ and $\varepsilon := |a|$. 
	We consider the standard scheme $\mathfrak{X} = \mathfrak{S}(1, a)$.
	Below there is a sketch of the Berkovich unit disc $B$. If one removes 
	the gray part, we obtain the generic fiber $X = \mathfrak{S}(1, a)_{\eta}$, 
	which is the closed annulus with outer radius $1$ and inner radius $\varepsilon$.
	Then the skeleton $S$ is homeomorphic to $\Delta(1, -\log(\varepsilon))$,
	which is visualized as the red segment. Its upper endpoint corresponds to the Gauß norm
	whereas the lower endpoint corresponds to the “modified Gauß norm” given by sending
	a strictly convergent power series $\sum_{n = 0}^{\infty} c_n T^n \in K\{T\}$ 
	to $\max\limits_{n \in \N} |c_n|\varepsilon^n$.
	\begin{center}\includegraphics[scale=1]{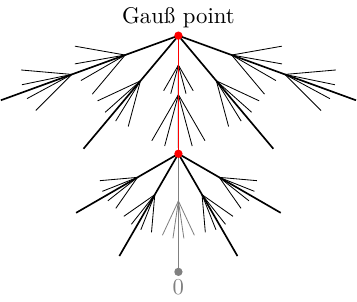}\end{center}
\end{example}

\begin{noname} \label{GroupAction}
	Next we construct a proper strong deformation retraction $\Phi: X \times [0, 1] \to X$.
	Let $\mathscr{B}$ be the $K$-affinoid algebra $K\{T_{10},\dots ,T_{1n_1}, \dots , T_{p0},\dots ,T_{pn_p}\} / \mathfrak{b}$,
	where $\mathfrak{b}$ is the ideal generated by $T_{i0} \cdots T_{in_i} - 1$ for all $i \in \dotsbra{1}{p}$.  
	We denote the $K$-affinoid space $G := \mathscr{M}(\mathscr{B})$, which has a canonical structure of a $K$-analytic group induced by
	the multiplication of coordinates. We call $G$ the \emph{$K$-affinoid torus} and note that it is the generic fiber of
	$\mathfrak{S}(\textbf{n}, \textbf{b})$ with $\textbf{b} = (1, \dots, 1)$.
	There is an obvious $K$-analytic group action $\mu: G \times X \to X$. 
	For every $t \in [0, 1]$ consider the $K$-analytic subgroups
	\[G_t := \bracket{x \in G}{|(T_{ij} - 1)(x)| \leq t \text{ for all } i \in \dotsbra{1}{p} \text{ and } j \in \dotsbra{0}{n_i}}.\]
	Each subgroup $G_t$ has a unique maximal point, which we will denote by $g_t$. Since these $g_t$ are peaked points, 
	we can consider for all $x \in X$ the $\ast$-product $g_t \ast x$ as introduced in \cite[p.\,100]{Ber90}.
	This provides us a continuous map $\Phi: X \times [0, 1] \rightarrow X$, $(x, t) \mapsto g_t \ast x$.
	For details see \cite[§\,5.2 and §\,6.1]{Ber90}.  
	
	In this situation, the $\ast$-product can be explicitly described, namely if 
	$f =\sum_{\boldsymbol{\mu}} c_{\boldsymbol{\mu}} T^{\boldsymbol{\mu}}$ we have
	\[|f(g_t \ast x)| = \max\limits_{\boldsymbol{\nu}} |\partial_{\boldsymbol{\nu}}f(x)|t^{\boldsymbol{\nu}} 
	\text{, where } \partial_{\boldsymbol{\nu}}f := \sum\limits_{\boldsymbol{\mu} \geq 
		\boldsymbol{\nu}} \binom{\boldsymbol{\mu}}{\boldsymbol{\nu}} c_{\boldsymbol{\mu}}T^{\boldsymbol{\mu}}\]
	with $\boldsymbol{\nu} = \dotsroundbra{\nu_{10}}{\nu_{pn_p}}$ and the usual multi-index notations. 
	It now becomes easy to check that $\Phi(\cdot, 0) = \id_{X}$, 
	$\Phi(\cdot, 1) = \tau$ and $\Phi(\cdot, t)|_S = \id_{S}$ for all $t \in [0, 1]$. 
	This means that $\Phi$ is a strong deformation retraction of $X$ onto $S$. 
	It is clear from $X$ being a compact space, that $\Phi$ is proper. 
\end{noname}

\begin{proposition} \label{InjProp}
	We use the same notations as above. Then for every $x \in X$ there exists an element
	$t' \in [0, 1]$ such that $x = \Phi(x, t)$ for all $t \in [0, t']$ and the 
	map $[t', 1] \to X$, $t \mapsto \Phi(x, t)$ is injective.
\end{proposition}

\begin{proof}
	We use the explicit description of $\Phi(x, t) = g_t \ast x$ from above.
	Pick $t'$ to be the maximum value for $t \in [0, 1]$ such that 
	$|\partial_{\boldsymbol{\nu}}f(x)|t^{\boldsymbol{\nu}} \leq |f(x)|$
	for all $\boldsymbol{\nu}$ and $f \in \mathscr{A}$.
	One can check that this $t'$ satisfies the claim.
\end{proof}

\begin{noname} \label{BuildingBlockSkeleton}
	Suppose we are given a formal scheme $\mathfrak{S} := \mathfrak{S}(\textbf{n}, \textbf{a})$, where 
	$p \in \Nn$, $\textbf{n} = (n_1,\dots,n_p)$ and $\textbf{a} = (a_1,\dots,a_p)$ with $n_i \in \Nn$ and
	$a_i \in K^{\circ}$ for all $i \in \dotsbra{1}{p}$.
	We consider $\mathfrak{G} := \Spf(K^{\circ}\{T_{10},\dots ,T_{1n_1}, \dots , T_{p0},\dots ,T_{pn_p}\} / \mathfrak{b})$,
	where $\mathfrak{b}$ is the ideal generated by $T_{i0},\dots ,T_{in_i} - 1$ for all $i \in \dotsbra{1}{p}$.
	Then $\mathfrak{G}$ can be considered as a group object in the category of formal $K^{\circ}$-schemes.
	Its generic fiber is the analytic group $G$ as introduced in \ref{GroupAction}. 
	We denote by $\widehat{\mathfrak{G}}$ the formal completion of $\mathfrak{G}$ along its unit.
	The formal scheme $\widehat{\mathfrak{G}}$ is a formal group and it acts in the evident way on $\mathfrak{S}$. 
	Its generic fiber is $\widehat{\mathfrak{G}}_{\eta} = \bigcup_{t < 1} G_t$.
	
	Now let $\mathfrak{X} = \Spf(A)$ be an admissible affine formal scheme over $K^{\circ}$ and
	$\varphi: \mathfrak{X} \to \mathfrak{S}$ an \et morphism. We denote $X := \mathfrak{X}_{\eta} = \mathscr{M}(\mathscr{A})$,
	where $\mathscr{A} := A \otimes_{k^{\circ}} k$.
	By \cite[Lemma 5.5]{Ber99} the formal group action of $\widehat{\mathfrak{G}}$
	on $\mathfrak{S}$ extends uniquely to an action of $\widehat{\mathfrak{G}}$ on $\mathfrak{X}$.
	In particular this gives rise to an analytic group action of $\widehat{\mathfrak{G}}_{\eta}$ on $X$
	extending the one on $\mathfrak{S}_{\eta}$. 
	
	We now define a map $\tau: X \rightarrow X$ by sending a point $x \in X$ to the point $\tau(x)$ given by
	\[|f(\tau(x))| := \sup\limits_{t < 1} |f(g_t \ast x)|  \]
	for all $f \in \mathscr{A}$. It is clear that this yields a bounded seminorm on $\mathscr{A}$. 
	One can show that this seminorm is also multiplicative and thus $\tau$ is well-defined.
	
	Next we consider the map 
	$\Phi_{\mathfrak{X}}: X \times  [0, 1] \rightarrow X$,
	\[ (x, t) \mapsto \left\{ 
	\begin{array}{cl}
	g_t \ast x & \text{, if } t < 1 \\
	\tau(x) & \text{, else.}  	
	\end{array}\right.\]
	We define the \emph{skeleton} 
	$S(\mathfrak{X})$ of $\mathfrak{X}$ to be the image of the map $\tau$. One can show that the maps $\tau$ and $\Phi_{\mathfrak{X}}$
	are continuous and do not depend on the choice of the \et morphism $\varphi$, that $S(\mathfrak{X})$ 
	is a closed subset of $X$ and that $\Phi_{\mathfrak{X}}$ is a proper strong deformation retraction
	from $X$ onto $S(\mathfrak{X})$, see \cite[§5, Steps 5--7]{Ber99}. 
	
	We also note that if $\mathfrak{Y}$ is an admissible affine formal scheme over $K^{\circ}$ and
	$\psi: \mathfrak{Y} \to \mathfrak{X}$ an \et morphism, then \cite[Lemma 5.5]{Ber99} shows that
	$\psi_{\eta} \circ \Phi_{\mathfrak{Y}} = \Phi_{\mathfrak{X}} \circ (\psi_{\eta} \times \id)$.
	
	A particularly important special case used in the proof of these claims is the following:
\end{noname}

\begin{lemma} \label{BuildingBlockLemma} 
	Let $\varphi: \mathfrak{X} \to \mathfrak{S}$ be a building block, 
	in particular $\varphi_s$ induces a bijection $\irr(\mathfrak{X}_s) \xrightarrow{\sim} \irr(\mathfrak{S}_s)$.
	Then $S(\mathfrak{X}) = \varphi_{\eta}^{-1}(S(\mathfrak{S}))$ and $\varphi_{\eta}$ induces a homeomorphism 
	$S(\mathfrak{X}) \xrightarrow{\sim} S(\mathfrak{S})$.
	
	Let $x$ resp. $\zeta$ be the generic point of the minimal stratum of $\mathfrak{X}_s$ resp. $\mathfrak{S}_s$.
	Then the homeomorphism from above restricts to a homeomorphism
	\[S(\mathfrak{X}) \cap \red^{-1}_{\mathfrak{X}}(x) \xrightarrow{\sim} S(\mathfrak{S}) \cap \red^{-1}_{\mathfrak{S}}(\zeta). \] 
	
	Moreover there is a canonical bijection between $\irr(\mathfrak{X}_s)$ and the \emph{vertices} of $S(\mathfrak{X})$,
	by which we mean the points in $S(\mathfrak{X})$ corresponding to the vertices of 
	$S(\mathfrak{S}) = \Delta(\textbf{n}, \textbf{r})$ under $\varphi_{\eta}$. 
	The set of vertices of $S(\mathfrak{X})$ is independent of $\mathfrak{S}$ and $\varphi$.     
\end{lemma}

\begin{proof}
	For the first two claims see \cite[§5, Step 6]{Ber99}. 
	For the remaining claims concerning the vertices, consider the generic point $y$ of
	an irreducible component of $\mathfrak{X}_s$. Then $S(\mathfrak{X}) \cap \red^{-1}_{\mathfrak{X}}(y)$ 
	consists of a single point. The reason for this is, that $\xi := \varphi_s(y)$ is the generic point 
	of an irreducible component of $\mathfrak{S}_s$ and $\varphi^{-1}_s(\xi) = \{y\}$.
	Therefore we can conclude with $\red_{\mathfrak{S}} \circ \varphi_{\eta} = \varphi_s \circ \red_{\mathfrak{X}}$ 
	that $\varphi_{\eta}$ induces a homeomorphism 
	$S(\mathfrak{X}) \cap \red^{-1}_{\mathfrak{X}}(y) \xrightarrow{\sim} S(\mathfrak{S}) \cap \red^{-1}_{\mathfrak{S}}(\xi)$
	and we know that $S(\mathfrak{S}) \cap \red^{-1}_{\mathfrak{S}}(\xi)$ consists of a vertex point of $S(\mathfrak{S})$. 
\end{proof}

\begin{noname}
	In the upcoming step, we want to globalize these constructions and extend it to arbitrary poly-stable formal schemes. 
	Let $\mathfrak{X}$ be a poly-stable formal scheme over $K^{\circ}$. We denote $X := \mathfrak{X}_{\eta}$.
	By definition there exists a surjective \et morphism 
	$\psi: \mathfrak{Y} \rightarrow \mathfrak{X}$ where $\mathfrak{Y}$ is strictly poly-stable, meaning that every point in $\mathfrak{Y}$ 
	has an open neighborhood that admits an \et morphism to a standard scheme.
	Now we want to define a map $\Phi: X \times [0, 1] \rightarrow X$ as follows:
	
	Let $x \in X$ and $t \in [0, 1]$. Choose any point $y \in \mathfrak{Y}_{\eta}$ with $\psi_{\eta}(y) = x$,
	which exists by Lemma \ref{FactorMap}.
	Then we can choose an open affine subscheme $\mathfrak{U}$ of $\mathfrak{Y}$ such that $y \in \mathfrak{U}_{\eta}$ and
	which admits an \et morphism to a formal scheme as in \ref{BuildingBlockSkeleton}. In particular we have a deformation retraction
	$\Phi_{\mathfrak{U}}$ of $\mathfrak{U}_{\eta}$ onto $S(\mathfrak{U})$ and we set $\Phi(x, t) := \psi_{\eta}(\Phi_{\mathfrak{U}}(y, t))$. 
	
	This yields a well-defined continuous proper map $\Phi: X \times [0, 1] \rightarrow X$
	which is independent of the choice of $\mathfrak{Y}$ and $\psi$. 
	We use this to define a continuous map $\tau: X \rightarrow X$ by $\tau(x) := \Phi(x, 1)$.
	Furthermore we define the \emph{skeleton of $\mathfrak{X}$} to be $S(\mathfrak{X}) := \tau(X)$.
	According to \cite[Theorem 5.2]{Ber99} the skeleton $S(\mathfrak{X})$ 
	is a closed subset of $X$ and $\Phi$ is a proper strong deformation retraction
	of $X$ onto $S(\mathfrak{X})$. 
\end{noname}

\begin{example} \label{PatchTogether}
	Let $\mathfrak{X}$ be a strictly poly-stable formal scheme over $K^{\circ}$. Then there exists a cover
	$(\mathfrak{U}_i)_{i \in I}$ of $\mathfrak{X}$ such that every $\mathfrak{U}_i$ is of the form discussed
	in \ref{BuildingBlockSkeleton}. According to the above construction, where we set $\mathfrak{Y}$ to be 
	the disjoint union of all the $\mathfrak{U}_i$, the skeleton $S(\mathfrak{X})$ is equal to the union
	of all skeletons $S(\mathfrak{U}_i)$. 
\end{example}

\begin{example} \label{BallSkeleton}
	Let $\mathfrak{X} := \mathfrak{S}(\textbf{n}, \textbf{a}, d)$ be a standard scheme.
	We can use \ref{ToriCover} to get an open cover of $\mathfrak{X}$ 
	by copies of $\mathfrak{T} := \mathfrak{S}(\textbf{n}, \textbf{a}) \times \mathfrak{T}(d, 1)$ and  
	it is easy to see, that the skeleton $S(\mathfrak{X})$ is equal to the image of 
	$S(\mathfrak{T})$ under any of the induced Laurent domain embeddings 
	$\mathfrak{T}_{\eta} \to \mathfrak{X}_{\eta}$, in fact the restrictions of these embeddings to 
	$S(\mathfrak{T})$ are equal. In particular
	$S(\mathfrak{X})$ is homeomorphic to $\Delta(\textbf{n}, \textbf{r})$, where 
	$\textbf{r} := \val(\textbf{a})$. We denote the map 
	$\mathfrak{X}_{\eta} \to \Delta(\textbf{n}, \textbf{r})$ which is obtained by
	composition of the retraction $\mathfrak{X}_{\eta} \to S(\mathfrak{X})$ with the homeomorphism
	$\trop: S(\mathfrak{T}) \xrightarrow{\sim} \Delta(\textbf{n}, \textbf{r})$ 
	by $\trop$ as well. 
\end{example}

\begin{example} \label{BlockBallSkeleton}
	Let $\mathfrak{S} := \mathfrak{S}(\textbf{n}, \textbf{a}, d)$ be a standard scheme,
	$\mathfrak{X} = \Spf(A)$ an admissible affine formal scheme over $K^{\circ}$ and
	$\varphi: \mathfrak{X} \to \mathfrak{S}$ a building block.
	We consider the open cover of $\mathfrak{S}$ as provided in the example above by copies of $\mathfrak{T}$. 
	Then the base changes of $\varphi$ with respect to these open immersions $\mathfrak{T} \to \mathfrak{S}$
	give an open cover of $\mathfrak{X}$. By Example \ref{PatchTogether} and Lemma \ref{BuildingBlockLemma} 
	we can conclude that the skeleton $S(\mathfrak{X})$ is equal to $\varphi^{-1}(S(\mathfrak{S}))$
	and that $\varphi_{\eta}$ induces a homeomorphism between $S(\mathfrak{X})$ and $S(\mathfrak{S})$.
	Moreover this homeomorphism restricts to 
	$S(\mathfrak{X}) \cap \red^{-1}_{\mathfrak{X}}(x) \xrightarrow{\sim} S(\mathfrak{S}) \cap \red^{-1}_{\mathfrak{S}}(\zeta)$
	just like in Lemma \ref{BuildingBlockLemma}. 	
\end{example}


\subsection{Standard pairs}

Let $(\mathfrak{S}, \mathfrak{G}) := (\mathfrak{S}(\textbf{n}, \textbf{a}, d), \mathfrak{G}(s))$ be a standard pair
with $p$-tuples $\textbf{n},\textbf{a}$ and we denote $\textbf{r} := \val(\textbf{a})$. 
We assume in this subsection that $K$ is non-trivially valued.

\begin{noname} \label{SkelConstr1}
	We will adapt the $\varepsilon$-approximation procedure from \cite[4.2]{GRW16} to define 
	the skeleton of our standard pair.
	Let $\varepsilon \in |K^*|$ with $\varepsilon \leq 1$ and choose $b_{\varepsilon} \in K^*$ such that $|b_{\varepsilon}|=\varepsilon$.	
	We consider the formal scheme 
	$\mathfrak{S}_{\varepsilon} := \mathfrak{S}(\textbf{n}, \textbf{a}) \times \mathfrak{T}(1,b_{\varepsilon})^s \times \mathfrak{B}^{d-s}$
	with generic fiber $S_{\varepsilon}$ and associate with it the skeleton $S(\mathfrak{S}_{\varepsilon})$ as 
	in Example \ref{BallSkeleton}.
	These $S_{\varepsilon}$ are identified with subsets of 
	$Z := \mathfrak{S}_{\eta} \setminus \mathfrak{G}_{\eta} = 
	\mathfrak{S}(\textbf{n}, \textbf{a})_{\eta} \times (B \setminus \{0\})^s \times B^{d-s}$ via the affinoid domain embedding 
	$\mathfrak{T}(1,b_{\varepsilon})_{\eta} \to B$ induced by $T \mapsto T_1$, where $T$ denotes the coordinate of $\mathfrak{B}$ and $T_0$, $T_1$
	denote the coordinates of $\mathfrak{T}(1,b_{\varepsilon})$. 
	Then the $S_{\varepsilon}$ do not depend 
	on the choice of $b_{\varepsilon}$ but only on $\varepsilon$.
	For every $\varepsilon' \in |K^*|$ with $\varepsilon' \leq 1$
	and $\varepsilon' \leq \varepsilon$ we obviously have $S_{\varepsilon} \subseteq S_{\varepsilon'}$ and 
	$S(\mathfrak{S}_{\varepsilon}) = S(\mathfrak{S}_{\varepsilon'}) \cap S_{\varepsilon}$. 
	Applying $\trop$ to 
	$\mathfrak{T}(1,b_{\varepsilon})_{\eta} \subseteq \mathfrak{T}(1,b_{\varepsilon'})_{\eta}$ gives the inclusion 
	\[\Delta(1, -\log(\varepsilon)) \hookrightarrow \Delta(1, -\log(\varepsilon'))~,~(x_0, x_1) \mapsto (x_0 + \log(\varepsilon)-\log(\varepsilon'), x_1).\]
	Furthermore the union of the $S_{\varepsilon}$ over all $\varepsilon \in |K^*|$ 
	with $\varepsilon \leq 1$ is equal to $Z$. The union of all $\Delta(1, -\log(\varepsilon))$ is homeomorphic to
	$\R_{\geq 0}$ by projecting $(x_0, x_1) \mapsto x_1$.   
\end{noname}	

\begin{definition}
	We define the \emph{extended skeleton} $S(\mathfrak{S}, \mathfrak{G})$ of our standard pair as the union of the skeletons 
	$S(\mathfrak{S}_{\varepsilon})$ as subsets of $Z$ over all $\varepsilon$ as above.
	In the following we will often just say “skeleton” instead of “extended skeleton”, but it is clear that when
	we talk about the skeleton of a pair, that this means the extended skeleton
	as opposed to the classical skeleton, which is defined for a single formal scheme.   
\end{definition}

\begin{example}
	The graphic below illustrates this $\varepsilon$-approximation procedure in the case of
	$(\mathfrak{S}, \mathfrak{G}) := (\mathfrak{S}(1, a, 1), \mathfrak{G}(1))$, also see Example~\ref{AnnulusSkeleton}.
	The length of the horizontal segments of $S(\mathfrak{S}_{\varepsilon})$ is $-\log(\varepsilon)$.
	\begin{center}\includegraphics[scale=1]{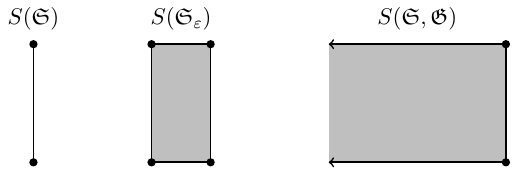}\end{center}
\end{example}

\begin{proposition} \label{SkeletonClosed}
	$S(\mathfrak{S}, \mathfrak{G})$ is a closed subset of $Z$.
\end{proposition}	

\begin{proof}
	Consider the subset $Y := \{|\enspace|_v ~|~ v \in [0, \infty]\} \subseteq B = \mathscr{M}(K\{T\})$.
	Here we define for $v \in \R_{\geq 0}$ the seminorm $|\enspace|_v$ as being given by   
	$|f|_v = \max_m |c_m| \exp(-vm)$ where $f = \sum_m c_mT^m$.
	The seminorm $|\enspace|_{\infty}$ is defined by $|f|_{\infty} = |f(0)|$,
	which corresponds to $0 \in B$.
	Note that $Y$ is closed in $B$. Indeed this follows from 
	$Y = B \setminus \bigcup_{a \in K^{\circ}\setminus\{0\}} B^{\circ}(a, |a|)$. 
	We conclude that $Y \setminus \{|\enspace|_{\infty}\}$ is closed in $B \setminus \{0\}$,
	which also implies our claim.
\end{proof}

\begin{noname} \label{StdDeform}
	For every $\varepsilon$ we have a map 
	$\trop_{\varepsilon}: S_{\varepsilon} \to \Delta_{\varepsilon} := \Delta(\textbf{n}, \textbf{r}) \times \Delta(1, -\log(\varepsilon))^s$,
	its canonical right inverse $\sigma_{\varepsilon}$, the retraction $\tau_{\varepsilon} = \sigma_{\varepsilon} \circ \trop_{\varepsilon}$
	from $S(\mathfrak{S}_{\varepsilon})$ to $\Delta_{\varepsilon}$
	and the proper strong deformation retraction $\Phi_{\varepsilon}: S_{\varepsilon} \times [0, 1] \to S_{\varepsilon}$.
	These $\trop_{\varepsilon}$ extend to a continuous map
	\[\trop: Z \to \bigcup\limits_{\varepsilon \in |K^*| \cap (0,1)} \Delta_{\varepsilon}
	= \Delta(\textbf{n}, \textbf{r}, s),\] 
	whose restriction to $S(\mathfrak{S}, \mathfrak{G})$ is a homeomorphism. 
	We denote its inverse by $\sigma$. Then $\tau := \sigma \circ \trop$ is a retraction 
	from $Z$ to $S(\mathfrak{S}, \mathfrak{G})$ extending the retractions $\tau_{\varepsilon}$. Similarly we obtain a deformation retraction
	$\Phi: Z \times [0,1] \to Z$ from $Z$ onto $S(\mathfrak{S}, \mathfrak{G})$ extending the homotopies $\Phi_{\varepsilon}$. That these are compatible, follows immediately from their explicit description.
	
	From the construction we get an explicit description of 
	$\trop: S(\mathfrak{S}, \mathfrak{G}) \xrightarrow{\sim} \Delta(\textbf{n}, \textbf{r}, s)$: 
	We use our usual coordinates for the standard scheme. Then $\mathfrak{G}$ is equal to $\{T_1 \cdots T_s = 0\}$ 
	and for every $x \in S(\mathfrak{S}, \mathfrak{G})$ we have
	\begin{align*}
	\trop(x) = (&-\log|T_{10}(x)|,\dots,-\log|T_{1n_1}(x)|, \dots ,-\log|T_{p0}(x)|,\dots,-\log|T_{pn_p}(x)|, \\
	&-\log|T_1(x)|,\dots,-\log|T_s(x)|).
	\end{align*}	 	
\end{noname}


\begin{proposition} \label{StdRedExtended}
	We denote $(\mathfrak{S}, \mathfrak{G}) := (\mathfrak{S}(\textbf{n}, \textbf{a}, d), \mathfrak{G}(s))$.
	Let $x$ be the generic point of the minimal stratum of $(\mathfrak{S}_s, \mathfrak{G}_s)$.
	Then the homeomorphism $\trop: S(\mathfrak{S}, \mathfrak{G}) \xrightarrow{\sim} \Delta(\textbf{n}, \textbf{r}, s)$
	from above retricts to a homeomorphism
	\[S(\mathfrak{S}, \mathfrak{G}) \cap \red^{-1}_{\mathfrak{S}}(x) \xrightarrow{\sim} \Delta^{\circ}(\textbf{n}, \textbf{r}, s). \]
\end{proposition}

\begin{proof}
	This is an easy exercise requiring nothing but the definitions
	and the same arguments as in the proof of Proposition~\ref{StdRed}.
\end{proof}


\subsection{Building blocks}

We continue to assume that $K$ is non-trivially valued.

Again the $\varepsilon$-approximation method helps us to translate 
results from the classical case to our situation. In particular
we show that the homeomorphism between the building block skeleton
and the extended poly-simplex via the tropicalization map is 
in some sense “combinatorial”, see Proposition~\ref{CoordChange}.

\begin{noname} \label{ExtendedStep5}
	Let $\mathfrak{X}$ be an affine admissible formal $K^{\circ}$-scheme,
	$\mathfrak{H} \subseteq \mathfrak{X}$ a closed subscheme,
	$(\mathfrak{S}, \mathfrak{F}) := (\mathfrak{S}(\textbf{n}, \textbf{a}, d), \mathfrak{G}(s))$  
	and $\varphi: (\mathfrak{X}, \mathfrak{H}) \to (\mathfrak{S}, \mathfrak{F})$ an \et morphism. 
	Denote $Z := \mathfrak{X}_{\eta} \setminus \mathfrak{H}_{\eta}$.
	We set $\varepsilon$ and $\mathfrak{S}_{\varepsilon}$ as in \ref{SkelConstr1}. 
	Using the same technique as in Example \ref{BlockBallSkeleton}, we may assume $d = s$.
	Consider the base change $\varphi_{\varepsilon}: \mathfrak{X}_{\varepsilon} \to \mathfrak{S}_{\varepsilon}$
	of $\varphi$ with respect to $\mathfrak{S}_{\varepsilon} \to \mathfrak{S}$ which we obtain via
	$\mathfrak{T}(1,b_{\varepsilon}) \to \mathfrak{B}$ induced by $T \mapsto T_1$. 
	Now let $\varepsilon' \in |K^*|$ with $\varepsilon' \leq 1$ and $\varepsilon' \leq \varepsilon$.
	We have a morphism $\mathfrak{T}(1,b_{\varepsilon}) \to \mathfrak{T}(1,b_{\varepsilon'})$
	induced by $T_0 \mapsto \frac{b_{\varepsilon'}}{b_{\varepsilon}} T_0$ and $T_1 \mapsto T_1$.
	It gives rise to morphisms $\mathfrak{S}_{\varepsilon} \to \mathfrak{S}_{\varepsilon'}$ and 
	$\iota: \mathfrak{X}_{\varepsilon} \to \mathfrak{X}_{\varepsilon'}$.
	Have a look at the following diagram of formal group actions by $\widehat{\mathfrak{G}}$
	as described in \ref{BuildingBlockSkeleton}:
	\begin{center}
		\begin{tikzpicture}[scale = 2]
		
		\node (1) at (0, 1) {$\widehat{\mathfrak{G}} \times \mathfrak{X}_{\varepsilon}$};
		\node (2) at (1, 1) {$\mathfrak{X}_{\varepsilon}$};
		\node (3) at (0, 0) {$\widehat{\mathfrak{G}} \times \mathfrak{X}_{\varepsilon'}$};
		\node (4) at (1, 0) {$\mathfrak{X}_{\varepsilon'}$};		
		
		\draw[->, thick] (1)--(2); 
		\draw[->, thick] (1)--(3) node[pos=0.5, left] {$\id \times \iota$};
		\draw[->, thick] (2)--(4) node[pos=0.5, right] {$\iota$};
		\draw[->, thick] (3)--(4); 
		
		\end{tikzpicture}
	\end{center}
	The explicit description of the extension of the formal group action from 
	$\mathfrak{S}_{\varepsilon}$ to $\mathfrak{X}_{\varepsilon}$, obtained in the proof of \cite[Lemma 5.5]{Ber99},
	shows that the diagram commutes. We can now apply \cite[Proposition 5.2.8 (ii)]{Ber90}
	and infer $\iota_{\eta}(g_t \ast y) = g_t \ast \iota_{\eta}(y)$ for all $y \in \mathfrak{X}_{\varepsilon, \eta}$
	and $t \in [0, 1)$. Note that $\iota_{\eta}: \mathfrak{X}_{\varepsilon, \eta} \to \mathfrak{X}_{\varepsilon', \eta}$ 
	is a Laurent domain embedding and 
	$Z = \bigcup_{\varepsilon > 0} \mathfrak{X}_{\varepsilon, \eta}$. 
	
	In particular we can glue together all homotopies 
	$\Phi_{\mathfrak{X}_{\varepsilon}}: \mathfrak{X}_{\varepsilon, \eta} \times [0, 1] \to \mathfrak{X}_{\varepsilon, \eta}$
	to obtain a homotopy $\Phi_{(\mathfrak{X}, \mathfrak{H})}: Z \times [0, 1] \to Z$. One easily concludes from \ref{BuildingBlockSkeleton},
	that $\Phi_{(\mathfrak{X}, \mathfrak{H})}$ is a proper strong deformation retraction from $Z$ onto 
	$\bigcup_{\varepsilon > 0} S(\mathfrak{X}_{\varepsilon})$ and that it is
	independent of the choice of $(\mathfrak{S}, \mathfrak{F})$ and $\varphi$.
	
	Moreover if $\mathfrak{X}'$ is an admissible affine formal $K^{\circ}$-scheme
	with a closed subset $\mathfrak{H}' \subseteq \mathfrak{X}'$ and an \et morphism
	$\psi: (\mathfrak{X}', \mathfrak{H}') \to (\mathfrak{X}, \mathfrak{H})$, 
	then the composition $\varphi \circ \psi$ induces a homotopy $\Phi_{(\mathfrak{X}', \mathfrak{H}')}$
	which satisfies
	$\psi_{\eta} \circ \Phi_{(\mathfrak{X}', \mathfrak{H}')} = \Phi_{(\mathfrak{X}, \mathfrak{H})} \circ (\psi_{\eta} \times \id)$.
\end{noname}

\noindent Let for the rest of this subsection $(\mathfrak{X}, \mathfrak{H})$ 
be a strictly poly-stable pair and $x \in \str(\mathfrak{X}_s, \mathfrak{H}_s)$.
For convenience we will also write $s$ instead of $s_x$.

\begin{definition}
	Let $\mathfrak{U}$ together with $\varphi: (\mathfrak{U}, \mathfrak{H} \cap \mathfrak{U}) \to (\mathfrak{S}, \mathfrak{G})$ 
	be a building block of $(\mathfrak{X}, \mathfrak{H})$ in $x$, where
	$(\mathfrak{S}, \mathfrak{G}) = (\mathfrak{S}(\textbf{n}, \textbf{a}, d), \mathfrak{G}(s))$. 
	We define the \emph{extended skeleton} $S(\mathfrak{U}, \mathfrak{H} \cap \mathfrak{U})$ of this building block 
	as the preimage of $S(\mathfrak{S}, \mathfrak{G})$ under $\varphi_{\eta}$.
\end{definition}

\begin{proposition} \label{BuildingBlockProp}
	Assume we are in the same situation as in the above definition and we denote $Z := \mathfrak{U}_{\eta} \setminus \mathfrak{H}_{\eta}$.
	Let us denote the generic point of the minimal stratum of $(\mathfrak{S}_s, \mathfrak{G}_s)$ by $\zeta$.
	Then the following statements hold:
	\begin{enumerate}
		\item $S(\mathfrak{U}, \mathfrak{H} \cap \mathfrak{U})$ is a closed subset of $Z$. 
		\item $\varphi_{\eta}$ induces a homeomorphism 
		$S(\mathfrak{U}, \mathfrak{H} \cap \mathfrak{U}) \xrightarrow{\sim} S(\mathfrak{S}, \mathfrak{G})$.
		Moreover it restricts to a homeomorphism
		\[S(\mathfrak{U}, \mathfrak{H} \cap \mathfrak{U}) \cap \red^{-1}_{\mathfrak{U}}(x) 
		\xrightarrow{\sim} S(\mathfrak{S}, \mathfrak{G}) \cap \red^{-1}_{\mathfrak{S}}(\zeta). \] 
		\item There is a retraction $\tau: Z \to S(\mathfrak{U}, \mathfrak{H} \cap \mathfrak{U})$
		and a proper strong deformation retraction $\Phi: Z \times [0,1] \to Z$ 
		onto $S(\mathfrak{U}, \mathfrak{H} \cap \mathfrak{U})$.
		\item $S(\mathfrak{U}, \mathfrak{H} \cap \mathfrak{U})$, $\tau$ and $\Phi$ do not depend on the choice of $\varphi$.
	\end{enumerate} 
\end{proposition}

\begin{proof}
	Since $\varphi^{-1}(\mathfrak{G}) = \mathfrak{H} \cap \mathfrak{U}$ we get that 
	the continuous map $\varphi_{\eta}$ restricts to a map 
	$Z \to \mathfrak{S}_{\eta} \setminus \mathfrak{G}_{\eta}$.
	In \ref{SkeletonClosed} we saw that $S(\mathfrak{S}, \mathfrak{G})$ is a closed subset of 
	$\mathfrak{S}_{\eta} \setminus \mathfrak{G}_{\eta}$, thus (i) follows.
	
	In order to show (ii) we want to apply the $\varepsilon$-approximation procedure from \ref{SkelConstr1}. 
	Let $\varepsilon$, $b_{\varepsilon}$ and $\mathfrak{S}_{\varepsilon}$ be as there and consider the base change 
	$\varphi_{\varepsilon}: \mathfrak{U}_{\varepsilon} \to \mathfrak{S}_{\varepsilon}$ of $\varphi: \mathfrak{U} \to \mathfrak{S}$  
	with respect to the morphism $\mathfrak{S}_{\varepsilon} \to \mathfrak{S}$ which we obtain via
	$\mathfrak{T}(1,b_{\varepsilon}) \to \mathfrak{B}$ induced by $T \mapsto T_1$.
	Observe that there is a closed immersion $\mathfrak{S}_s \to \mathfrak{S}_{\varepsilon, s}$ which is a right inverse
	of $\mathfrak{S}_{\varepsilon, s} \to \mathfrak{S}_s$. 
	The induced maps from $\mathfrak{U}_{\varepsilon, \eta}$ to $\mathfrak{U}_{\eta}$ 
	and from $\mathfrak{S}_{\varepsilon, \eta}$ to $\mathfrak{S}_{\eta}$ are Laurent domain embeddings. 
	
	We claim that $\varphi_{\varepsilon}: \mathfrak{U}_{\varepsilon} \to \mathfrak{S}_{\varepsilon}$ is a building block. 
	First of all note that $x$ is the least element of 
	$\str(\mathfrak{U}_s, (\mathfrak{H} \cap \mathfrak{U})_s)$ and its image under $\varphi_s$ is $\zeta$. 
	However $\zeta$ is also the minimal element in $\str(\mathfrak{S}_{\varepsilon, s})$.
	Now \cite[\href{https://stacks.math.columbia.edu/tag/0555}{Tag 0555}]{stacks-project}
	shows that the fiber of $\mathfrak{U}_{\varepsilon,s} \to \mathfrak{U}_s$ over $x$
	is irreducible, which together with dimensionality reasons implies 
	that there can only be one stratum of $\mathfrak{U}_{\varepsilon,s}$
	lying over $\zeta \in \mathfrak{S}_{\varepsilon, s}$. 
	This is the requested minimal stratum.
	
	Using Example \ref{BlockBallSkeleton} we infer that the skeleton $S(\mathfrak{U}_{\varepsilon})$
	is equal to $(\varphi_{\varepsilon, \eta})^{-1}(S(\mathfrak{S}_{\varepsilon}))$ 
	and furthermore $\varphi_{\varepsilon, \eta}$ induces a homeomorphism between it and $S(\mathfrak{S}_{\varepsilon})$. 
	These homeomorphisms are compatible and they glue together to give a homeomorphism 
	$S(\mathfrak{U}, \mathfrak{H} \cap \mathfrak{U}) \xrightarrow{\sim} S(\mathfrak{S}, \mathfrak{G})$, which proves the first claim of (ii).
	
	The other claim of (ii) follows from the equality $\red^{-1}_{\mathfrak{U}}(x) = \varphi_{\eta}^{-1}(\red^{-1}_{\mathfrak{S}}(\zeta))$,
	which we get from $\varphi^{-1}_s(\zeta) = \{x\}$, see Proposition \ref{DimensionProp}, and the 
	functoriality of the reduction map, in explicit terms $\red_{\mathfrak{S}} \circ \varphi_{\eta} = \varphi_s \circ \red_{\mathfrak{U}}$.  
	%
	With this we can conclude (ii).  
	
	The statements (iii) and (iv) are consequences of our considerations in \ref{ExtendedStep5}.	
\end{proof}

\begin{remark}
	We continue to use the above notions.
	Let $\Delta(\textbf{n}, \textbf{r}) \hookrightarrow 
	\Delta(\textbf{n}, \textbf{r}, s) = \Delta(\textbf{n}, \textbf{r}) \times \R_{\geq 0}^s$ be
	the inclusion given by $v \mapsto (v, \mathbf{0})$. Then by the above construction the classical skeleton
	$S(\mathfrak{U})$ is contained in $S(\mathfrak{U}, \mathfrak{H} \cap \mathfrak{U})$ such that the following diagram,
	involving the homeomorphisms induced by $\varphi_{\eta}$ and $\trop$, is commutative:
	\begin{center}
		\begin{tikzpicture}[scale = 2]
		\node (1) at (0, 1) {$S(\mathfrak{U})$};
		\node (2) at (1.4, 1) {$S(\mathfrak{S})$};
		\node (3) at (2.8, 1) {$\Delta(\textbf{n}, \textbf{r})$};
		\node (4) at (0, 0) {$S(\mathfrak{U}, \mathfrak{H} \cap \mathfrak{U})$};
		\node (5) at (1.4, 0) {$S(\mathfrak{S}, \mathfrak{G})$};
		\node (6) at (2.8, 0) {$\Delta(\textbf{n}, \textbf{r}, s)$};		
		
		\draw[->, thick] (1)--(2) node[pos=0.5, above] {$\sim$};
		\draw[->, thick] (2)--(3) node[pos=0.5, above] {$\sim$};
		\draw[->, thick] (4)--(5) node[pos=0.5, below] {$\sim$};
		\draw[->, thick] (5)--(6) node[pos=0.5, below] {$\sim$};
		\draw[right hook->, thick] (1)--(4);
		\draw[right hook->, thick] (2)--(5);
		\draw[right hook->, thick] (3)--(6);
		\end{tikzpicture}
	\end{center}	 
\end{remark}

\begin{noname} \label{KeyPoints}
	For later reference we further investigate the irreducible components of $\mathfrak{S}_{\varepsilon, s}$
	from the proof of Proposition~\ref{BuildingBlockProp} and their connection to the components of 
	$\mathfrak{U}_s$ and $(\mathfrak{H} \cap \mathfrak{U})_s$.
	
	We have isometric bijections $[\textbf{n}] \xrightarrow{\sim} \irr(\mathfrak{S}_s)$
	and $[\textbf{n}] \times [1]^s \xrightarrow{\sim} \irr(\mathfrak{S}_{\varepsilon, s})$ 
	as introduced in Example \ref{MetrEx}.
	For every $\textbf{k} \in [\textbf{n}]$ and $i \in \dotsbra{1}{s}$ consider
	the element $\mathcal{W}_{\textbf{k}} \in \irr(\mathfrak{S}_s)$
	corresponding to the tuple $\textbf{k}$ and the element $\mathcal{W}_{\textbf{k}, i} \in \irr(\mathfrak{S}_{\varepsilon, s})$
	corresponding to the tuple $(\textbf{k}, e_i)$, where $e_i$
	is the tuple having $0$ at the $i$-th component and $1$ in every other component. 
	
	Then the image of $\mathcal{W}_{\textbf{k}, i}$ under $\mathfrak{S}_{\varepsilon, s} \to \mathfrak{S}_s$ 
	is equal to $\mathcal{W}_{\textbf{k}} \cap \{T_i = 0\} \in \irr(\mathfrak{G}_s)$. 
	It follows now from Theorem~\ref{BijThm} that the preimage of $\mathcal{W}_{\textbf{k}} \cap \{T_i = 0\}$
	under $\varphi_s$ is equal to 
	$\mathcal{V}_{\alpha_{\varphi}(\textbf{k}), \gamma_{\varphi}(i)} \cap \mathfrak{U}_s \in \irr((\mathfrak{H} \cap \mathfrak{U})_s)$
	with the notation used there. 
	By dimensionality reasons the $\mathcal{W}_{\textbf{k}, i}$ are the only elements in $\irr(\mathfrak{S}_{\varepsilon, s})$
	which yield $\mathcal{V}_{\alpha_{\varphi}(\textbf{k}), \gamma_{\varphi}(i)}$ in this fashion.
	
	We also consider the element $\mathcal{W}_{\textbf{k}, \textbf{1}} \in \irr(\mathfrak{S}_{\varepsilon, s})$
	corresponding to the tuple $(\textbf{k}, \textbf{1})$, where $\textbf{1}$
	is the tuple having $1$ in every component. Then the image of $\mathcal{W}_{\textbf{k}, \textbf{1}}$
	under $\mathfrak{S}_{\varepsilon, s} \to \mathfrak{S}_s$ is equal to $\mathcal{W}_{\textbf{k}} \in \irr(\mathfrak{S}_s)$.  
	The preimage of $\mathcal{W}_{\textbf{k}}$ under $\varphi_s$ is equal to $\alpha_{\varphi}(\textbf{k}) \in \irr(\mathfrak{U}_s)$.
	Again by dimensionality reasons the $\mathcal{W}_{\textbf{k}, \textbf{1}}$ are the only elements in $\irr(\mathfrak{S}_{\varepsilon, s})$
	which yield $\alpha_{\varphi}(\textbf{k})$ in this fashion.      
\end{noname}

\begin{proposition} \label{DistStrIndep}
	The skeleton $S(\mathfrak{U}, \mathfrak{H} \cap \mathfrak{U})$ associated to a building block $\mathfrak{U}$
	together with $\varphi: (\mathfrak{U}, \mathfrak{H} \cap \mathfrak{U}) \to (\mathfrak{S}, \mathfrak{G})$ 
	depends only on the generic point of the minimal stratum of $\mathfrak{U}_s$. 
\end{proposition}

\begin{proof}
	Let $\mathfrak{U}'$ together with $\varphi': (\mathfrak{U}', \mathfrak{H} \cap \mathfrak{U}') \to (\mathfrak{S}', \mathfrak{G}')$  
	be another building block of $(\mathfrak{X}, \mathfrak{H})$ such that the generic points of the minimal strata of 
	$\mathfrak{U}_s$  and $\mathfrak{U}'_s$ agree. Since we can work with the intersection $\mathfrak{U} \cap \mathfrak{U}'$, 
	it is enough to consider the case that $\mathfrak{U}' \subseteq \mathfrak{U}$.
	Let us show that $S(\mathfrak{U}, \mathfrak{H} \cap \mathfrak{U}) = S(\mathfrak{U}', \mathfrak{H} \cap \mathfrak{U}')$. 
	Observe that $\mathfrak{U}'$ is also a building block via the composition $\psi: \mathfrak{U}' \hookrightarrow \mathfrak{U} \to \mathfrak{S}$
	using $\varphi$. The skeleton $S(\mathfrak{U}', \mathfrak{H} \cap \mathfrak{U}')$ does not depend on the choice of the \et morphism, as 
	stated in Proposition \ref{BuildingBlockProp} (iv). Since $\psi_{\eta}$ resp. $\varphi_{\eta}$ induces a homeomorphism 
	$S(\mathfrak{U}', \mathfrak{H} \cap \mathfrak{U}') \xrightarrow{\sim} S(\mathfrak{S}, \mathfrak{G})$ 
	resp. $S(\mathfrak{U}, \mathfrak{H} \cap \mathfrak{U}) \xrightarrow{\sim} S(\mathfrak{S}, \mathfrak{G})$
	and $\psi_{\eta}$ is the restriction of $\varphi_{\eta}$, we finished the proof.  	
\end{proof}

\begin{proposition} \label{CoordChange}
	Let $\mathfrak{U}$ resp. $\mathfrak{U}'$ be a building block of $(\mathfrak{X}, \mathfrak{H})$ in $x$ together with 
	$\varphi: (\mathfrak{U}, \mathfrak{H} \cap \mathfrak{U}) \to (\mathfrak{S}, \mathfrak{G}(s))$ resp.
	$\varphi': (\mathfrak{U}', \mathfrak{H} \cap \mathfrak{U}') \to (\mathfrak{S}', \mathfrak{G}'(s))$.
	Then the homeomorphism
	\[h: \Delta(\textbf{n}, \textbf{r}, s) \xrightarrow{\sim} S(\mathfrak{S}, \mathfrak{G}(s)) \xrightarrow{\sim} 
	S(\mathfrak{U}, \mathfrak{H} \cap \mathfrak{U}) = S(\mathfrak{U}', \mathfrak{H} \cap \mathfrak{U}') \xrightarrow{\sim} S(\mathfrak{S}', \mathfrak{G}'(s))
	\xrightarrow{\sim} \Delta(\textbf{n}', \textbf{r}', s) \]
	via $\varphi_{\eta}$, $\varphi'_{\eta}$ and $\trop$ is the geometric realization of
	the isomorphism $h_{C_{\varphi}, C_{\varphi'}}$ induced by the combinatorial charts
	$C_{\varphi} = ([\mathbf{n}, \mathbf{r}, s], \alpha_{\varphi}, \gamma_{\varphi})$
	and $C_{\varphi'} = ([\mathbf{n}', \mathbf{r}', s], \alpha_{\varphi'}, \gamma_{\varphi'})$ in $x$,
	see Proposition \ref{ChartIsom}.
\end{proposition}

\begin{proof}
	We have $S(\mathfrak{U}, \mathfrak{H} \cap \mathfrak{U}) 
	= S(\mathfrak{U} \cap \mathfrak{U}', \mathfrak{H} \cap \mathfrak{U} \cap \mathfrak{U}') 
	= S(\mathfrak{U}', \mathfrak{H} \cap \mathfrak{U}')$ and $\varphi_{\eta}$ resp. $\varphi'_{\eta}$
	restrict to homeomorphisms 
	$S(\mathfrak{U} \cap \mathfrak{U}', \mathfrak{H} \cap \mathfrak{U} \cap \mathfrak{U}') \xrightarrow{\sim} S(\mathfrak{S}, \mathfrak{G}(s))$
	resp. 
	$S(\mathfrak{U} \cap \mathfrak{U}', \mathfrak{H} \cap \mathfrak{U} \cap \mathfrak{U}') \xrightarrow{\sim} S(\mathfrak{S}', \mathfrak{G}'(s))$.
	Moreover the restriction of $\varphi$ resp. $\varphi'$ to $\mathfrak{U} \cap \mathfrak{U}'$ still induces the
	combinatorial charts $C_{\varphi}$ resp. $C_{\varphi'}$.   
	Consequently we may assume that $\mathfrak{U} = \mathfrak{U}'$. 
	
	Again we make use of the $\varepsilon$-approximation argument, which helped us in the proof of Proposition~\ref{BuildingBlockProp}.
	The homeomorphism $h$ from above is obtained by gluing together the homeomorphisms
	\[h_{\varepsilon}: \Delta_{\varepsilon} \xrightarrow{\sim} S(\mathfrak{S}_{\varepsilon}) \xrightarrow{\sim} 
	S(\mathfrak{U}_{\varepsilon}) \xrightarrow{\sim} S(\mathfrak{S}'_{\varepsilon})
	\xrightarrow{\sim} \Delta'_{\varepsilon}.\] 
	
	As we have seen $\varphi_{\varepsilon}: \mathfrak{U}_{\varepsilon} \to \mathfrak{S}_{\varepsilon}$ and
	$\varphi'_{\varepsilon}: \mathfrak{U}_{\varepsilon} \to \mathfrak{S}'_{\varepsilon}$ are building blocks. 
	This is the classical situation treated by Berkovich and it follows from \cite[§\,5, Step~13]{Ber99}
	that $h_{\varepsilon}$ respects the affine linear functions as introduced in Definition~\ref{AffLinFct}.
	We can then infer from \cite[Lemma~4.1]{Ber99} that $h_{\varepsilon}$ is the geometric realization of an
	isomorphism of colored poly-simplices. 
	Consequently $h_{\varepsilon}$ is completely determined by the images of the vertices of $\Delta_{\varepsilon}$. 
	The vertices of $\Delta_{\varepsilon}$ resp. $\Delta'_{\varepsilon}$ correspond to the irreducible components of 
	$\mathfrak{S}_{\varepsilon, s}$ resp. $\mathfrak{S}'_{\varepsilon, s}$. 
	Moreover the vertices of $S(\mathfrak{U}_{\varepsilon})$ correspond to the irreducible components of 
	$\mathfrak{U}_{\varepsilon, s}$, see Lemma~\ref{BuildingBlockLemma}.
	We conclude from our considerations in~\ref{KeyPoints} that $h_{\varepsilon}$ agrees with the geometric realization 
	of an isomorphism $[\textbf{n}] \times [1]^s \xrightarrow{\sim} [\textbf{m}] \times [1]^s$ 
	which maps $(\textbf{k}, e_i)$ to 
	$((\alpha^{-1}_{\varphi'} \circ \alpha_{\varphi})(\textbf{k}), e_{(\gamma^{-1}_{\varphi'} \circ \gamma_{\varphi})(i)})$ 
	and $(\textbf{k}, \textbf{1})$ to 
	$((\alpha^{-1}_{\varphi'} \circ \alpha_{\varphi} )(\textbf{k}), \textbf{1})$ 
	for all $\textbf{k} \in [\textbf{n}]$ and $i \in \dotsbra{1}{s}$. Note that the isomorphism is already uniquely determined
	by the images of these tuples. This shows that the $h_{\varepsilon}$ are compatible and that $h$, which is obtained by
	gluing them together, is given by the data $\alpha^{-1}_{\varphi'} \circ \alpha_{\varphi}$ and $\gamma^{-1}_{\varphi'} \circ \gamma_{\varphi}$.    
	Altogether we have proven that $h$ is the geometric realization of the isomorphism $h_{C_{\varphi}, C_{\varphi'}}$.
\end{proof}


\subsection{Strictly poly-stable pairs}

Let $(\mathfrak{X}, \mathfrak{H})$ again be a strictly poly-stable pair.
We continue to assume that $K$ is non-trivially valued.

In this subsection we explicitly describe the extended skeleton of $(\mathfrak{X}, \mathfrak{H})$ 
and its piecewise linear structure coming from the canonical homeomorphism to the dual intersection complex.
However we will not yet define the deformation retraction, this will be done
more generally for arbitrary poly-stable pairs in the next subsection.

\begin{definition}
	Let $x \in \str(\mathfrak{X}_s, \mathfrak{H}_s)$. Choose any building block $\mathfrak{U}$ together with 
	$\varphi: (\mathfrak{U}, \mathfrak{H} \cap \mathfrak{U}) \to (\mathfrak{S}(\textbf{n}, \textbf{a}, d), \mathfrak{G}(s))$ 
	of $(\mathfrak{X}, \mathfrak{H})$ in $x$. Let $\textbf{r} := \val(\textbf{a})$. We call the closed subset
	\[S(x) := S(\mathfrak{U}, \mathfrak{H} \cap \mathfrak{U}) \subseteq \mathfrak{U}_{\eta} 
	\setminus \mathfrak{H}_{\eta} \subseteq \mathfrak{X}_{\eta} \setminus \mathfrak{H}_{\eta}\] 
	the \emph{building block skeleton} associated to $x$. In light of Proposition~\ref{DistStrIndep} this is independent of
	the choice of the building block. 
	
	We have a homeomorphism 
	$S(x) \xrightarrow{\sim} S(\mathfrak{S}(\textbf{n}, \textbf{a}, d), \mathfrak{G}(s)) \xrightarrow{\sim} \Delta(\textbf{n}, \textbf{r}, s)$
	via $\varphi_{\eta}$ and $\trop$. Recall the canonical polyhedron $\Delta(x)$ from Definition \ref{CanPolyDef}.
	There is a homeomorphism $\Delta(\textbf{n}, \textbf{r}, s) \xrightarrow{\sim} \Delta(x)$
	via the combinatorial chart $([\mathbf{n}, \mathbf{r}, s], \alpha_{\varphi}, \gamma_{\varphi})$. 
	Altogether we obtain a homeomorphism
	\[S(x) \xrightarrow{\sim} \Delta(x) \]
	between the building block skeleton $S(x)$ and the canonical polyhedron $\Delta(x)$.
	It follows from Proposition~\ref{CoordChange} and the definition of the canonical polyhedron
	that this homeomorphism is independent of the choice of the building block. So it is justified to call it the
	\emph{canonical homeomorphism}.  
\end{definition}

\begin{definition}
	We define the \emph{extended skeleton} of $(\mathfrak{X}, \mathfrak{H})$ as the following subset of 
	$\mathfrak{X}_{\eta} \setminus \mathfrak{H}_{\eta}$:
	\[S(\mathfrak{X}, \mathfrak{H}) := \bigcup\limits_{x \in \str(\mathfrak{X}_s, \mathfrak{H}_s)} S(x).\]
	Note that it is possible that there are infinitely many elements in $\str(\mathfrak{X}_s, \mathfrak{H}_s)$, but
	in any case the collection of all $S(x)$ for $x \in \str(\mathfrak{X}_s, \mathfrak{H}_s)$ is a locally finite 
	family of closed subsets of $\mathfrak{X}_{\eta} \setminus \mathfrak{H}_{\eta}$,
	which implies that $S(\mathfrak{X}, \mathfrak{H})$ is closed.  
\end{definition}

\begin{proposition} \label{RedBlocks}
	Let $x \in \str(\mathfrak{X}_s, \mathfrak{H}_s)$. The canonical homeomorphism $S(x) \xrightarrow{\sim} \Delta(x)$
	restricts to a homeomorphism $S(x) \cap \red^{-1}_{\mathfrak{X}}(x) \xrightarrow{\sim} \Delta^{\circ}(x)$.  
\end{proposition}

\begin{proof}
	This is an immediate consequence of Proposition~\ref{StdRedExtended} and Proposition~\ref{BuildingBlockProp}~(ii).
\end{proof}

\begin{proposition} \label{SkeletonBlocksProp}
	Let $x, y \in \str(\mathfrak{X}_s, \mathfrak{H}_s)$.
	\begin{enumerate}
		\item If $x \leq y$, then $S(y) \subseteq S(x)$.
		More precisely, $S(y)$ is contained in $S(x)$ as a face, i.\,e. the following diagram,
		involving the canonical homeomorphisms and the face embedding $\iota_{y, x}$, commutes:
		\begin{center}
			\begin{tikzpicture}[scale = 2]
			\node (1) at (0, 1) {$S(y)$};
			\node (2) at (1, 1) {$\Delta(y)$};
			\node (3) at (0, 0) {$S(x)$};
			\node (4) at (1, 0) {$\Delta(x)$};	
			
			\draw[->, thick] (1)--(2) node[pos=0.5, above] {$\sim$};
			\draw[right hook->, thick] (1)--(3) ;
			\draw[->, thick] (3)--(4) node[pos=0.5, below] {$\sim$};
			\draw[right hook->, thick] (2)--(4) ;
			\end{tikzpicture}
		\end{center}
		\item $S(x)$ is equal to the union of all $S(z) \cap \red^{-1}_{\mathfrak{X}}(z)$, where
		$z \in \str(\mathfrak{X}_s, \mathfrak{H}_s)$ such that $x \leq z$,
		and this union is disjoint.  
		\item The intersection $S(x) \cap S(y)$ is given as the union of all $S(z)$, where
		$z \in \str(\mathfrak{X}_s, \mathfrak{H}_s)$ such that $x \leq z$ and $y \leq z$. 
		\item $S(\mathfrak{X}, \mathfrak{H})$ is equal to the union of all $S(z) \cap \red^{-1}_{\mathfrak{X}}(z)$, where
		$z \in \str(\mathfrak{X}_s, \mathfrak{H}_s)$, and this union is disjoint. 
	\end{enumerate}
\end{proposition}

\begin{proof}
	For (i) let us assume that $x \leq y$ and let $\mathfrak{U}$ together with 
	$\varphi: (\mathfrak{U}, \mathfrak{H} \cap \mathfrak{U}) \to (\mathfrak{S}, \mathfrak{G})$  
	be a building block of $(\mathfrak{X}, \mathfrak{H})$ in $x$. 
	As explained in~\ref{SliceBuildingBlock} we obtain a building block
	$\varphi': (\mathfrak{U}', \mathfrak{H} \cap \mathfrak{U}') 
	\to (\mathfrak{T}, \mathfrak{G}' \cap \mathfrak{T}) \hookrightarrow (\mathfrak{S}', \mathfrak{G}')$ in $y$.  
	The homeomorphism $S(\mathfrak{T}, \mathfrak{G}' \cap \mathfrak{T}) \xrightarrow{\sim} S(\mathfrak{S}', \mathfrak{G}')$ 
	induced by the open immersion $\mathfrak{T} \to \mathfrak{S}'$ is actually an identity of subsets of 
	$\mathfrak{S}'_{\eta}$ as one can easily check, also see Example \ref{BallSkeleton}.
	Consider the diagram
	\begin{center}
		\begin{tikzpicture}[scale = 2]
		\node (1) at (0, 1) {$S(\mathfrak{U}', \mathfrak{H} \cap \mathfrak{U}')$};
		\node (2) at (1.5, 1) {$S(\mathfrak{S}', \mathfrak{G}')$};
		\node (3) at (3, 1) {$S(\mathfrak{T}, \mathfrak{G}' \cap \mathfrak{T})$};
		\node (4) at (0, 0) {$S(\mathfrak{U}, \mathfrak{H} \cap \mathfrak{U})$};	
		\node (5) at (3, 0) {$S(\mathfrak{S}, \mathfrak{G})$};	
		
		\draw[->, thick] (1)--(2) node[pos=0.5, above] {$\sim$};
		\draw[double] (2)--(3); 
		\draw[->, thick] (4)--(5) node[pos=0.5, below] {$\sim$};
		\draw[right hook->, thick] (3)--(5);
		\draw[->, thick] (1)--(5) node[pos=0.5, above] {$\varphi_{\eta}$};;
		\end{tikzpicture}
	\end{center}	
	where the upper resp. lower homeomorphism is induced by $\varphi'_{\eta}$ resp. $\varphi_{\eta}$. The inclusion 
	$S(\mathfrak{T}, \mathfrak{G}' \cap \mathfrak{T}) \hookrightarrow S(\mathfrak{S}, \mathfrak{G})$
	is induced by the open immersion $\mathfrak{T} \to \mathfrak{S}$ and corresponds to an inclusion
	$\Delta(\mathbf{n}', \mathbf{r}', s') \hookrightarrow \Delta(\mathbf{n}, \mathbf{r}, s)$
	induced by the injective morphism $\iota_{C_{\varphi'}, C_{\varphi}}$ from~\ref{FaceEmbDef}. 
	It becomes clear that 
	\[S(y) = S(\mathfrak{U}', \mathfrak{H} \cap \mathfrak{U}') 
	\subseteq S(\mathfrak{U}, \mathfrak{H} \cap \mathfrak{U}) = S(x) \]
	and this inclusion corresponds to $\iota_{C_{\varphi'}, C_{\varphi}}$
	and therefore to the face embedding $\iota_{y, x}$. 
	
	The second part (ii) immediately follows from~(i) by using 
	Lemma~\ref{DisjointUnion}~(i) and Proposition~\ref{RedBlocks}.
	
	For the third part (iii) it suffices to show that for every point $v \in S(x) \cap S(y)$
	there exists an element $z \in \str(\mathfrak{X}_s, \mathfrak{H}_s)$ with $x \leq z$ and $y \leq z$ such that
	$v \in S(z)$. Let $z$ be the generic point of the stratum of $(\mathfrak{X}_s, \mathfrak{H}_s)$ which contains
	$\red_{ \mathfrak{X}}(v)$. As a consequence of~(ii) we get $x \leq z$, $y \leq z$ and 
	$v \in S(z) \cap \red^{-1}_{\mathfrak{X}}(z)$.
	
	At last (iv) is obvious from (ii), which finishes the proof. 
\end{proof}

\begin{theorem} \label{StrictHomeoThm}
	The skeleton $S(\mathfrak{X}, \mathfrak{H})$ is homeomorphic to the dual intersection complex
	$C(\mathfrak{X}, \mathfrak{H})$ via the canonical homeomorphisms on faces.
\end{theorem}

\begin{proof}
	We just glue together the canonical homeomorphisms $S(x) \xrightarrow{\sim} \Delta(x)$ for all 
	$x \in \str(\mathfrak{X}_s, \mathfrak{H}_s)$. The $S(x)$ are closed subsets of $S(\mathfrak{X}, \mathfrak{H})$,
	the $\Delta(x)$ are closed subsets of $C(\mathfrak{X}, \mathfrak{H})$
	and the $S(x) \xrightarrow{\sim} \Delta(x)$ agree on the intersections as can be seen from
	Proposition~\ref{SkeletonBlocksProp}.   
\end{proof}

\begin{lemma} \label{StrictSkeletonMap}
	Let $\psi: \mathfrak{Y} \to \mathfrak{X}$ be an \et morphism. 
	Consider $\mathfrak{G} := \psi^{-1}(\mathfrak{H})$. 
	Then the induced morphism $\psi: (\mathfrak{Y}, \mathfrak{G}) \to (\mathfrak{X}, \mathfrak{H})$
	of strictly poly-stable pairs induces a continuous map $S(\psi): S(\mathfrak{Y}, \mathfrak{G}) \to S(\mathfrak{X}, \mathfrak{H})$
	of skeletons, which restricts to a homeomorphism on the building block skeletons.
	
	Moreover this construction is functorial, i.\,e. if $\psi': \mathfrak{Z} \to \mathfrak{Y}$ is another \et morphism,
	then $S(\psi \circ \psi') = S(\psi) \circ S(\psi')$.
\end{lemma}

\begin{proof}
	Let $y \in \str(\mathfrak{Y}_s, \mathfrak{G}_s)$. We consider $x := \psi_s(y) \in \str(\mathfrak{X}_s, \mathfrak{H}_s)$. 
	Let $\mathfrak{U}$ together with 
	$\varphi: (\mathfrak{U}, \mathfrak{H} \cap \mathfrak{U}) \to (\mathfrak{S}, \mathfrak{G}')$  
	be a building block of $(\mathfrak{X}, \mathfrak{H})$ in $x$.
	Now we remove from $\psi^{-1}(\mathfrak{U})$ all closures $\overline{\{z\}}$ for all 
	$z \in \str(\mathfrak{Y}_s, \mathfrak{G}_s)$ such that $y \notin \overline{\{z\}}$
	and pass to an affine open neighborhood of $y$. We denote the resulting open neighborhood of $y$
	by $\mathfrak{V}$.
	Then $\mathfrak{V}$	together with $\varphi \circ \psi|_{\mathfrak{V}}$ is a building block
	of $(\mathfrak{Y}, \mathfrak{G})$ in $y$. We see now that $\psi_{\eta}$ restricts to a homeomorphism 
	$S(y) \xrightarrow{\sim} S(x)$. By the definition of the skeleton we obtain that $\psi_{\eta}$ induces
	a continuous map $S(\psi): S(\mathfrak{Y}, \mathfrak{G}) \to S(\mathfrak{X}, \mathfrak{H})$
	restricting to a homeomorphism on the building block skeletons. 
	The functoriality of this procedure is evident.
\end{proof}

\begin{lemma} \label{MapsAgree}
	Under the identification of skeletons and dual intersection complexes from Theorem \ref{StrictHomeoThm}
	the constructed maps in Lemma~\ref{IntersectionComplexMap} and Lemma~\ref{StrictSkeletonMap} agree. 
\end{lemma}

\begin{proof}
	This is immediately clear from the constructions and Proposition \ref{SkeletonBlocksProp}.
\end{proof}

\begin{proposition} \label{SkeletonPreimage}
	With the notations from Lemma \ref{StrictSkeletonMap} we have 
	$\psi^{-1}_{\eta}(S(\mathfrak{X}, \mathfrak{H})) = S(\mathfrak{Y}, \mathfrak{G})$.
\end{proposition}

\begin{proof}
	The inclusion “$\supseteq$” is obvious from Lemma \ref{StrictSkeletonMap}. 
	For the other inclusion let $x \in S(\mathfrak{X}, \mathfrak{H})$ and $y \in \psi^{-1}_{\eta}(x)$.
	We denote $\widetilde{x} := \red_{\mathfrak{X}}(x) \in \mathfrak{X}_s$
	and $\widetilde{y} := \red_{\mathfrak{Y}}(y) \in \mathfrak{Y}_s$. 
	By Proposition \ref{SkeletonBlocksProp} (iv) we know that $\widetilde{x} \in \str(\mathfrak{X}_s, \mathfrak{H}_s)$.
	Since $\psi_s(\widetilde{y}) = \widetilde{x}$ it follows from Proposition \ref{StrataMap} that
	$\widetilde{y} \in \str(\mathfrak{Y}_s, \mathfrak{G}_s)$.
	Now let $\mathfrak{U}$ together with 
	$\varphi: (\mathfrak{U}, \mathfrak{H} \cap \mathfrak{U}) \to (\mathfrak{S}, \mathfrak{G}')$
	be a building block of $(\mathfrak{X}, \mathfrak{H})$ in $\widetilde{x}$.
	Note that $x \in S(\widetilde{x}) = S(\mathfrak{U}, \mathfrak{H} \cap \mathfrak{U}) \subseteq \mathfrak{U}_{\eta} \setminus \mathfrak{H}_{\eta}$.
	We define $\mathfrak{V}$ to be $\psi^{-1}(\mathfrak{U})$ 
	minus the closed subsets $\overline{\{z\}}$ for all $z \in \str(\mathfrak{Y}_s, \mathfrak{G}_s)$
	with $\widetilde{y} \notin \overline{\{z\}}$. 
	It holds $y \in \mathfrak{V}_{\eta}$ and we may replace $\mathfrak{V}$ by an affine open subset of 
	$\mathfrak{V}$ whose generic fiber contains $y$.
	Then $\mathfrak{V}$ together with
	$\varphi \circ \psi: (\mathfrak{V}, \mathfrak{G} \cap \mathfrak{V}) \to (\mathfrak{S}, \mathfrak{G}')$
	is a building block of $(\mathfrak{Y}, \mathfrak{G})$ in $\widetilde{y}$.   
	We conclude from $\varphi_{\eta}(\psi_{\eta}(y)) = \varphi_{\eta}(x) \in S(\mathfrak{S}, \mathfrak{G}')$ 
	that $y \in S(\widetilde{y}) = S(\mathfrak{V}, \mathfrak{G} \cap \mathfrak{V})$, in particular
	$y \in S(\mathfrak{Y}, \mathfrak{G})$. 
\end{proof}


\subsection{Poly-stable pairs} \label{PolyStableParagraph}

Let now $(\mathfrak{X}, \mathfrak{H})$ be a poly-stable pair.
We denote $Z := \mathfrak{X}_{\eta} \setminus \mathfrak{H}_{\eta}$.
We continue to assume that $K$ is non-trivially valued.

Our goal is to define a canonical subset $S(\mathfrak{X}, \mathfrak{H})$ of $Z$, which we will call the “extended skeleton”
of the pair $(\mathfrak{X}, \mathfrak{H})$, and a proper strong deformation retraction $\Phi: Z \times [0, 1] \to Z$
onto $S(\mathfrak{X}, \mathfrak{H})$. Let us start with the construction of the map $\Phi$.

\begin{definition}
	We call a continuous map $f: Y \to X$ a \emph{quotient map}, if the induced map $Y/{\sim} \to X$ is a homeomorphism, 
	where $\sim$ is the equivalence relation on $Y$ identifying
	two elements $y, y' \in Y$ iff $f(y) = f(y')$. 
	Obviously quotient maps are surjective. Moreover one easily checks that $f: Y \to X$ is a quotient map if and only if the induced map 
	$\coker(Y \times_X Y \rightrightarrows Y) \rightarrow X$ is a homeomorphism, where the coequalizer is considered
	with respect to the canonical projections from the fibered product $Y \times_X Y$.
\end{definition}

\begin{proposition} \label{QuotMapProp}
	We state some important properties of quotient maps. Let $f: Y \to X$ be a quotient map.
	Then the following hold:
	\begin{enumerate}
		\item Let $U \subseteq X$. Then $U$ is open resp. closed in $X$ 
		if and only if $f^{-1}(U)$ is open resp. closed in $Y$.
		\item A map $g: X \to W$ is continuous if and only if $g \circ f$ is continuous.
		\item For any subset $U \subseteq X$ which is open or closed, the restriction $f: f^{-1}(U) \to U$
		is a quotient map.
	\end{enumerate}
\end{proposition}

\begin{proof}
	These are simple exercises in topology. We refer to \cite[Proposition 2.4.3 and Proposition 2.4.15]{Eng89}.
\end{proof}

\begin{lemma} \label{FactorMap}
	Let $\psi: \mathfrak{Y} \rightarrow \mathfrak{X}$ be a surjective \et morphism of admissible formal schemes. %
	Then $\psi_{\eta}: \mathfrak{Y}_{\eta} \rightarrow \mathfrak{X}_{\eta}$ is a quotient map.  
\end{lemma}

\begin{proof}
	It is shown in \cite[§§\,2-3]{Ber96} that in this case $\psi_{\eta}$ is a quasi-\et covering of $\mathfrak{X}_{\eta}$.
	Then the claim follows from \cite[Lemma 5.11]{Ber99}. 
\end{proof}

\begin{noname}
	We choose a strictly poly-stable pair $(\mathfrak{Y}, \mathfrak{G})$
	and a surjective \et morphism $\psi: (\mathfrak{Y}, \mathfrak{G}) \to (\mathfrak{X}, \mathfrak{H})$. 
	Let $x \in Z$ and $t \in [0, 1]$. 
	According to Lemma \ref{FactorMap} the map $\psi_{\eta}$ is a quotient map. Then Proposition \ref{QuotMapProp} (iii) 
	tells us that the restriction $\mathfrak{Y}_{\eta} \setminus \mathfrak{G}_{\eta} \to Z$ 
	of $\psi_{\eta}$ is a quotient map as well. 
	In particular there exists an element $y \in \mathfrak{Y}_{\eta} \setminus \mathfrak{G}_{\eta}$ with $\psi_{\eta}(y) = x$. 
	Now choose a building block $\mathfrak{U}$ of $(\mathfrak{Y}, \mathfrak{G})$ such that $y \in \mathfrak{U}_{\eta}$.
	We denote $W := \mathfrak{U}_{\eta} \setminus \mathfrak{G}_{\eta}$.
	As seen in Proposition \ref{BuildingBlockProp} (iii) there is a homotopy 
	$\Phi_{\mathfrak{U}}: W \times [0, 1] \to W$ and we define  
	$\Phi(x, t) := \psi_{\eta}(\Phi_{\mathfrak{U}}(y, t))$. 
\end{noname}

\begin{proposition} \label{PhiProp}
	The map $\Phi: Z \times [0, 1] \to Z$ is a well-defined continuous proper map
	which is independent of the choice of $(\mathfrak{Y}, \mathfrak{G})$ and $\psi$.
\end{proposition}	

\begin{proof}
	First let us check that $\Phi$ is well-defined and 
	independent of the choice of $(\mathfrak{Y}, \mathfrak{G})$ and $\psi$.
	So assume we have another strictly poly-stable pair $(\mathfrak{Y}', \mathfrak{G}')$,
	a surjective \et morphism $\psi': (\mathfrak{Y}', \mathfrak{G}') \to (\mathfrak{X}, \mathfrak{H})$
	an element $y' \in \mathfrak{Y}'_{\eta} \setminus \mathfrak{G}'_{\eta}$ with $\psi'_{\eta}(y') = x$
	and a building block $\mathfrak{U}'$ of $(\mathfrak{Y}', \mathfrak{G}')$ such that $y' \in \mathfrak{U}'_{\eta}$. 
	Consider the fiber product $\mathfrak{V} := \mathfrak{U} \times_{\mathfrak{X}} \mathfrak{U}'$. 
	Denoting by $\pi_1$ and $\pi_2$ the canonical projections on $\mathfrak{V}$ 
	we get an element $z \in \mathfrak{V}_{\eta}$ such
	that $\pi_{1, \eta}(z) = y$ and $\pi_{2, \eta}(z) = y'$. Passing to an open affine subset $\mathfrak{W}$ of $\mathfrak{V}$ whose
	generic fiber contains $z$, we can conclude from \ref{ExtendedStep5}
	that $\pi_{1, \eta}(\Phi_{\mathfrak{W}}(z, t)) = \Phi_{\mathfrak{U}}(y, t)$ and 
	$\pi_{2, \eta}(\Phi_{\mathfrak{W}}(z, t)) = \Phi_{\mathfrak{U}}(y', t)$.
	Consequently $\psi_{\eta}(\Phi_{\mathfrak{U}}(y, t)) = \psi'_{\eta}(\Phi_{\mathfrak{U'}}(y', t))$. 
	
	Next we show that $\Phi$ is continuous. Note that we may replace $\mathfrak{Y}$
	by a disjoint union of building blocks $(\mathfrak{U}_i)_{i \in I}$ for $(\mathfrak{Y}, \mathfrak{G})$
	which cover $\mathfrak{Y}$. Since $\psi_{\eta}$ is a quotient map, see Lemma \ref{FactorMap}, and 
	we already know continuity of the $\Phi_{\mathfrak{U}_i}$, we are done by Proposition \ref{QuotMapProp} (ii).    
	
	Finally we convince ourselves that $\Phi$ is proper.
	First we consider the case where $\mathfrak{X}$ is a disjoint union of $(\mathfrak{U}_i)_{i \in I}$ of building blocks. 
	Then we can choose $\mathfrak{Y} = \mathfrak{X}$ and $\psi = \id$. Let $A \subseteq Z$
	be a compact subset. Note that $Z$ is the disjoint union of the $\mathfrak{U}_{i, \eta} \setminus \mathfrak{H}_{\eta}$,
	in particular the $\mathfrak{U}_{i, \eta}$ form an open and closed cover of $Z$ 
	and $A$ intersects only finitely many of them. 
	Since we already know properness in the building block situation and the union of finitely many 
	compact subsets is again compact, we get that $\Phi^{-1}(A)$ is compact. Therefore $\Phi$ is proper. 
	
	Now let $\mathfrak{X}$ be arbitrary and $A \subseteq Z$ be compact. 
	Let $(\mathfrak{U}_i)_{i \in I}$ be a locally finite open affine cover of $\mathfrak{X}$.
	Each $\mathfrak{U}_i$ is poly-stable as well, therefore we can find surjective \'{e}tale morphisms
	$\psi_i: \mathfrak{Y}_i \rightarrow \mathfrak{U}_i$ where the $\mathfrak{Y}_i$ are disjoint unions
	of $(\mathfrak{V}_{ij})_{j \in J_i}$ as above. As \'{e}tale morphisms are open, the images of the  
	$\mathfrak{V}_{ij}$ under $\psi_i$ are open. Consequently the generic fiber $\mathfrak{U}_{i, \eta}$ is a union
	of closed analytic domains $\psi_i(\mathfrak{V}_{ij})_{\eta}$ and because $\mathfrak{U}_{i, \eta}$ is compact
	we may assume that $J_i$ is finite, using that every point in an analytic domain has a neighborhood.
	We take $\mathfrak{Y}$ to be the disjoint union of all the $\mathfrak{V}_{ij}$ and get a 
	surjective \'{e}tale morphism $\psi: \mathfrak{Y} \rightarrow \mathfrak{X}$. Note that 
	$(\mathfrak{U}_{i, \eta} \setminus \mathfrak{H}_{\eta})_{i \in I}$ is a locally finite cover of $Z$,
	in particular we can choose for every point in $Z$ an open neighborhood 
	which intersects only finitely many $\mathfrak{U}_{i, \eta} \setminus \mathfrak{H}_{\eta}$. 
	Since $A$ is compact, it can covered by finitely many of those neighborhoods. 
	In particular $A$ intersects only finitely many of the
	closed analytic domains $\mathfrak{U}_{i, \eta}  \setminus \mathfrak{H}_{\eta}$ in $Z$ 
	and the intersection is closed. Now $\psi_{\eta}^{-1}(A)$
	is given as a finite union of compact sets, namely the preimages of said intersections under the corresponding $\psi_{i, \eta}$.
	Consequently $\psi_{\eta}^{-1}(A)$ is compact. The claim now follows by applying properness in the first case.  	
\end{proof}	

\begin{definition}
	Consider the continuous map $\tau: Z \to Z$ defined by $\tau(x) := \Phi(x, 1)$ for all $x \in Z$.
	We define the \emph{extended skeleton} of $(\mathfrak{X}, \mathfrak{H})$ to be $S(\mathfrak{X}, \mathfrak{H}) := \tau(Z)$.
	Immediately from the definitions it is clear, that $S(\mathfrak{X},  \mathfrak{H}) = \psi_{\eta}(S(\mathfrak{Y}, \mathfrak{G}))$. 
\end{definition}

\begin{lemma} \label{PhiLemma}
	The map $\tau: Z \to S(\mathfrak{X}, \mathfrak{H})$ is a retraction.
	Moreover the map $\Phi: Z \times [0, 1] \to Z$ is a proper strong deformation retraction onto $S(\mathfrak{X}, \mathfrak{H})$.
	In particular $S(\mathfrak{X}, \mathfrak{H})$ is a closed subset of $Z$.	
\end{lemma}

\begin{proof}
	We use that we already know these properties in the building block situation,
	see Proposition~\ref{BuildingBlockProp}.
	Then the properties in the general situation easily follow from Proposition \ref{PhiProp} and the construction.
	In particular $\Phi(\cdot, 0) = \id_{Z}$, 
	$\Phi(\cdot, 1) = \tau$ and $\Phi(\cdot, t)|_S = \id_{S}$ for all $t \in [0, 1]$,
	if we write $S := S(\mathfrak{X}, \mathfrak{H})$. 
	
	Next we show that $S(\mathfrak{X}, \mathfrak{H})$ is a closed subset of $Z$.
	Since $Z$ is locally compact, the proper continuous map $\Phi$ 
	is closed. Now $S(\mathfrak{X}, \mathfrak{H})$ is the image of the closed subset $Z \times \{1\}$
	under $\Phi$, which finishes the proof.
\end{proof}

\begin{proposition} \label{ClosedProp} 
	It holds $\psi^{-1}_{\eta}(S(\mathfrak{X}, \mathfrak{H})) = S(\mathfrak{Y},  \mathfrak{G})$. 
\end{proposition}

\begin{proof}
	The inclusion “$\supseteq$” is clear. For the other inclusion
	let $x \in S(\mathfrak{X}, \mathfrak{H})$ and $y \in \psi^{-1}_{\eta}(x)$. Assume for contradiction
	that $y \notin S(\mathfrak{Y}, \mathfrak{G})$. 
	Choose a building block $\mathfrak{U}$ of $(\mathfrak{Y}, \mathfrak{G})$ whose generic fiber contains $y$.
	Then there exists an element $t' \in [0, 1)$ such that the function 
	$[t', 1] \to \mathfrak{U}_{\eta}$, $t \mapsto \Phi_{\mathfrak{U}}(y, t)$
	is injective. This is easily shown for standard pairs using Proposition \ref{InjProp} 
	and the $\varepsilon$-approximation procedure, and can then be generalized to building blocks
	with the arguments from §\,$3.3$. 
	Now by construction and the fact that $\Phi$ leaves elements of the skeleton fixed,
	all images of this function are mapped to $x$ under $\psi_{\eta}$.
	But this is not possible since $\psi_{\eta}$ is quasi-\etNS,
	which means its fibers are $0$-dimensional. 
	Details concerning the properties of quasi-\et morphisms can be found in \cite[§\,5]{Duc18}.
	With that we have shown the equality.
	%
\end{proof}


\begin{noname} \label{SkeletonIntComplex}
	Again we consider the fiber product $\mathfrak{Z} := \mathfrak{Y} \times_{\mathfrak{X}} \mathfrak{Y}$ via $\psi$
	and the two canonical projections $p_1$ and $p_2$ from $\mathfrak{Z}$ to $\mathfrak{Y}$.
	Let us denote $\mathfrak{F} := p^{-1}_1(\mathfrak{G}) = p^{-1}_2(\mathfrak{G})$.
	
	From Proposition \ref{ClosedProp} and Proposition \ref{QuotMapProp} (iii) we infer
	that $\psi_{\eta}$ restricts to a quotient map $S(\mathfrak{Y}, \mathfrak{G}) \to S(\mathfrak{X}, \mathfrak{H})$.
	Then Proposition \ref{SkeletonPreimage} implies 
	that $S(\mathfrak{X}, \mathfrak{H})$ is canonically homeomorphic to the coequalizer
	\[\coker(S(\mathfrak{Z}, \mathfrak{F}) \rightrightarrows S(\mathfrak{Y}, \mathfrak{G}))\]
	with respect to $S(p_1)$ and $S(p_2)$. By Lemma \ref{MapsAgree} and the definition
	of the dual intersection complex we then obtain a canonical homeomorphism between
	$S(\mathfrak{X}, \mathfrak{H})$ and $C(\mathfrak{X}, \mathfrak{H})$.
	By functoriality of the constructions, the coequalizers themselves only depend 
	on the choice of $(\mathfrak{Y}, \mathfrak{G})$ and $\psi$ up to canonical homeomorphism.  
	This means that our canonical homeomorphism 
	$S(\mathfrak{X}, \mathfrak{H}) \xrightarrow{\sim} C(\mathfrak{X}, \mathfrak{H})$
	is independent from these choices.
	
	Now we consider an element $x \in \str(\mathfrak{X}_s, \mathfrak{H}_s)$.
	Observe that $\psi_{\eta}$ also restricts to a quotient map 
	$S(\mathfrak{Y}, \mathfrak{G}) \cap \red^{-1}_{\mathfrak{Y}}(\psi^{-1}_s(x)) \to 
	S(\mathfrak{X}, \mathfrak{H}) \cap \red^{-1}_{\mathfrak{X}}(x)$.
	The preimage $\psi^{-1}_s(x)$ consists of all points $y \in \str(\mathfrak{Y}_s, \mathfrak{G}_s)$
	over $x$. We recall from Proposition \ref{RedBlocks} that 
	$S(y) \cap \red^{-1}_{\mathfrak{Y}}(y)$
	is canonically identified with $\Delta^{\circ}(y)$. 
	Proposition \ref{SkeletonBlocksProp} (iv) implies that 
	$S(y) \cap \red^{-1}_{\mathfrak{Y}}(y) = S(\mathfrak{Y}, \mathfrak{G}) \cap \red^{-1}_{\mathfrak{Y}}(y)$.
	Now it becomes clear from the construction
	of the face $\Delta^{\circ}(x)$, see Lemma \ref{FaceFacts}, that the above homeomorphism
	$S(\mathfrak{X}, \mathfrak{H}) \xrightarrow{\sim} C(\mathfrak{X}, \mathfrak{H})$
	restricts to a homeomorphism 
	$S(\mathfrak{X}, \mathfrak{H}) \cap \red^{-1}_{\mathfrak{X}}(x) \xrightarrow{\sim} \Delta^{\circ}(x)$.
\end{noname}

\begin{theorem} \label{MainThmA}
	Let $(\mathfrak{X}, \mathfrak{H})$ be a poly-stable pair
	and $Z := \mathfrak{X}_{\eta} \setminus \mathfrak{H}_{\eta}$.
	The skeleton $S(\mathfrak{X}, \mathfrak{H})$ has the following properties:
	\begin{enumerate}
		\item $S(\mathfrak{X}, \mathfrak{H})$ is a closed subset of $Z$.
		\item There is a proper strong deformation retraction $\Phi: Z \times [0, 1] \to Z$ onto $S(\mathfrak{X}, \mathfrak{H})$.
		\item $S(\mathfrak{X}, \mathfrak{H})$ is canonically homeomorphic to the
		dual intersection complex $C(\mathfrak{X}, \mathfrak{H})$.
		\item For every generic point $x \in \str(\mathfrak{X}_s, \mathfrak{H}_s)$ of a stratum this canonical homeomorphism
		restricts to a homeomorphism $S(\mathfrak{X}, \mathfrak{H}) \cap \red^{-1}_{\mathfrak{X}}(x) \xrightarrow{\sim} \Delta^{\circ}(x)$.  
	\end{enumerate} 
\end{theorem}

\begin{proof}
	This is just a collection of our results from Proposition~\ref{ClosedProp}, Lemma~\ref{PhiLemma} and 
	the construction in~\ref{SkeletonIntComplex}.
\end{proof}

\begin{remark}
	For any poly-stable scheme $\mathfrak{X}$ the extended skeleton $S(\mathfrak{X}, \emptyset)$ defined above coincides with the 
	classical skeleton $S(\mathfrak{X})$ introduced by Berkovich in \cite[§\,5]{Ber99}. 
	This is immediately clear from the construction. 
\end{remark}


\subsection{The trivially valued case} \label{TrivialSection}

In this subsection we consider the situation where $K$ is trivially valued.
Let $(\mathfrak{X}, \mathfrak{H})$ be a poly-stable pair.

We extend our results from above by passing to a suitable non-trivially valued extension of $K$.
This strategy was brought to my attention after my advisor Walter Gubler 
suggested to me the considerations in \cite[§\,1.6]{BJ18}.

\begin{noname} \label{TrivFieldExt}
	Let us fix an element $r \in (0, 1)$ and let $F := K_r$ be the 
	complete non-archimedean field as defined in \cite[§\,2.1]{Ber90}.
	One can show that it is a $K$-affinoid algebra and canonically isomorphic to the field $K(\!(t)\!)$ 
	of formal Laurent series over $K$ with variable $t$,
	equipped with the $t$-adic absolute value, i.\,e. the absolute value of a non-zero
	formal Laurent series $\sum^{\infty}_{k=n} a_k t^k$, with $n \in \Z$ such that $a_n \neq 0$,
	is given as $r^n$. Note that $F$ is non-trivially valued, is peaked over $K$ and has residue field $K$.
	
	Consider now another element $r' \in (0, 1)$. If $r$ and $r'$ are $\Q$-linearly independent in
	the $\Q$-vector space $\R_{>0}$ (with multiplicative group law), then $K_r$ and $K_{r'}$ can
	be embedded in the complete non-archimedean field $K_{r, r'} := K_r \widehat{\otimes}_K K_{r'}$,
	Otherwise there exist numbers $n, n' \in \N_{>0}$ such that $r^n = (r')^{n'}$. 
	Then put $\rho := r^{1/n'} = (r')^{1/n}$ and we can embed $K_r$ resp. $K_{r'}$
	in $K_{\rho}$ via $t \mapsto t^{n'}$ resp. $t \mapsto t^n$, using the formal Laurent series expression. 
	In any case, the extension field is non-trivially valued, peaked over $K$ and has residue field $K$ as well. 
	However beware that in the second case, the field $K_{\rho}$ is in general not peaked over $K_r$ and $K_{r'}$.    
\end{noname}

\begin{noname} \label{TrivPrep}
	We pass to the base change 
	$(\mathfrak{X}', \mathfrak{H}') := 
	(\mathfrak{X} \times_{\Spf(K^{\circ})} \Spf(F^{\circ}), \mathfrak{H} \times_{\Spf(K^{\circ})} \Spf(F^{\circ}))$ 
	of $(\mathfrak{X}, \mathfrak{H})$ with respect to $\Spf(F^{\circ}) \to \Spf(K^{\circ})$,
	which gives a poly-stable pair over $F^{\circ}$.  
	
	One important observation is, that $F$ and $K$ have the same residue field, namely $\widetilde{F} = \widetilde{K} = K$.
	Consequently the special fibers $(\mathfrak{X}'_s, \mathfrak{H}'_s)$ and $(\mathfrak{X}_s, \mathfrak{H}_s)$ agree,
	in particular $\str(\mathfrak{X}'_s, \mathfrak{H}'_s) = \str(\mathfrak{X}_s, \mathfrak{H}_s)$.
	In the case of a strictly poly-stable pair $(\mathfrak{X}, \mathfrak{H})$, the building blocks of 
	$(\mathfrak{X}', \mathfrak{H}')$ are just the base changes of the building blocks of $(\mathfrak{X}, \mathfrak{H})$.
	Now we see that the dual intersection complexes $C(\mathfrak{X}', \mathfrak{H}')$ and
	$C(\mathfrak{X}, \mathfrak{H})$ coincide. This also holds true for an arbitrary poly-stable pair $(\mathfrak{X}, \mathfrak{H})$,
	since the construction of the dual intersection complex, by taking the coequalizer, reduces 
	the problem to the strict case.
	
	We will consider the map $\pi: \mathfrak{X}'_{\eta} \to \mathfrak{X}_{\eta}$ induced by
	the natural morphism $\mathfrak{X}' \to \mathfrak{X}$ of formal $K^{\circ}$-schemes.
	Note that $\pi$ restricts to a map 
	$\mathfrak{X}'_{\eta} \setminus \mathfrak{H}'_{\eta} \to \mathfrak{X}_{\eta} \setminus \mathfrak{H}_{\eta}$
	and that $\pi$ is a proper map, see for instance \cite[p.\,9]{CT19}.
	
	Since $F$ is a peaked valuation field over $K$, we get
	a canonical right inverse $\sigma: \mathfrak{X}_{\eta} \to \mathfrak{X}_{\eta}'$ to the map $\pi$.
	According to \cite[§\,5.2]{Ber90} it is given as follows: Let $x \in \mathfrak{X}_{\eta}$.
	Then $\sigma$ maps $x$ to the point in $\pi^{-1}(x) = \mathscr{M}(\mathscr{H}(x) \widehat{\otimes} F)$,
	which corresponds to the multiplicative norm on the Banach $F$-algebra $\mathscr{H}(x) \widehat{\otimes} F$. 
	It follows from \emph{loc.cit.}, Corollary 5.2.7, that $\sigma$ is continuous.
	This right inverse map $\sigma$ also exists in the case of a non-peaked extension, but then
	it is only defined on the peaked points.
	We also want to mention that $\sigma$ is a closed map.
	In the case, where $\mathfrak{X}$ is affine, this is clear. One easily reduces the general case to the affine case,
	using that the generic fiber $\mathfrak{X}_{\eta}$ is obtained by gluing together the generic fibers of affines and these
	form a locally finite closed cover of $\mathfrak{X}_{\eta}$.
	
	We use this to conclude that $\sigma$ is actually a proper map, since its image is a closed subset of $\mathfrak{X}_{\eta}'$
	and $\sigma$ is a right inverse to the continuous map $\pi$. 
\end{noname}


\begin{definition}
	We define the \emph{extended skeleton} $S(\mathfrak{X}, \mathfrak{H})$ as the image of
	the skeleton $S(\mathfrak{X}', \mathfrak{H}')$ from the non-trivially valued case
	under the map $\pi: \mathfrak{X}'_{\eta} \to \mathfrak{X}_{\eta}$.
	Since $S(\mathfrak{X}', \mathfrak{H}') \subseteq \mathfrak{X}'_{\eta} \setminus \mathfrak{H}'_{\eta}$,
	we get that $S(\mathfrak{X}, \mathfrak{H}) \subseteq \mathfrak{X}_{\eta} \setminus \mathfrak{H}_{\eta}$. 
	One also easily verifies that this definition of the skeleton in the trivially valued
	case agrees with the skeleton from the classical construction where $\mathfrak{H} = \emptyset$.
\end{definition}

\begin{proposition} \label{TrivHomeom}
	The map $\pi: \mathfrak{X}'_{\eta} \to \mathfrak{X}_{\eta}$ restricts to a homeomorphism
	$S(\mathfrak{X}', \mathfrak{H}') \xrightarrow{\sim} S(\mathfrak{X}, \mathfrak{H})$,
	whose inverse is induced by $\sigma$.
	In particular $\sigma(S(\mathfrak{X}, \mathfrak{H})) = S(\mathfrak{X}', \mathfrak{H}')$.
	Moreover $S(\mathfrak{X}, \mathfrak{H})$ is independent of the choice of $r$. 
\end{proposition}

\begin{proof}
	We make use of the fact that for standard pairs and building blocks over non-trivially valued fields, the extended skeleton
	was constructed from the classical skeleton via the $\varepsilon$-approximation technique. 
	In the case of a standard pair $(\mathfrak{X}, \mathfrak{H}) = (\mathfrak{S}, \mathfrak{G})$,
	the claimed properties are easily verified using the explicit description of the norms from~\ref{ClassicStart}. 
	
	For a building block $\varphi: (\mathfrak{U}, \mathfrak{H} \cap \mathfrak{U}) \to (\mathfrak{S}, \mathfrak{G})$
	of a strictly poly-stable pair $(\mathfrak{X}, \mathfrak{H})$, we have the following commutative diagrams
	involving the obvious base changes to $F$:
	\begin{center}\begin{tikzpicture}[scale = 2]
		\node (1) at (0, 1) {$S(\mathfrak{U}', \mathfrak{H}' \cap \mathfrak{U}')$};
		\node (2) at (2, 1) {$S(\mathfrak{S}', \mathfrak{G}')$};
		\node (3) at (0, 0) {$S(\mathfrak{U}, \mathfrak{H} \cap \mathfrak{U})$};
		\node (4) at (2, 0) {$S(\mathfrak{S}, \mathfrak{G})$};	
		
		\draw[->, thick] (1)--(2) node[pos=0.5, below] {$\varphi'_{\eta}$} node[pos=0.5, above] {$\sim$};
		\draw[->, thick] (1)--(3) node[pos=0.5, right] {$\pi_{\mathfrak{U}}$};
		\draw[->, thick] (3)--(4) node[pos=0.5, above] {$\varphi_{\eta}$};
		\draw[->, thick] (2)--(4) node[pos=0.5, left] {$\pi_{\mathfrak{S}}$} node[pos=0.5, above, sloped] {$\sim$};
		
		\draw[->, thick] (3) .. controls (-0.5, 0.5) .. (1) node[pos=0.5, left] {$\sigma_{\mathfrak{U}}$};
		\draw[->, thick] (4) .. controls (2.5, 0.5) .. (2) node[pos=0.5, right] {$\sigma_{\mathfrak{S}}$};
		
		\begin{scope}[shift={(-3,0)}]
		
		\node (1) at (0, 1) {$\mathfrak{U}'_{\eta}$};
		\node (2) at (1, 1) {$\mathfrak{S}'_{\eta}$};
		\node (3) at (0, 0) {$\mathfrak{U}_{\eta}$};
		\node (4) at (1, 0) {$\mathfrak{S}_{\eta}$};	
		
		\draw[->, thick] (1)--(2) node[pos=0.5, above] {$\varphi'_{\eta}$};
		\draw[->, thick] (3)--(1) node[pos=0.5, left] {$\sigma_{\mathfrak{U}}$};
		\draw[->, thick] (3)--(4) node[pos=0.5, below] {$\varphi_{\eta}$};
		\draw[->, thick] (4)--(2) node[pos=0.5, right] {$\sigma_{\mathfrak{S}}$};
		
		\end{scope}
		\end{tikzpicture}\end{center}
	From $(\varphi'_{\eta})^{-1}(S(\mathfrak{S}', \mathfrak{G}')) = S(\mathfrak{U}', \mathfrak{H}' \cap \mathfrak{U}')$
	we can infer, that $\sigma_{\mathfrak{U}}$ actually maps $S(\mathfrak{U}, \mathfrak{H} \cap \mathfrak{U})$
	into $S(\mathfrak{U}', \mathfrak{H}' \cap \mathfrak{U}')$.
	This readily implies that 
	$\pi_{\mathfrak{U}}: S(\mathfrak{U}', \mathfrak{H}' \cap \mathfrak{U}') \to S(\mathfrak{U}, \mathfrak{H} \cap \mathfrak{U})$
	is a homeomorphism with inverse $\sigma_{\mathfrak{U}}$. 
	In particular $\varphi_{\eta}$ induces a homeomorphism
	$S(\mathfrak{U}, \mathfrak{H} \cap \mathfrak{U}) \xrightarrow{\sim} S(\mathfrak{S}, \mathfrak{G})$ 
	and $\varphi^{-1}_{\eta}(S(\mathfrak{S}, \mathfrak{G})) = S(\mathfrak{U}, \mathfrak{H} \cap \mathfrak{U})$,
	which also shows that the skeleton does not depend on $r$.  
	
	Now the skeletons of strictly poly-stable pairs 
	are glued together from building block skeletons. The arbitrary poly-stable case is
	then obtained from the coequalizer description of the skeleton, see~\ref{SkeletonIntComplex}.
\end{proof}

\begin{proposition} \label{TrivDeform}
	Let us denote $Z := \mathfrak{X}_{\eta} \setminus \mathfrak{H}_{\eta}$ and $Z' := \mathfrak{X}'_{\eta} \setminus \mathfrak{H}'_{\eta}$. 
	Consider the deformation retraction $\Phi': Z' \times [0, 1] \to Z'$ onto $S(\mathfrak{X}', \mathfrak{H}')$
	from the non-trivially valued case. Then the map $\Phi: Z \times [0, 1] \to Z$,
	given by sending $(z, t)$ to $\pi(\Phi'(\sigma(z), t))$, is a proper strong deformation retraction onto $S(\mathfrak{X}, \mathfrak{H})$
	not depending on the choice of $r$. 	
\end{proposition}

\begin{proof}
	It follows from $\Phi'$ being a strong deformation retraction onto $S(\mathfrak{X}', \mathfrak{H}')$ and $\pi \circ \sigma = \id$, 
	that $\Phi$ is a strong deformation retraction onto $S(\mathfrak{X}, \mathfrak{H})$.
	
	We already know properness of $\pi$, $\sigma$ and $\Phi'$,
	thus altogether we conclude that $\Phi$ is proper as well. 
	
	Now for the independence with respect to $r$: Again it is enough to consider the cases
	of standard pairs and building blocks for $(\mathfrak{X}, \mathfrak{H})$.
	Let $L$ be a complete non-archimedean field extension of $F$ peaked over $K$
	with residue field $K$ and let $\pi': \mathfrak{X}''_{\eta} \to \mathfrak{X}'_{\eta}$
	be the corresponding base change map.
	We do not assume that $L$ is peaked over $F$, but there still is a right inverse 
	$\sigma: (\mathfrak{X}'_{\eta})_p \to \mathfrak{X}''_{\eta}$ of $\pi'$,
	which is defined on the set $(\mathfrak{X}'_{\eta})_p$ of peaked points over $F$ of $\mathfrak{X}'_{\eta}$. 
	Consider the deformation retraction $\Phi'': Z'' \times [0, 1] \to Z''$ onto $S(\mathfrak{X}'', \mathfrak{H}'')$,
	where $Z'' := \mathfrak{X}''_{\eta} \setminus \mathfrak{H}''_{\eta}$.
	See \ref{GroupAction} and \ref{BuildingBlockSkeleton} for the definition 
	of the deformation retractions with the help of $\ast$-products.
	Again we make use the the $\varepsilon$-approximation technique. 
	Because the $\ast$-product is compatible with $\sigma'$,
	see \cite[Proposition 5.2.8~(i)]{Ber90}, the following diagram commutes:
	\begin{center}\begin{tikzpicture}[scale = 2]
		\node (1) at (0, 1) {$Z'' \times [0, 1]$};
		\node (2) at (2, 1) {$Z''$};
		\node (3) at (0, 0) {$(Z')_p \times [0, 1]$};
		\node (4) at (2, 0) {$Z'$};	
		
		\draw[->, thick] (1)--(2) node[pos=0.5, below] {$\Phi''$};
		\draw[->, thick] (3)--(1) node[pos=0.5, right] {$\sigma' \times \id$};
		\draw[->, thick] (2)--(4) node[pos=0.5, left] {$\pi'$};
		\draw[->, thick] (3)--(4) node[pos=0.5, above] {$\Phi'$};
		\end{tikzpicture}\end{center}
	Here $(Z')_p$ denotes the set of peaked points over $F$ in $Z'$. 
	Recall that for two different choices of $r$, we can embed the resulting two fields $F$ in
	a common extension $L$ as in~\ref{TrivFieldExt}. Since the image of every point of $Z$
	under $\sigma$ is a peaked point over $F$ in $Z'$, see \cite[Corollary 5.2.3]{Ber90}, the above diagram suffices 
	in order to see the independence of the deformation retraction $\Phi$ with respect to $r$. 
\end{proof}

\begin{theorem}
	The statements from Theorem \ref{MainThmA} stay true for the skeleton $S(\mathfrak{X}, \mathfrak{H})$
	in the trivially valued case. The canonical homeomorphism between $S(\mathfrak{X}, \mathfrak{H})$
	and $C(\mathfrak{X}, \mathfrak{H})$ does not depend on the choice of $r$.
\end{theorem}

\begin{proof}
	The analytic part is clear from Propositions \ref{TrivHomeom} and \ref{TrivDeform}.
	The combinatorial part involving the dual intersection complex follows from our considerations
	in \ref{TrivPrep}. The canoncial homeomorphism 
	$S(\mathfrak{X}, \mathfrak{H}) \xrightarrow{\sim} C(\mathfrak{X}, \mathfrak{H})$ 
	is obviously obtained via $\pi$ and the canoncial homeomorphism 
	$S(\mathfrak{X}', \mathfrak{H}') \xrightarrow{\sim} C(\mathfrak{X}', \mathfrak{H}')$
	from the non-trivially valued case.
	\begin{center}\begin{tikzpicture}[scale = 2]
		\node (1) at (0, 1) {$S(\mathfrak{X}', \mathfrak{H}')$};
		\node (2) at (2, 1) {$S(\mathfrak{X}, \mathfrak{H})$};
		\node (3) at (0, 0) {$C(\mathfrak{X}', \mathfrak{H}')$};
		\node (4) at (2, 0) {$C(\mathfrak{X}, \mathfrak{H})$};	
		
		\draw[->, thick] (1)--(2) node[pos=0.5, below] {$\pi$} node[pos=0.5, above] {$\sim$};
		\draw[->, thick] (1)--(3) node[pos=0.5, below, sloped] {$\sim$};
		\draw[double, thick] (3)--(4);
		\draw[->, thick] (2)--(4) node[pos=0.5, above, sloped] {$\sim$};
		\end{tikzpicture}\end{center}
	In order to show the independence with respect to $r$, let us consider $L$ as in the proof
	of Proposition \ref{TrivDeform} and use similar notations as there.
	First one checks that in the standard pair case 
	$(\mathfrak{X}, \mathfrak{H}) := (\mathfrak{S}, \mathfrak{G}) := (\mathfrak{S}(\textbf{n}, \textbf{a}, d), \mathfrak{G}(s))$ 
	the map $\pi'$ induces a homeomorphism between $S(\mathfrak{S}'', \mathfrak{G}'')$ 
	and $S(\mathfrak{S}', \mathfrak{G}')$, which is compatible with the tropicalization maps.
	Then for a building block $\varphi: (\mathfrak{U}, \mathfrak{H} \cap \mathfrak{H}) \to (\mathfrak{S}, \mathfrak{G})$
	we obtain the following commutative diagram:
	\begin{center}\begin{tikzpicture}[scale = 2]
		\node (1) at (0, 1) {$S(\mathfrak{U}'', \mathfrak{H}'' \cap \mathfrak{U}'')$};
		\node (2) at (2.5, 1) {$S(\mathfrak{S}'', \mathfrak{G}'')$};
		\node (3) at (0, 0) {$S(\mathfrak{U}', \mathfrak{H}' \cap \mathfrak{U}')$};
		\node (4) at (2.5, 0) {$S(\mathfrak{S}', \mathfrak{G}')$};
		\node (5) at (5, 0.5) {$\Delta(\textbf{n}, \textbf{r}, s)$};	
		
		\draw[->, thick] (1)--(2) node[pos=0.5, below] {$\varphi''$} node[pos=0.5, above] {$\sim$};
		\draw[->, thick] (1)--(3) node[pos=0.5, left] {$\pi'_{\mathfrak{U}}$};
		\draw[->, thick] (3)--(4) node[pos=0.5, above] {$\varphi'$} node[pos=0.5, below] {$\sim$};;
		\draw[->, thick] (2)--(4) node[pos=0.5, above, sloped] {$\sim$} node[pos=0.5, left] {$\pi'_{\mathfrak{S}}$};
		
		\draw[->, thick] (2)--(5) node[pos=0.5, below, sloped] {$\sim$} node[pos=0.5, above, sloped] {$\trop''$};
		\draw[->, thick] (4)--(5) node[pos=0.5, above, sloped] {$\sim$} node[pos=0.5, below, sloped] {$\trop'$};
		\end{tikzpicture}\end{center}
	One now easily sees that the canonical homeomorphisms between skeleton and dual intersection complex
	are compatible with the base change maps for every arbitrary poly-stable pair $(\mathfrak{X}, \mathfrak{H})$.
	This finishes the proof. 	
\end{proof}


\section{Closure}

So far we have constructed the extended skeleton $S(\mathfrak{X}, \mathfrak{H})$ of a poly-stable pair
$(\mathfrak{X}, \mathfrak{H})$. We have seen in Proposition \ref{DivStrPst} that $\mathfrak{H}$ is a 
poly-stable formal scheme, therefore we can associate with it the classical skeleton $S(\mathfrak{H})$.
This can be used to form the closure of our extended skeleton as follows.

\begin{definition}
	Let us denote by $\overline{S}(\mathfrak{X}, \mathfrak{H})$ the union of 
	$S(\mathfrak{X}, \mathfrak{H}) \subseteq Z := \mathfrak{X}_{\eta} \setminus \mathfrak{H}_{\eta}$ 
	and $S(\mathfrak{H}) \subseteq \mathfrak{H}_{\eta} \subseteq \mathfrak{X}_{\eta}$.
	We call $\overline{S}(\mathfrak{X}, \mathfrak{H})$ 
	the \emph{closure} of $S(\mathfrak{X}, \mathfrak{H})$.
\end{definition}

\begin{theorem} \label{MainThmB}
	$\overline{S}(\mathfrak{X}, \mathfrak{H})$ is equal to
	the topological closure of $S(\mathfrak{X}, \mathfrak{H})$ in $\mathfrak{X}_{\eta}$.
	
	Moreover let $\Phi_{(\mathfrak{X}, \mathfrak{H})}: Z \times [0, 1] \to Z$ resp. 
	$\Phi_{\mathfrak{H}}: \mathfrak{H}_{\eta} \times [0, 1] \to \mathfrak{H}_{\eta}$ be the 
	deformation retraction to the extended resp. classical skeleton. They glue together 
	to a proper strong deformation retraction $\Phi: \mathfrak{X}_{\eta} \times [0, 1] \to \mathfrak{X}_{\eta}$ 
	to the closure $\overline{S}(\mathfrak{X}, \mathfrak{H})$.  
\end{theorem}

\begin{proof}
	For now let the absolute value on $K$ be non-trivial.
	
	Let us confirm the claims first in the case of a standard pair 
	$(\mathfrak{S}, \mathfrak{G}) := (\mathfrak{S}(\textbf{n}, \textbf{a}, d), \mathfrak{G}(s))$.
	By covering the balls with tori as in Example \ref{BlockBallSkeleton}, we may assume $d = s$. 
	Recall that the extended skeleton was obtained via our $\varepsilon$-approximation procedure,
	where we exhausted the punctured ball $B^d \setminus \{0\}$ with annuli of inner radius $\varepsilon$.
	
	One easily sees from the explicit description given in \ref{ClassicStart}, that the elements
	of $S(\mathfrak{G})$ can be seen as limits of elements from $S(\mathfrak{S}, \mathfrak{G})$,
	which demonstrates the claim about the closure. We use here that the absolute value on $K$
	is non-trivial.
	
	Now we want to have a look at the deformation retraction.
	As explained in~\ref{GroupAction} and~\ref{StdDeform} the deformation retraction $\Phi_{(\mathfrak{S}, \mathfrak{G})}$
	from $\mathfrak{S}_{\eta} \setminus \mathfrak{G}_{\eta}$ to  $S(\mathfrak{S}, \mathfrak{G})$ 
	is given by $(x, t) \mapsto g_t \ast x$, where the $\ast$-multiplication comes from the described analytic group action
	of an affinoid torus associated with $\mathfrak{S}_{\varepsilon}$ and the $g_t$ are the 
	peaked points as introduced earlier.
	Recall that these group actions were compatible for different values of $\varepsilon$.
	Now observe that the action on $\mathfrak{T}(1, b_{\varepsilon})_{\eta}$ extends to an action
	on $B$, where we include the annulus in the unit disc $B$ as in~\ref{SkelConstr1} 
	and the action of $\mathscr{M}(K\{T_0, T_1\}/(T_0T_1 - 1))$ on $B$ is given by multiplying with the $T_1$ coordinate. 
	The resulting action on $\mathfrak{S}_{\eta}$ extends the one on $\mathfrak{G}_{\eta}$,
	see also \cite[Proposition 5.2.8 (ii)]{Ber90}. 
	This shows that $\Phi$ is continuous. 
	Properness of $\Phi$ is now an easy consequence of compactness of $\mathfrak{S}_{\eta}$.
	Consequently we are done with the standard pair situation.
	
	Next one can consider the case of buildings blocks 
	$\varphi: (\mathfrak{U}, \mathfrak{H} \cap \mathfrak{U}) \to (\mathfrak{S}, \mathfrak{G})$.
	From Proposition~\ref{BuildingBlockProp} we conclude that $\varphi$ induces a homeomorphism
	$\overline{S}(\mathfrak{U}, \mathfrak{H} \cap \mathfrak{U}) \xrightarrow{\sim} \overline{S}(\mathfrak{S}, \mathfrak{G})$
	and that $\varphi^{-1}_{\eta}(\overline{S}(\mathfrak{S}, \mathfrak{G})) = \overline{S}(\mathfrak{U}, \mathfrak{H} \cap \mathfrak{U})$.
	This shows that $\overline{S}(\mathfrak{U}, \mathfrak{H} \cap \mathfrak{U})$
	is closed and is equal to the closure of $S(\mathfrak{U}, \mathfrak{H} \cap \mathfrak{U})$.
	The continuity of $\Phi$ follows from the standard pair case and our constructions from~\ref{ExtendedStep5}.
	
	Finally the general case with a surjective \et morphism $\psi: (\mathfrak{Y}, \mathfrak{G}) \to (\mathfrak{X}, \mathfrak{H})$
	and a strictly poly-stable pair $(\mathfrak{Y}, \mathfrak{G})$
	is then deduced by going through the construction performed in §\,\ref{PolyStableParagraph}.
	Again we may replace $\mathfrak{Y}$ by a disjoint union of building blocks.
	It is then clear, that $\overline{S}(\mathfrak{Y}, \mathfrak{G})$ is a closed subset of $\mathfrak{Y}_{\eta}$.
	Then from Proposition \ref{ClosedProp} we learn that 
	$\psi^{-1}_{\eta}(\overline{S}(\mathfrak{X}, \mathfrak{H})) = \overline{S}(\mathfrak{Y}, \mathfrak{G})$ 
	and Proposition \ref{QuotMapProp} (i) tells us that $\overline{S}(\mathfrak{X}, \mathfrak{H})$ is
	closed in $\mathfrak{X}_{\eta}$. 
	
	Now we use that $S(\mathfrak{Y}, \mathfrak{G})$ is dense in $\overline{S}(\mathfrak{Y}, \mathfrak{G})$ 
	to show that $S(\mathfrak{X}, \mathfrak{H})$ is dense in $\overline{S}(\mathfrak{X}, \mathfrak{H})$,
	in particular the closure of $S(\mathfrak{X}, \mathfrak{H})$ in $\mathfrak{X}_{\eta}$ is 
	$\overline{S}(\mathfrak{X}, \mathfrak{H})$.	
	The continuity of $\Phi$ can be seen from the definition and the 
	continuity from the building block situation, using properties of quotient maps.
	
	In the case of a trivially valued field $K$, we consider the base change to $F$ as in Subsection~\ref{TrivialSection}
	and reduce all our problems to the non-trivially valued situation.
	We denote $\pi: \mathfrak{X}'_{\eta} \to \mathfrak{X}_{\eta}$ and $\sigma: \mathfrak{X}_{\eta} \to \mathfrak{X}'_{\eta}$ as there.
	Since $\pi$ is a surjective closed map, one also sees without difficulty that $\overline{S}(\mathfrak{X}, \mathfrak{H})$
	is equal to the closure of $S(\mathfrak{X}, \mathfrak{H})$ in $\mathfrak{X}_{\eta}$.  
	If we write $\Phi': \mathfrak{X}'_{\eta} \times [0, 1] \to \mathfrak{X}'_{\eta}$ for the proper strong deformation retraction 
	onto $\overline{S}(\mathfrak{X}', \mathfrak{H}')$, then we obtain the claimed proper strong deformation retraction
	$\Phi: \mathfrak{X}_{\eta} \times [0, 1] \to \mathfrak{X}_{\eta}$ onto $\overline{S}(\mathfrak{X}, \mathfrak{H})$
	by mapping $(x, t)$ to $\pi(\Phi'(\sigma(x), t))$. 	 
\end{proof}

\begin{remark}
	If $\mathfrak{X}$ is quasicompact, then going through the proof above one easily sees that
	$\overline{S}(\mathfrak{X}, \mathfrak{H})$ is compact.
\end{remark}

\begin{example}
	We want to illustrate the closure of the skeleton of $(\mathfrak{S}(2), \mathfrak{G}(2))$.
	Note that $\mathfrak{G}(2) = \mathfrak{T}(1, 0)$.
	The extended skeleton $S(\mathfrak{S}(2), \mathfrak{G}(2))$ resp. the classical skeleton $S(\mathfrak{T}(1, 0))$ is canonically homeomorphic
	to $\Delta(0, 0, 2) = \R^2_{\geq 0}$ resp. $\Delta(1, \infty)$ via the tropicalization map.
	It follows that $\overline{S}(\mathfrak{S}(2), \mathfrak{G}(2)) = \overline{\R}^2_{\geq 0}$. 
	\begin{center}\includegraphics[scale=1]{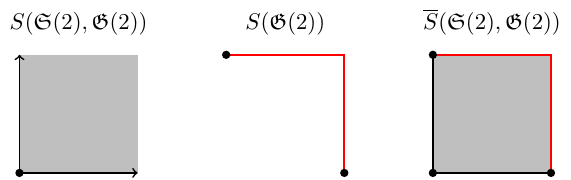}\end{center}
\end{example}

\begin{remark}
	It is possible by defining a suitable notion of the closure of the dual intersection complex $\overline{C}(\mathfrak{X}, \mathfrak{H})$
	and its faces, to establish a canonical homeomorphism 
	$\overline{S}(\mathfrak{X}, \mathfrak{H}) \to \overline{C}(\mathfrak{X}, \mathfrak{H})$
	in analogy to Theorem \ref{MainThmA} (iii). One then also recovers a correspondence between the strata of 
	$(\mathfrak{X}_s, \mathfrak{H}_s)$ and the faces of $\overline{C}(\mathfrak{X}, \mathfrak{H})$.
	Basically this is achieved by replacing $\R_{\geq 0}$ with $\overline{\R}_{\geq 0}$ in the construction
	from Section \ref{DualIntComplexChap}.      
\end{remark}



\end{document}